\documentclass[reqno,oneside,11pt]{amsart}
\usepackage[T1]{fontenc}
\usepackage[utf8]{inputenc}
\usepackage{lmodern}
\usepackage[french,english]{babel}
\usepackage{geometry} 

\usepackage{amsmath,amssymb,amsfonts,amsthm,mathtools}
\usepackage{mathrsfs}
\usepackage{graphicx}
\usepackage{caption}
\usepackage{subcaption} 
\usepackage{booktabs}
\usepackage{array}
\usepackage{paralist}
\usepackage{fancyhdr}
\usepackage{emptypage}
\usepackage{comment} 
\usepackage{xcolor}
\usepackage{epigraph}
\allowdisplaybreaks
\usepackage[noadjust]{cite}
\usepackage[
pagebackref=true,
colorlinks=true,
urlcolor=purple,
linkcolor=purple!87!black,
citecolor=green!60!black,
pdfborder={0 0 0}
]{hyperref}
\renewcommand*{\backref}[1]{}
\renewcommand*{\backrefalt}[4]{[{\tiny%
		\ifcase #1 Not cited.%
		\or Cited on page~#2.%
		\else Cited on pages #2.%
		\fi%
	}]}
\usepackage{bookmark}
\usepackage{dsfont}

\newcommand{\cay}[2]{\mathrm{Cay}(#1,#2)}

\newcommand{\Z}{\mathbb{Z}}
\newcommand{\R}{\mathbb{R}}
\renewcommand{\P}{\mathbb{P}}
\newcommand{\E}{\mathbb{E}}
\newcommand{\supp}[1]{\mathrm{supp}(#1)}
\newcommand{\re}[1]{\mathrm{Re}\left(#1\right)}
\newcommand{\pesc}[1]{p_{\mathrm{esc}}(#1)}

\newcommand{\thistheoremname}{}

\newtheorem*{genericthm*}{\thistheoremname}
\newenvironment{namedthm*}[1]
{\renewcommand{\thistheoremname}{#1}%
	\begin{genericthm*}}
	{\end{genericthm*}}
\theoremstyle{plain}
\newtheorem{thm}{Theorem}[section] 
\newtheorem{prop}[thm]{Proposition}
\newtheorem{lem}[thm]{Lemma}
\newtheorem{cor}[thm]{Corollary}
\theoremstyle{definition}
\newtheorem{rem}[thm]{Remark}
\newtheorem{claim}[thm]{Claim}
\newtheorem{defn}[thm]{Definition} 
\newtheorem{question}[thm]{Question}

\usepackage{cleveref}

\author{Eduardo Silva} 
\email[Eduardo Silva]{eduardo.silva@uni-muenster.de, edosilvamuller@gmail.com}
\address{University of Münster, Einsteinstrasse 62, Münster 48149, Germany}
\urladdr{\url{https://edoasd.github.io/eduardo_silva_math/}}

\title[Continuity of asymptotic entropy on wreath products]{Continuity of asymptotic entropy on wreath products}
\begin{document}
	
	\maketitle
	\begin{abstract} 
	We prove the continuity of asymptotic entropy as a function of the step distribution for non-degenerate probability measures with finite entropy on wreath products $ A \wr B = \bigoplus_B A \rtimes B $, where $A$ is any countable group and $B$ is a countable hyper-FC-central group that contains a finitely generated subgroup of at least cubic growth. As one step in proving the above, we show that on any countable group $G$ the probability that the $\mu$-random walk on $G$ never returns to the identity is continuous in $\mu$, for measures $\mu$ such that the semigroup generated by the support of $\mu$ contains a finitely generated subgroup of at least cubic growth. Finally, we show that among random walks on a group $G$ that admit a separable completely metrizable space $X$ as a model for their Poisson boundary, the weak continuity of the associated harmonic measures on $X$ implies the continuity of the asymptotic entropy. This result recovers the continuity of asymptotic entropy on known cases, such as Gromov hyperbolic groups and acylindrically hyperbolic groups, and extends it to new classes of groups, including linear groups and groups acting on $\mathrm{CAT}(0)$ spaces.	
\end{abstract}
\section{Introduction}\label{section: introduction}
Asymptotic entropy is a fundamental quantity in understanding the long-term behavior of random walks on countable infinite groups. For a probability measure $\mu$ on a group $G$, the \emph{Shannon entropy} of $\mu$ is defined as $H(\mu)\coloneqq -\sum_{g\in G}\mu(g)\log(\mu(g))$.

\begin{defn}\label{defn: asymptotic entropy} The \emph{asymptotic entropy} of a probability measure $\mu$ on a countable group $G$ is defined as
	\[h(\mu)\coloneqq\lim_{n\to \infty} \frac{H(\mu^{*n})}{n}.\]
\end{defn}
This notion was introduced by Avez \cite{Avez1972}, who proved that if $\mu$ is finitely supported and $h(\mu)=0$, then the random walk $(G,\mu)$ has the \emph{Liouville property}, meaning that every bounded $\mu$-harmonic function $f:G\to \R$ is constant on the cosets of $\langle \supp{\mu}\rangle$. We recall that a function $f:G\to \R$ is called \emph{$\mu$-harmonic} if $f(g)=\sum_{h\in G}f(gh)\mu(h)$, for all $g\in G$. The \emph{entropy criterion} of Derriennic \cite{Derrienic1980} and Kaimanovich-Vershik \cite[Theorem 1.1]{KaimanovcihVershik1983} states that if $H(\mu)<\infty$, then $h(\mu)>0$ if and only if $(G,\mu)$ does not have the Liouville property. Recall that $\mu$ is called non-degenerate if $G=\langle \supp{\mu}\rangle_{+}$ (the semigroup generated by $\supp{\mu}$). If $G$ is a non-amenable group, then any non-degenerate probability measure $\mu$ on $G$ does not have the Liouville property \cite[Proposition II.1]{Azencott1970} \cite{Furstenberg1973}. Hence, if in addition $H(\mu)<\infty$, then $h(\mu)>0$. Reciprocally, every amenable group admits a non-degenerate symmetric probability measure $\mu$ with the Liouville property \cite[Theorem 1.10]{Rosenblatt1981} \cite[Theorem 4.4]{KaimanovcihVershik1983}. In general, one cannot guarantee in this statement that $H(\mu)<\infty$. Indeed, there are amenable groups on which every non-degenerate probability measure $\mu$ with $H(\mu)<\infty$ does not have the Liouville property \cite[Theorem 3.1]{Erschler2004Liouville}.

The value $h(\mu)$ determines that, asymptotically, the $n$-th instant of the $\mu$-random walk belongs to a subset of $G$ of size roughly $\exp(h(\mu)n)$; see Remark \ref{rem: entropy as long scale Hausdorff dimension} for a precise statement. For random walks with a finite first moment on a finitely generated group, another related quantity is the asymptotic drift $\ell(\mu)\coloneqq \lim_{n\to \infty} \E[|w_n|]/n$ of the random walk $\{w_n\}_{n\ge 0}$ with respect to a word metric $|\cdot|$. In many classes of groups, the ratio $h(\mu)/\ell(\mu)$ coincides with the Hausdorff dimension of the hitting measure of the $\mu$-random walk on a geometric boundary of the group $G$. This is the case for free groups \cite[Theorem 2.1.3]{Kaimanovich1998} and hyperbolic groups \cite[Theorem 3.1]{LePrince2008}; see  \cite{Tanaka2019,DussauleYang2023} for further families of groups where this equality holds. We also refer the reader to \cite{Vershik2000} for a more extensive explanation on the connection between the asymptotic entropy of random walks and other dynamical and algebraic properties of the group.

The question of whether the function $\mu\mapsto h(\mu)$ on a given group $G$ is continuous was first explicitly asked in \cite[Section 6, Item 5]{Erschler2004Liouville}, and later the first positive results on continuity of asymptotic entropy were given in \cite{Erschler2011}. One motivation for this problem is that there are multiple situations exhibiting discontinuity in asymptotic entropy within some classes of (amenable) groups. The first examples of such discontinuity were given by Kaimanovich \cite[Theorems 3.1 and 4.1]{Kaimanovich1983nontrivial}, who showed that certain locally finite groups -- such as the group of finitely supported permutations of a countable set and the wreath product $\mathbb{Z}/2\mathbb{Z} \wr (\bigoplus_{\mathbb{N}} \mathbb{Z}/2\mathbb{Z})$ --  admit infinitely supported probability measures with finite entropy, and which do not have the Liouville property; see also \cite[Section 6, Remark 3]{KaimanovcihVershik1983}. By the entropy criterion, such measures have positive asymptotic entropy, which cannot be approximated by the asymptotic entropy of finitely supported measures. The above examples involve infinitely supported measures, but discontinuity can also occur for probability measures supported on a fixed finite symmetric subset of a group. Indeed, such an example is given in \cite[Lemma 4]{Erschler2011} on the group $\Z/2\Z\wr D_{\infty}$, where $D_{\infty}$ is the infinite dihedral group. It is claimed by Vershik in the paragraph following Lemma 2 in \cite[page 682]{Vershik2000} that the estimates from \cite{Vershik1999} can be used to show that asymptotic entropy is continuous in the set of non-degenerate symmetric probability measures supported on a fixed finite generating set of a group. Note that Erschler's example does not contradict Vershik's claim, since it consists of a sequence of non-symmetric probability measures that converge to a degenerate probability measure. However, we have not been able to reconstruct Vershik's proof, and we believe that the author might have meant to say ``upper-semicontinuous'' instead of ``continuous'', which is enough to prove \cite[Lemma 2]{Vershik2000}. We remark that the results we prove in the current paper guarantee the continuity of asymptotic entropy for possibly non-symmetric and infinitely supported probability measures.

In the literature there are several results on regularity properties of asymptotic entropy that go beyond continuity, such as Lipschitz continuity, differentiability and analyticity. Families of groups studied so far include free groups \cite{Ledrappier2012}, hyperbolic groups \cite{Ledrappier2013,Gouezel2017,GilchLedrappier2013} and acylindrically hyperbolic groups \cite[Theorem 13.2]{MathieuSisto2020}, for probability measures with suitable moment conditions and that in particular include measures supported on a fixed finite symmetric generating subset of the group. In the current paper we focus exclusively on continuity results, and the study of further regularity properties of asymptotic entropy on wreath products will be developed elsewhere.

\subsection{Continuity of entropy on wreath products}
Our first main result shows that asymptotic entropy is continuous on wreath products $A\wr B\coloneqq \bigoplus_B A\rtimes B$ where the base group $B$ is a hyper-FC-central group with sufficiently fast growth. Recall that a group $G$ is called \emph{hyper-FC-central} if $G$ has no non-trivial quotient on which any non-trivial element has an infinite conjugacy class. Recall also that a group $ G $ has \emph{at least cubic growth} if there is a finite symmetric subset $ S\subseteq G$ such that $\liminf_{n\to\infty} |S^n|/n^3>0$. Any finitely generated subgroup of a countable group that does not have at least cubic growth is has a finite index subgroup that is either trivial, $\Z$ or $\Z^2$.
\begin{thm}\label{thm: main corollary continuity of asymptotic entropy over Zd}
	Let $A$ be a countable group and let $B$ be a countable hyper-FC-central group of at least cubic growth. Let $\mu$ be a non-degenerate probability measure $\mu$ on $A\wr B$ with $H(\mu)<\infty$. Consider a sequence $\{\mu_k\}_{k\ge 1}$ of probability measures on $A\wr B$ with $H(\mu_k)<\infty$ for all $k\ge 1$, such that $\lim_{k\to \infty }\mu_k(g)=\mu(g) $ for every $g \in A\wr B$ and $\lim_{k\to \infty} H(\mu_k)=H(\mu)$. Then, we have $\lim_{k\to \infty}h(\mu_k)=h(\mu)$.
\end{thm}

For a finitely generated group, being hyper-FC-central is equivalent to having a nilpotent subgroup of finite index \cite[Theorem 2]{DuguidMcLain1956} \cite[Theorem 2]{McLain1956} (see also \cite{FrischPooya2018} for a self-contained proof of this result). Therefore, hyper-FC-central groups of at least cubic growth contain a copy of either $\Z^3$ or the Heisenberg group $H_3(\Z)$. Theorem \ref{thm: main corollary continuity of asymptotic entropy over Zd} applies in particular when $B = \Z^d$, for $d \geq 3$. This is the first known result that establishes the continuity of asymptotic entropy for amenable groups that are not hyper-FC-central, and which works for probability measures with an infinite support. Additionally, it is the first result to address arbitrary sequences of finitely supported probability measures on an amenable group. The only prior result in this direction is \cite[Theorem 5]{Erschler2011}, which shows the continuity of asymptotic entropy for $A \wr \Z$ with $A$ finite, under the assumption that the supports of all probability measures $\mu$, $\mu_k$, $k \geq 1$, are contained in a fixed finite subset of $A \wr \Z$. Note that Erschler's example of discontinuity consider wreath products over $D_{\infty}$, which has linear growth. This shows that the hypothesis of at least cubic growth of $B$ cannot be removed. The hypotheses of the convergence of the Shannon entropies of the measures is also natural, since for infinitely supported probability measures it is possible to have the pointwise convergence $\lim_{k\to \infty}\mu_k(g)=\mu(g)$ for each $g\in G$ and still not have the convergence of $\{H(\mu_k)\}_{k\ge 1}$ to $H(\mu)$. Such examples would also provide probability measures for which asymptotic entropy is not continuous. Note that if one further supposes that all the measures $\{\mu_k\}_{k\ge 1}$ are supported on a fixed finite subset of $G$, then the convergence of Shannon entropies is implied by the pointwise convergence of the probability measures. This is the setting considered by Erschler in \cite{Erschler2011}, called ``strong convergence'' there. The non-degeneracy assumption on the probability measure $\mu$ in Theorem \ref{thm: main corollary continuity of asymptotic entropy over Zd} can be replaced with the condition that the semigroup generated by $\supp{\mu}$ is a subgroup of $A\wr B$ whose projection to the base group $B$ has at least cubic growth. We stated Theorem \ref{thm: main corollary continuity of asymptotic entropy over Zd} with the non-degeneracy assumption for clarity in the exposition of the result.

 A feature of Theorem \ref{thm: main corollary continuity of asymptotic entropy over Zd} is that it holds for all non-degenerate probability measures with finite entropy. For all such measures on $A \wr \Z^d$, the Poisson boundary is non-trivial \cite[Theorem 3.1]{Erschler2004Liouville}. This makes Theorem \ref{thm: main corollary continuity of asymptotic entropy over Zd} the first one to establish the continuity of asymptotic entropy in a setting where the Poisson boundary remains unidentified. Theorem \ref{thm: main corollary continuity of asymptotic entropy over Zd} is a consequence of a more general criterion for the continuity of asymptotic entropy on wreath products, which is stated in Section \ref{section: proof of the main theorem} (see Theorem \ref{thm: continuity asymptotic entropy wreath prods}). In Subsection \ref{subsection: sketch of proof continuity of entropy wreath products} we sketch the main ideas behind the proof of Theorem \ref{thm: main corollary continuity of asymptotic entropy over Zd}.

The entropy of a random walk on a wreath product $A \wr B$ at time $n$ is closely related to the number of distinct positions visited by the projection to the base group $B$ between times $0$ and $n$. This, in turn, is linked to the return probability of the induced random walk on the base group (see Subsection \ref{subsection: escape prob and the asymptotic range}). To state our next result, we introduce the following notation for the probability that a random walk on a group never returns to the identity element.
\begin{defn}\label{def: escape proba} For a given probability measure $\mu$ on a countable group $G$, we define its associated \emph{escape probability} by $\pesc{\mu}\coloneqq \P(w_n\neq e_G \text{ for every }n\ge 1)$.
\end{defn}

Recall that the $\mu$-random walk on $G$ is called transient if $\pesc{\mu}>0$, and recurrent otherwise.  

The following theorem on the continuity of the escape probability plays a role in the proof of Theorem \ref{thm: main corollary continuity of asymptotic entropy over Zd}. Recall that $\langle \supp{\mu}\rangle_{+}$ denotes the semigroup generated by the support of $\mu$.

\begin{thm}\label{thm: continuity of range} Let $\mu$ be a probability measure on a countable group $G$. Suppose that $\langle \supp{\mu}\rangle_{+}$ contains a finitely generated subgroup of at least cubic growth. Let $\{\mu_k\}_{k\ge 1}$ be a sequence of probability measures on $G$ such that $\lim_{k\to \infty}\mu_{k}(g)=\mu(g)$ for every $g\in G$. Then, we have $\lim_{k\to \infty}\pesc{\mu_k}=\pesc{\mu}$.
\end{thm}

This result generalizes \cite[Proposition 2]{Erschler2011}, which does not have hypotheses on the growth of the subgroup generated by $\supp{\mu}$, but supposes that there is a fixed finite subset $F\subseteq G$ such that $\supp{\mu_k}\subseteq F$ for all $k\ge 1$, and that $\mu$ is symmetric. The proof of Theorem \ref{thm: continuity of range} is presented in Section \ref{section: continuity of return probability}, and it is based on a uniform version of a comparison lemma due to Coulhon and Saloff--Coste \cite[Proposition IV.4]{Coulhon1996} \cite[Section II]{CoulhonSaloffCoste1990} (see also Proposition \ref{prop: universal comparison constants}).

It is a result of Varopoulos \cite{Varopoulos1986} that the only finitely generated groups admitting a non-degenerate recurrent random walk are those that have at most quadratic growth (see also \cite[Theorem 3.24]{Woess2000}). That is, groups $G$ that have a finite symmetric generating set $ S$ for which there exists $ C > 0$ such that  $|(S \cup \{e_G\})^n| \le Cn^2$ for all $ n \geq 1 $. By Gromov’s Theorem on groups of polynomial growth \cite{Gromov1981} and the Bass–Guivarc’h formula \cite{Bass1972,Guivarch1973}, any such group contains a subgroup of finite index that is either trivial, $\Z$, or $\Z^2$. In groups with at most quadratic growth, one can find sequences of probability measures that define recurrent random walks, and which converge to a probability measure that defines a transient random walk \cite[Lemma 3]{Erschler2011}. Such sequences give examples of discontinuity of the function $\mu\mapsto\pesc{\mu}$. Hence, Theorem \ref{thm: continuity of range} guarantees that one can witness discontinuity of $\pesc{\cdot}$ only on groups that admit recurrent random walks.

\subsection{Connection between the continuity of asymptotic entropy and the weak continuity of harmonic measures}
The \emph{Poisson boundary} $(\partial_{\mu} G,\nu)$ of the $\mu$-random walk on $G$ is a probability space endowed with a measurable $G$-action, such that there is a $G$-equivariant isomorphism between the space of bounded $\mu$-harmonic functions on $G$ and the space of bounded measurable functions on $(\partial_{\mu} G,\nu)$ \cite[Theorem 3.1]{Furstenberg1971}. In particular, having the Liouville property is equivalent to the triviality of the Poisson boundary. We refer to Subsection \ref{subsection: entropy and Poisson boundary} for the definition of the Poisson boundary, and to \cite{KaimanovcihVershik1983,Kaimanovich2000} for a more detailed exposition.

Even though the Poisson boundary of a random walk is a measure-theoretic object, there are several classes of groups for which this space is realized as a separable and completely metrizable space $X$ endowed with a stationary probability measure (called the \emph{harmonic measure}). In the following theorem, we prove that the weak continuity of harmonic measures on $X$ implies the continuity of asymptotic entropy. 
\begin{thm}\label{thm: convergence of harmonic measures} Let $G$ be a countable group, and let $\mu$, $\{\mu_k\}_{k\ge 1}$, be non-degenerate probability measures on $G$ with finite entropy. Assume that 
	\begin{enumerate}
		\item $\lim_{k\to \infty}\mu_k(g)=\mu(g)$ for each $g\in G$, and
		\item  $\lim_{k\to \infty}H(\mu_k)=H(\mu)$.
	\end{enumerate} Suppose that there is a Polish space $X$ on which $G$ acts by continuous transformations, and Borel probability measures $\nu$, $\{\nu_k\}_{k\ge 1}$ on $X$ such that $(X,\nu)$ (resp.\ $(X,\nu_k)$) is the Poisson boundary of $(G,\mu)$ (resp.\ $(G,\mu_k)$ for each $k\ge 1$).
	
	If the sequence $\{\nu_k\}_{k\ge 1}$ converges weakly to $\nu$, then we have $\lim_{k\to \infty} h(\mu_k)=h(\mu)$. This holds in particular whenever, in addition to the above assumptions, $X$ is compact and admits a unique $\mu$-stationary probability measure.
\end{thm}
The proof of this theorem is presented in Section \ref{subsection: continuity of harmonic measures}, and it is based on the combination of three results: the upper semicontinuity of asymptotic entropy \cite[Proposition 4]{AmirAngelVirag2013} (see also Proposition \ref{prop: upper semicontinuous entropy}), the expression for the asymptotic entropy as the average of Kullback-Leibler distances of harmonic measures on the Poisson boundary \cite[Theorem 3.1]{KaimanovcihVershik1983} (see also Theorem \ref{thm: entropy as furstenberg formula}), and the joint weak lower semicontinuity of the Kullback-Leibler distances of probability measures on a Polish space \cite[Theorem 1]{Posner1975} (see also Proposition \ref{prop: KL is lower semicontinuous}).

There are several families of groups to which Theorem \ref{thm: convergence of harmonic measures} applies. In particular, we obtain the following.
\begin{cor}\label{cor: applications}
	Let $G$ be a countable group, and let $\mu$, $\{\mu_k\}_{k\ge 1}$ be non-degenerate probability measures on $G$. Assume that $\lim_{k\to \infty}\mu_k(g)=\mu(g)$ for each $g\in G$, and that $\lim_{k\to \infty}H(\mu_k)=H(\mu)$. Consider any of the following situations.
	\begin{enumerate}
		\item $G$ is acylindrically hyperbolic and the probability measures $\mu$, $\{\mu_k\}_{k\ge 1}$ have finite entropy.
		\item $G$ is a Zariski dense discrete subgroup of $\mathrm{SL}_d(\R)$ and $\mu$, $\{\mu_k\}_{k\ge 1}$ have finite entropy and a finite logarithmic moment.
		\item $G$ is a subgroup of isometries of a proper $\mathrm{CAT}(0)$ space $X$ such that the action of $G$ on $X$ is proper and cocompact, $G$ has a rank one element, and the measures $\mu$, $\{\mu_k\}_{k\ge 1}$, have a finite first moment.
		\item  $G$ is a discrete subgroup of isometries of a finite dimensional $\mathrm{CAT}(0)$ cube complex $X$ such that the action of $G$ on $X$ is proper, and the measures $\mu$, $\{\mu_k\}_{k\ge 1}$ have finite entropy and a finite logarithmic moment.
	\end{enumerate}
	Then, $\lim_{k\to \infty}h(\mu_k)=h(\mu)$.
\end{cor}
We refer to Subsection \ref{subsection: applications} for the definitions of the terms appearing in the statement. The case of acylindrically hyperbolic groups is not new: it has already been proved in \cite[Theorem F]{Choi2024}, and also independently by Anna Erschler and Joshua Frisch in unpublished work. In the particular case of hyperbolic groups, the above was proved in \cite[Theorem 2]{ErschlerKaimanovich2013} and \cite[Theorem 2.9]{GouezelSebastienMatheusMacourant2018}  with additional conditions on the moments of the measures. Our approach is different from the ones used in the previous references. The continuity of asymptotic entropy for the remaining cases of Corollary \ref{cor: applications} is new, even if one only considers step distributions $\mu$, $\{\mu_k\}_{k\ge 1}$ that are supported on a fixed finite generating subset of the group. Some concrete new examples of groups covered by our result are $\mathrm{SL}_d(\Z)$, $d\ge 3$, and discrete subgroups of isometries of an irreducible Cartan-Hadamard manifold $M$, acting on $M$ properly, cocompactly and with a rank one element. A more extensive and detailed recollection of the classes of groups to which Theorem \ref{thm: convergence of harmonic measures} applies is given in Subsection \ref{subsection: applications}, where each of the items in Corollary \ref{cor: applications} is proved. In particular, we show that a special case of Theorem~\ref{thm: main corollary continuity of asymptotic entropy over Zd}, under additional assumptions on the measures $\mu$ and $\mu_k$, $k \ge 1$, can be deduced from Theorem~\ref{thm: convergence of harmonic measures}. However, this argument does not recover the result in full generality; see Corollary~\ref{cor: wr} and the paragraph following it.

As a final remark, we emphasize that the results presented in this paper establish the continuity of asymptotic entropy in all cases known so far. There are many classes of groups not covered by our results, and for which it is not known whether asymptotic entropy is continuous. In particular, we ask the following.

\begin{question}
	For $m\ge 2$, denote by $F_m$ the free group of rank $m$, and let $n\ge 2$ be sufficiently large so that the free Burnside group $B(m,n)\coloneqq F_m/\langle g^n, g\in F_m\rangle$ is infinite. Consider non-degenerate probability measures $\mu$, $\{\mu_k\}_{k\ge 1}$ on $B(m,n)$ such that  $\lim_{k\to \infty}\mu_k(g)=\mu(g)$ for each $g\in B(m,n)$, and  $\lim_{k\to \infty}H(\mu_k)=H(\mu)$. Is it true that $\lim_{k\to \infty} h(\mu_k)=h(\mu)$?
\end{question}
The answer to this question would be interesting even in the particular case in which the supports of $\mu$ and $\{\mu_k\}_{k\ge 1}$ are contained in a fixed finite symmetric generating set of $B(m,n)$, and where $n$ is chosen to be sufficiently large, for example, to guarantee that $B(m,n)$ is non-amenable.

\subsection{Sketch of the proof of Theorem \ref{thm: main corollary continuity of asymptotic entropy over Zd}}\label{subsection: sketch of proof continuity of entropy wreath products}
We prove Theorem \ref{thm: main corollary continuity of asymptotic entropy over Zd} as a consequence of Theorem \ref{thm: continuity asymptotic entropy wreath prods}, which shows the continuity of asymptotic entropy on wreath products in a more general setting and with the minimal hypotheses that allow our proof to work. The proof builds on ideas presented in \cite{FrischSilva2024}, by estimating the entropy of a random walk on a wreath product via the uncertainty of the values of the lamp configuration at a given instant. The basic idea of the argument is the following.

Let $\mu$ be a non-degenerate probability measure on $A\wr B$, where $B$ is a countable hyper-FC-central group that contains a finitely generated subgroup of at least cubic growth. Suppose that $H(\mu)<\infty$. Consider a sequence of probability measures $\{\mu_k\}_{k\ge 1}$ on $A\wr B$ with finite entropy such that $\lim_{k\to \infty} H(\mu_k)=H(\mu)$ and $\lim_{k\to \infty} \mu_k(g)=\mu(g)$ for every $g\in A\wr B$. Recall that we denote by $\pi:A\wr B\to B$ the canonical projection to the base group $B$.
\begin{enumerate}
	\item \label{sketch: step 1} Thanks to the subadditivity of the sequence $\{H(\mu^{*n})\}_{n\ge 1}$, we get that the asymptotic entropy is an upper-semicontinuous function \cite[Proposition 4]{AmirAngelVirag2013} (see also Proposition \ref{prop: upper semicontinuous entropy}). From this, we have $\limsup_{k\to \infty} h(\mu_k)\le h(\mu)$. Hence, the continuity of asymptotic entropy will follow from proving that $\liminf_{k\to \infty} h(\mu_k)\ge h(\mu).$
\end{enumerate}	
Denote by $\{w_n\}_{n\ge 0}$ a trajectory of the random walk on $A\wr B$ and, for each $k\ge 1$ and $n\ge 0$, denote by $H_{\mu_k}(w_n)=H(\mu_k^{*n})$ the entropy of the $n$-th step of the $\mu_k$-random walk on $A\wr B$.
\begin{enumerate}\setcounter{enumi}{1}
	\item \label{sketch: step 2} The main technical step of the proof is showing that for each $\varepsilon>0$, there is $C>0$ such that for every sufficiently large values of $k,N$ and $n$, we have \begin{equation}\label{eq:maingoal} \frac{1}{n}H_{\mu_k}(w_n)\ge \left(\frac{1}{N}-\frac{1}{n}\right)H_{\mu_k}(w_N)-\varepsilon -\frac{C}{N}-\frac{C}{n};\end{equation}
	see Proposition \ref{prop: final entropy estimates for continuity of entropy on wreath products}. The inequality $\liminf_{k\to \infty}h(\mu_k)\ge h(\mu)$ will therefore follow from taking the limits as $n\to \infty$, $k\to \infty$, $N\to \infty$ and $\varepsilon\to 0$ in Inequality \eqref{eq:maingoal}.
	
	\item \label{sketch: step 3} The proof of Inequality \eqref{eq:maingoal} is given in Section \ref{section: proof of the main theorem}, by using entropy estimates from Sections \ref{section: entropy estimates} and \ref{section: entropy lamps inside}.
\end{enumerate} 
The intuition behind the entropy estimates that we will use is the following.

\begin{enumerate}\setcounter{enumi}{3}
	\item We justify that for any sufficiently large $n>N$, the value of $w_n$ allows us to recover the values of the \textit{$N$-coarse trajectory} $\mathcal{P}_{n}^{N}(A\wr B)\coloneqq (w_{N}, w_{2N},\ldots, w_{\lfloor n/N\rfloor N})$ on $A\wr B$ (see Definition \ref{def: coarse trajectory}), up to adding a controlled amount of entropy that is of order $n(\varepsilon+ 1/N )$. Since the entropy of $\mathcal{P}_{n}^{N}(A\wr B)$ equals $\left\lfloor \frac{n}{N}\right\rfloor H(w_{N})$ (Lemma \ref{lem: slices of independent increments have additive entropy}), this will show that Inequality \eqref{eq:maingoal} holds.
	
	\item In order to recover the value of $\mathcal{P}_{n}^{N}(A\wr B)$ from $w_n$, we need to determine both the projection to the base group $B$ and the lamp configuration every $N$ steps.
	
	\item The hypothesis that $B$ is hyper-FC-central guarantees that $h(\pi_{*}\mu)=0$. This implies that the values in the base group visited by the random walk every $N$ steps up to time $n$ carry a small amount of entropy (see Lemma \ref{lem: coarse trajectory has small entropy}).
	\item For the values of the lamp configuration every $N$ steps, we divide the base group $B$ into two regions: the positions that are close to the trajectory of the projection of the random walk on $B$ (which corresponds to the set $\mathcal{N}_n(t_0,R)$ from Definition \ref{def:neighborhood of coarse traj}), and the ones that are far away from it (which corresponds to $B\backslash \mathcal{N}_n(t_0,R)$).
	\item By controlling the parameters that quantify what we mean by being ``far'' from the trajectory, we can show that the lamp configuration at positions in $B\backslash \mathcal{N}_n(t_0,R)$, sampled every $N$ steps, contains a small amount of entropy. A precise statement of this can be found in Lemmas \ref{lem: bad increments have small entropy} and \ref{lem: entropy lamps outside}.
	\item Finally, the values of the lamp configuration at positions inside $\mathcal{N}_n(t_0,R)$ (i.e., positions near the trajectory in the base group) sampled every $N$ steps can be deduced from $w_n$ with the cost of adding an amount of entropy of order $n(\varepsilon+1/N)$ (see Proposition \ref{prop: entropy lamps inside}).
	
	\item The idea behind the above is that the transience of the random walk on $B$ implies that we can find a ``waiting time'' $n_0\ge 1$ with the following property: if $\pi(w_j)=b\in B$, then it is very unlikely that $\pi(w_{j+k})=b$ for some $k>n_0$, uniformly on $b$ and $j$. For the value $n_0$ to be uniform also among the sequence $\{\mu_k\}_{k\ge 1}$, it is important that we have the convergence $\lim_{k\to \infty}\pesc{\mu_k}=\pesc{\mu}$. The latter is used, in particular, to prove Proposition \ref{prop: uniform decay for return to a finite subset}, which is applied multiple times throughout Section \ref{section: entropy lamps inside}. With this, we can show that lamp modifications that occur on positions that were visited more than $n_0$ time units in the past can be deduced from the current lamp configuration at time $n$, at the cost of adding a small amount of entropy. This is detailed in Section \ref{section: entropy lamps inside}.
\end{enumerate}

\subsection{Organization}
In Section \ref{section: preliminaries}, we recall preliminary definitions related to asymptotic entropy and the return probability of random walks on groups. Section \ref{section: convergence of probability measures} establishes basic well-known results about the convergence of probability measures on countable groups, including Proposition \ref{prop: uniform decay for return to a finite subset}, which is frequently used throughout the paper. Next, in Section \ref{section: continuity of return probability}, we present Proposition \ref{prop: universal comparison constants}, a version of Coulhon and Saloff-Coste's comparison lemma, and use it to prove Theorem \ref{thm: continuity of range} on the continuity of the escape probability. Sections \ref{section: entropy estimates} and \ref{section: entropy lamps inside} consist of entropy estimates for wreath products, which are then applied in Section \ref{section: proof of the main theorem} to prove Theorem \ref{thm: continuity asymptotic entropy wreath prods}, our most general result on the continuity of asymptotic entropy for wreath products, and to deduce Theorem \ref{thm: main corollary continuity of asymptotic entropy over Zd}. Finally, in Section \ref{section: convergence of harmonic measures}, we prove Theorem \ref{thm: convergence of harmonic measures} and Corollary \ref{cor: applications}, addressing the continuity of asymptotic entropy as a consequence of the weak continuity of harmonic measures on the Poisson boundary.
\subsection{Acknowledgements}
I would like to thank Anna Erschler for suggesting this project, for indicating the reference \cite{Coulhon1996} to me, and for many helpful discussions as well as comments on the first draft of this paper. I am grateful to Inhyeok Choi for indicating the reference \cite{AmirAngelVirag2013}  for the proof of Lemma \ref{lem: convolutions entropy convergence}, as well as for helpful comments and references for the applications of Theorem \ref{thm: convergence of harmonic measures} to acylindrically hyperbolic groups. I also thank Petr Kosenko for many thought-provoking discussions around the continuity of harmonic measures on the Poisson boundary. I am very grateful to Anna Cascioli for her careful reading of the first version of this manuscript. I would like to thank the anonymous referee for their careful reading of this paper, as well as for suggesting the addition of Proposition \ref{prop: wr app} and Corollary \ref{cor: wr}. The author is funded by the Deutsche Forschungsgemeinschaft (DFG, German Research Foundation) under Germany's Excellence Strategy EXC 2044 –390685587, Mathematics Münster: Dynamics–Geometry–Structure.
\section{Preliminaries}\label{section: preliminaries}
\subsection{Random walks on groups, entropy and the Poisson boundary}\label{subsection: entropy and Poisson boundary}
Let $\mu$ be a probability measure on a countable group $G$. The \emph{$\mu$-random walk} on $G$ is the Markov chain $\{w_n\}_{n\ge 0}$ defined by $w_0=e_G$, and for $n\ge 1$,
$$
w_n=g_1g_2\cdots g_n,
$$
where $\{g_i\}_{i\ge 1}$ is a sequence of independent, identically distributed, random variables on $G$ with law $\mu$. The law $\P_{\mu}$ of the process $\{w_n\}_{n\ge 0}$ is defined as the push-forward of the Bernoulli measure $\mu^{\Z_{+}}$ through the map
\begin{equation*}
	\begin{aligned}
		G^{\mathbb{Z}_{\ge 1}}&\to G^{\Z_{+}}\\
		(g_1,g_2,g_3,\ldots)&\mapsto (w_0,w_1,w_2,w_3,\ldots)\coloneqq (e_G,g_1,g_1g_2,g_1g_2g_3,\ldots).
	\end{aligned}
\end{equation*}
The space $(G^{\Z_{+}},\P_{\mu})$ is called the \emph{space of sample paths} or the \emph{space of trajectories} of the $\mu$-random walk. In most of the paper we will be working at the same time with a probability measure $\mu$ as well as with a sequence of probability measures $\{\mu_k\}_{k\ge 1}$. In this context we will need to distinguish between $\P_{\mu}$ and $\P_{\mu_{k}}$, $k\ge 1$.

\begin{defn}\label{defn: Poisson boundary original def} Let $G$ be a countable group, and let $\mu$ be a probability measure on $G$. Two sample paths $\mathbf{w}=(w_1,w_2,\ldots)$, $\mathbf{w^\prime}=(w^\prime_1,w^\prime_2,\ldots)\in G^{\Z_{+}}$ are said to be equivalent if there exist $p,N\ge 0$ such that $w_n=w^\prime_{n+p}$ for all $n>N$. Consider the measurable hull associated with this equivalence relation. That is, the $\sigma$-algebra formed by all measurable subsets of the space of trajectories $(G^{\Z_{+}},\P)$ which are unions of the equivalence classes of $\sim$ up to $\P$-null sets. The associated quotient space is called the \emph{Poisson boundary} of the random walk $(G,\mu)$.
\end{defn}

The Poisson boundary can also be defined as the space of ergodic components of the shift map $T:G^{\Z_{+}}\to G^{\Z_{+}}$ on the space of trajectories, where $T(w_1,w_2,w_3,\ldots)\coloneqq (w_2,w_3,\ldots)$ for $(w_1,w_2,\ldots)\in G^{\Z_{+}}$. For further equivalent definitions the Poisson boundary, we refer to \cite{KaimanovcihVershik1983}, \cite[Section 1]{Kaimanovich2000} and the references therein. We also refer to the surveys \cite{Furman2002,Erschler2010} for an overview of the study of Poisson boundaries and to the survey \cite{Zheng2022} for more recent applications of random walks and Poisson boundaries to group theory.

In this paper we will work with countable partitions of the space of sample paths. One of the most important partitions will be the following.
\begin{defn}\label{def: partition rw at time n}
	For every $n\ge 1$, define the partition $\alpha_n$ of the space of sample paths $G^{\mathbb{Z}}$, where two trajectories $\mathbf{w},\mathbf{w^\prime}\in G^{\Z_{+}}$ belong to the same element of $\alpha_n$ if and only if $w_n=w^\prime_n$.
\end{defn} 
In other words, $\alpha_n$ is the partition given by the $n$-th instant of the $\mu$-random walk.

The \emph{Shannon entropy} of a countable partition $\rho$ of the space of sample paths $G^{\Z_{+}}$ with respect to the probability measure $\P$ is defined as
\begin{equation*}
	H(\rho)\coloneqq -\sum_{k\ge 1}\P(\rho_k)\log \P(\rho_k).
\end{equation*} 

Note that for each $n\ge 1$ we have
\begin{equation*}
	H(\alpha_n)= -\sum_{g\in G}\P(w_n=g)\log \P(w_n=g)=H(\mu^{*n}),
\end{equation*} 
so that the entropy of the partition $\alpha_n$ coincides with the entropy of the convolution $\mu^{*n}$, which is the law of the $n$-th step of the $\mu$-random walk on $G$. From this, the asymptotic entropy $h_{\mu}$ of the random walk $(G,\mu)$ (Definition \ref{defn: asymptotic entropy}) can be expressed in terms of the partitions $\{\alpha_n\}_{n\ge 1}$ as $h(\mu)=\lim_{n\to \infty} \frac{H(\alpha_n)}{n}$.

\begin{thm}[{\cite{Derrienic1980}, \cite{KaimanovcihVershik1983}}] Let $G$ be a countable group and consider a probability measure $\mu$ on $G$. Suppose that $H(\mu)<\infty$. Then the Poisson boundary of $(G,\mu)$ is trivial if and only if $h(\mu)=0$. 
\end{thm}

\begin{rem}\label{rem: entropy as long scale Hausdorff dimension}
	The asymptotic entropy $h(\mu)$ of the $\mu$-random walk on $G$ can be interpreted as a ``large-scale Hausdorff dimension'' of the space of sample paths. More precisely, for every $\varepsilon>0$, the following hold:
	\begin{enumerate}
		\item there exists a sequence of finite subsets $A_n\subseteq G$ with $|A_n|\le \exp\left( (h(G,\mu)+\varepsilon)n\right)$ such that $\P$-almost surely $w_n\in A_n$ for $n$ sufficiently large, and
		\item for any sequence of finite subsets $B_n\subseteq G$ with $|B_n|\le \exp\left((h(G,\mu)-\varepsilon)n\right)$, it holds that $\P$-almost surely $w_n\notin B_n$ for $n$ sufficiently large;
	\end{enumerate}
	See e.g.\ \cite[Remark 2.29]{Furman2002}.
\end{rem}

Recall that a countable group $ G $ is called ICC if every non-trivial element in the associated quotient group has an infinite conjugacy class. The hyper-FC-center of a group $G$ is the minimal normal subgroup of $G$ with the property that the associated quotient group is ICC. A group $G$ is said to be hyper-FC-central if it coincides with its hyper-FC-center.  A finitely generated group is hyper-FC central if and only if it contains a nilpotent subgroup of finite index \cite[Theorem 2]{McLain1956}. It is known that a countable group $G$ is hyper-FC-central if and only if every probability measure on $G$ has the Liouville property. The fact that hyper-FC-central groups have this property goes back to \cite{Blackwell1955,ChoquetDeny1960,DoobSnellWilliamson1960} for abelian groups, to \cite{DynkinMaljutov1961,Margulis1966} for nilpotent groups, and to \cite{LinZaidenberg1998,Jaworski2004} for the general case. Reciprocally, it is proved in \cite[Theorem 1]{FrischHartmanTamuzVahidi2019} that any countable group that is not hyper-FC-central admits a non-degenerate probability measure with finite entropy and without the Liouville property. 

Throughout this paper, we will often use some well-known properties of entropy, which we list in the following lemma.
\begin{lem}\label{lem: basic properties entropy} Consider countable partitions $\rho,\gamma$ and $\delta$ of a Borel space. The following properties hold:
	\begin{enumerate}
		\item \label{item: entropy 1} $	H(\rho\vee \gamma \mid \delta)= H(\rho\mid \gamma\vee \delta)+H(\gamma\mid\delta).$
		\item \label{item: entropy 2} $H(\rho\mid \gamma)\le H(\rho\vee \delta\mid \gamma).$
		\item \label{item: entropy 3} $H(\rho\mid \gamma \vee \delta)\le H(\rho\mid \gamma).$
	\end{enumerate}
\end{lem}
We refer to \cite[Corollaries 2.5 and 2.6]{MartinEngland1981} and \cite[Section 5]{Rohlin1967} for the proofs of these properties, and to \cite[Chapter 2]{MartinEngland1981} for a more detailed exposition to entropy of probability measures.
\begin{rem}\label{rem: partitions defined by random variables}
	In order to simplify notation, throughout this paper we will use the same symbol to denote both a random variable and the partition of the space of sample paths that it defines. More precisely, let $X:(G^{\Z_{+}},\P)\to D$ be a random variable from the space of sample paths with values on a countable set $D$. The countable partition $\rho_X$ of $(G^{\mathbb{Z}_{+}},\P)$ defined by $X$ is given by saying that two sample paths  $\mathbf{w},\mathbf{w^\prime}\in G^{\mathbb{N}}$ belong to the same element of $\rho_X$ if and only if $X(\mathbf{w})=X(\mathbf{w^{\prime}})$. Then, we will denote $H(X)\coloneqq H(\rho_X)$.
\end{rem}

\subsection{Escape probability and the asymptotic range}\label{subsection: escape prob and the asymptotic range}
For a probability measure $\mu$ on a countable group $G$, denote by $R_n\coloneqq |\{w_0,w_1,\ldots, w_n\}|$ the number of distinct elements of $G$ visited by the $\mu$-random walk on $G$ up to time $n$. It is well known that the sequence $\{R_n/n\}_{n\ge 1}$ converges almost surely and in $L^1(G^{\Z_{+}})$ to a constant, called the \emph{asymptotic range} of the $\mu$-random walk. In fact we have
\begin{equation}\label{equation: asymptotic range}
	\lim_{n\to \infty}\frac{R_n}{n}=\pesc{\mu}= \P_{\mu}(w_n\neq e_G\text{ for all }n\ge 1).
\end{equation}
Indeed, the existence of the limit and the fact that it is a constant are consequences of Kingman's subadditive ergodic theorem \cite[Theorem 5]{Kingman1968}. The proof of the fact that its value coincides with the escape probability of the $\mu$-random walk can be found in \cite[Theorem I.4.1]{Spitzer1976} for $\Z^d$, $d\ge 1$, and in \cite[Lemma 1]{Dyubina1999} for the general case of a countable group. 

Furthermore, we will use the following well-known equality that relates the escape probability of the $\mu$-random walk with the expected number of returns to the origin; see e.g. \cite[Lemma 1.13 (a)]{Woess2000}.
\begin{lem}\label{lem: expected number of visits is 1 over 1-return prob}
	Let $\mu$ be a probability measure on a countable group $G$. Then
	\[	\sum_{n= 0}^{\infty}\mu^{*n}(e_G)=\frac{1}{\pesc{\mu}}.\]
\end{lem}

\subsection{Wreath products}\label{subsection: wreath products}
Given two groups $A$ and $B$, their \textit{wreath product} $A\wr B$ is defined as the semidirect product $\bigoplus_{B} A \rtimes B$, where $\bigoplus_{B} A $ is the group of finitely supported functions $f:B\to A$, endowed with the operation $\oplus$ of componentwise multiplication. Here, the group $B$ acts on the direct sum $\bigoplus_{B} A $ from the left by translations. In other words, for $f:B\to A$, and any $b\in B$ we have
$$
(b\cdot f)(x)=f(b^{-1}x), \ x\in B.
$$

Each element of $A\wr B$ is expressed as a pair $(f,b)$, where $f\in \bigoplus_{B} A$ and $b\in B$. The multiplication of two elements $(f,b)$, $(f^{\prime},b^{\prime})\in A\wr B$ is given by
$$
(f,b)\cdot(f^{\prime},b^{\prime})= (f\oplus (b\cdot f^{\prime}),bb^{\prime}).
$$
There is a natural embedding of $B$ into $A\wr B$ via the mapping
\begin{align*}
	B&\to A\wr B\\
	b&\mapsto (\mathds{1},b),
\end{align*}
where $\mathds{1}(x)=e_A$ for any $x\in B$. Similarly, we can embed $A$ into $A\wr B$ via the mapping
\begin{align*}
	B&\to A\wr B\\
	a&\mapsto (\delta^{a}_{e_B},e_B),
\end{align*}
where $\delta^{a}_{e_B}(e_B)=a$ and $\delta^{a}_{e_B}(x)=e_A$ for any $x\neq e_B$. In particular, whenever $A$ and $B$ are finitely generated groups, for any choice of finite generating sets $S_A$ and $S_B$ of $A$ and $B$, respectively, their copies inside $A\wr B$ through the above embeddings will generate $A\wr B$.

It is usual to call wreath products of the form $\Z/2\Z\wr B$ by the name ``lamplighter groups'' due to the following interpretation:
the Cayley graph $\cay{B}{S_B}$ is imagined as a street with lamps at every element, where each lamp can be switched between on and off states independently from the others. The identity of $\Z/2\Z$ corresponds to the lamp being turned off, and the non-trivial element of $\Z/2\Z$ corresponds to the lamp being turned on. Given this, an element $g=(f,b)\in \Z/2\Z\wr B$ consists of a \emph{lamp configuration} $f\in \bigoplus_{B} \Z/2\Z$, which encodes the state of the lamp at each element of $B$, together with a position $b\in B$, which corresponds to a person standing next to a particular lamp. Multiplying on the right by elements of $\Z/2\Z$ corresponds to switching the state of the lamp at position $b$, whereas multiplying on the right by an element of $B$ changes the position of the person, without modifying the lamp configuration.

For a probability measure $\mu$ on $A\wr B$, let us denote by $w_n=(\varphi_n,X_n)$, $n\ge 0$, a sample path of a random walk on the wreath product $A\wr B$. It is proved in \cite[Theorem 3.1]{Erschler2004Liouville} that if $\mu$ is non-degenerate and the induced random walk $\{X_n\}_{n\ge 0}$ on the base group $B$ is transient, then the Poisson boundary of $(A\wr B,\mu)$ is non-trivial. This result uses the entropy criterion and does not provide an explicit description of the Poisson boundary. For a large class of step distributions $\mu$, the following stabilization phenomenon will occur: $\P$-almost surely, for each $b\in B$, there is $N\ge 1$ such that for every $n\ge N$ we have $\varphi_n(b)=\varphi_N(b)$. That is, the lamp configurations $\{\varphi_n\}_{n\ge 0}$ converge pointwise to a (possibly infinitely supported) lamp configuration $\varphi_{\infty}$ in $\prod_B A$. This result was first proved for finitely supported $\mu$ in \cite[Section 6.2]{KaimanovcihVershik1983}, and then more generally for any $\mu$ with a finite first moment in \cite[Theorem 3.3]{Kaimanovich2001} for $B=\Z^d$ and in \cite[Lemma 1.1]{Erschler2011wreath} for general $B$. If the lamp configuration stabilizes almost surely, then the Poisson boundary of the $\mu$-random walk on $A\wr B$ has been described in many situations in terms of the space $(\prod_B A,\lambda)$, where $\lambda$ is the hitting measure \cite{JamesPeres1996,Kaimanovich2001,KarlssonWoess2007,Sava2010,Erschler2011wreath,LyonsPeres2021,FrischSilva2024}. We refer to the introduction of \cite{FrischSilva2024} for a detailed recollection of the contributions and the cases considered on each of the above results. It is important to note, however, that the stabilization of the lamp configuration does not hold in general for probability measures with an infinite first moment. There are examples where there is no stabilization for measures $\mu$ with a finite $(1-\varepsilon)$-moment, for any $\varepsilon>0$ \cite[Proposition 1.1]{Kaimanovich1983nontrivial} (see also \cite[Section 6]{Erschler2011wreath} as well as the last paragraph of Section 5 in \cite{LyonsPeres2021}). In such cases, there is no known explicit $\mu$-boundary, and hence no candidate for the Poisson boundary.

\subsection{Growth of groups}
We finish this section by recalling the basic notions related to the growth of finitely generated groups. We refer to \cite[Chapter VI]{delaHarpe2000} for a more detailed exposition.

\begin{defn}
	The \emph{growth function} of a group $G$ with respect to a finite and symmetric generating set $S$ is given by $v_{(G,S)}:\mathbb{N}\to \R$, where $v_{(G,S)}(n)=|(S\cup \{e_G\})^n|$ for each $n\ge 1$.
\end{defn}

In general, the values of the growth function depend on the choice of $S$. Nonetheless, choosing a different generating set $S$ preserves a natural equivalence relation on functions. Let $f,g:\R_{+}\to \R_{+}$ be increasing functions. We say that $f\preccurlyeq g$ if there exist $C_1,C_2>0$ such that $f(x)\le C_1g(C_1 x +C_2)+C_2$ for every $x\in \R_{+}$. We say that $f\sim g$ if $f \preccurlyeq g$ and $g \preccurlyeq f$.  If $S_1$ and $S_2$ are finite symmetric generating sets of $G$, then $v_{(G,S_1)}(n)\sim v_{(G,S_2)}(n)$. Because of this, we will omit the reference to $S$ in our notation and just write $v_G(\cdot)$ for the growth function of $G$.

We say that $G$ has \emph{polynomial growth} if there exists $d\ge 0$ such that $v_G(n)\preccurlyeq n^d$. Every finitely generated nilpotent group has polynomial growth \cite{Wolf1968}. Furthermore, it was proved by \cite{Bass1972} and \cite{Guivarch1973} that every finitely generated nilpotent group has a growth function $v_G(n)\sim n^d$, where $d\in \mathbb{N}$ is an integer determined by the lower central series of $G$. Gromov's Theorem of polynomial growth states that every finitely generated group of polynomial growth has a nilpotent subgroup of finite index \cite{Gromov1981}. As mentioned in the introduction, one can deduce from these results that any group that has growth function $v_G(n)\preccurlyeq n^2$ has a finite index subgroup that is either trivial, $\Z$ or $\Z^2$.

\section{Convergence of probability measures}\label{section: convergence of probability measures}
In this section, we will prove some general results regarding convergence of probability measures on countable groups, which will be used in the following sections of the paper. The results in this section are well-known, and we provide their proofs for the convenience of the reader.

We start by noticing that the pointwise converge of probability measures is equivalent to the convergence in $\ell^1(G)$.
\begin{lem}\label{lem: pointwise convergence is equivalent to total variation convergence}
	Let $\mu$ be a probability measure on a countable group $G$, and consider a sequence $\{\mu_k\}_{k\ge 1}$ of probability measures on $G$. The following are equivalent.
	\begin{enumerate}
		\item  $\sum_{g\in G}|\mu_k(g)-\mu(g)|\xrightarrow[k\to \infty]{} 0$ (i.e.\ convergence in total variation).
		\item $\mu_{k}(g)\xrightarrow[k\to \infty]{}\mu(g)$, for every $g\in G$ (i.e.\ pointwise convergence).
	\end{enumerate}
\end{lem}
\begin{proof}
	Suppose that we have $\lim_{k\to \infty}\sum_{g\in G}|\mu_k(g)-\mu(g)|= 0$. Then, for each $h\in G$ we have
	\(
	|\mu_k(h)-\mu(h)|\le \sum_{g\in G}|\mu_k(g)-\mu(g)|\xrightarrow[k\to \infty]{}0,
	\)
	and hence $\lim_{k\to \infty}\mu_k(h)=\mu(h)$ for each $h\in G$. Therefore, the sequence of probability measures $\{\mu_k\}_{k\ge 1}$ converges pointwise to $\mu$.
	
	Conversely, let us suppose that there is pointwise convergence, so that $\mu_{k}(g)\xrightarrow[k\to \infty]{}\mu(g)$ for every $g\in G$. Let us fix an arbitrary $\varepsilon>0$. Since $\mu$ is a probability measure, there is a finite subset $F\subseteq G$ such that $\mu(F)>1-\varepsilon$. The pointwise convergence of the sequence $\{\mu_k\}_{k\ge 1}$ together with the fact that $F$ is finite imply that there is $K\ge 1$ such that, for each $k\ge K$ and every $g\in F$, we have $|\mu_k(g)-\mu(g)|<\frac{\varepsilon}{|F|}$. Hence, it holds that, for each $k\ge K$,
	\begin{equation*}
		\mu_k(F)=\sum_{g\in F}\mu_k(g)\ge \sum_{g\in F}\left(\mu(g)-\frac{\varepsilon}{F}\right)=\mu(F)-\varepsilon>1-2\varepsilon.
	\end{equation*}
	Thus, we obtain that for each $k\ge K$,
	\begin{align*}
		\sum_{g\in G}|\mu_k(g)-\mu(g)|&=\sum_{g\in F}|\mu_k(g)-\mu(g)|+\sum_{g\in G\backslash F}|\mu_k(g)-\mu(g)|\\
		&\le \sum_{g\in F}\frac{\varepsilon}{|F|}+\sum_{g\in G\backslash F}(\mu_k(g)+\mu(g))\\
		&=\varepsilon+\mu_k(G\backslash F)+\mu(G\backslash F)\\
		&\le \varepsilon+2\varepsilon +\varepsilon=4\varepsilon.
	\end{align*}
	This shows that there is convergence in $\ell^1(G)$, thus finishing the proof.
\end{proof}

Using Lemma \ref{lem: pointwise convergence is equivalent to total variation convergence} we can show that the convergence of probability measures implies the convergence of their convolution powers.

\begin{lem}\label{lem: convolutions continuity}
	Let $\mu$ be a probability measure on a countable group $G$, and consider a sequence $\{\mu_k\}_{k\ge 1}$ of probability measures on $G$ such that $\lim_{k\to \infty}\mu_{k}(g)=\mu(g)$ for every $g\in G$. Then, for each $n\ge 1$, we have $\lim_{k\to \infty}\mu_{k}^{*n}(g)=\mu^{*n}(g)$ for every $g\in G$.
\end{lem}
\begin{proof}
	We proceed by induction. The base case $n=1$ is the hypothesis of the lemma, so there is nothing additional to prove. Suppose now that the statement is true for some $n\ge 1$, and let us show that it holds for $n+1$. Note that the inductive hypothesis tells us that we have pointwise convergence of the probability measures $\mu_k^{*n}$ to $\mu^{*n}$ as $k\to \infty$. Thanks to Lemma \ref{lem: pointwise convergence is equivalent to total variation convergence} this implies that we also have convergence in total variation.
	
	For each $g\in G$, we obtain
	\begin{align*}
		|\mu_{k}^{*(n+1)}(g)-\mu^{*(n+1)}(g)|&=\left| \sum_{h \in G}\mu^{*n}_{k}(gh^{-1})\mu_k(h)-\mu^{*n}(gh^{-1})\mu(h) \right|\\
		&\le \sum_{h \in G} \left|\mu^{*n}_{k}(gh^{-1})\mu_k(h)-\mu_k^{*n}(gh^{-1})\mu(h)\right|+\\&\hspace{20pt}+ \sum_{h \in G}\left|\mu_k^{*n}(gh^{-1})\mu(h)-\mu^{*n}(gh^{-1})\mu(h)\right| \\
		&\le \sum_{h \in G} \mu_k^{*n}(gh^{-1})\left|\mu_k(h)-\mu(h)\right|+\\&\hspace{20pt}+ \sum_{h \in G}\left|\mu_k^{*n}(gh^{-1})-\mu^{*n}(gh^{-1})\right|\mu(h) \\
		&\le  \sum_{h \in G}\left|\mu_k(h)-\mu(h)\right|+ \sum_{h \in G}\left|\mu_k^{*n}(h)-\mu^{*n}(h)\right|.
	\end{align*}
	The right hand side converges to $0$ as $k\to \infty$ since we have convergence in total variation, and hence we conclude the proof.
	
\end{proof}

\begin{lem}\label{lem: finite time return to e continuity}
	Let $\mu$ be a probability measure on a countable group $G$, and consider a sequence $\{\mu_k\}_{k\ge 1}$ of probability measures on $G$ such that $\lim_{k\to \infty}\mu_{k}(g)=\mu(g)$ for every $g\in G$. Denote by $\{w_n\}_{n\ge 0}\in G^{\Z_{+}}$ a random walk sample path. Then for each $n\ge 1$ and every $g\in G$, we have \[\P_{\mu_k}\left(\exists \ 1\le \ell\le n\text{ s.t. }w_{\ell}=g \right)\xrightarrow[k\to \infty]{}\P_{\mu}\left(\exists \ 1\le \ell\le n\text{ s.t. }w_{\ell}=g \right).\]
\end{lem}
\begin{proof}
	One can write the event where there is $1\le \ell \le n$ such that $w_\ell=g$ as a disjoint union of events in which we precise the exact moments between $1$ and $n$ where the random walk visits $g$. From this, we obtain the equality
	\begin{align*}
		\P_{\mu_k}\left(\exists \ 1\le \ell\le n\text{ s.t. }w_{\ell}=g \right)=\sum_{m=1}^n\sum_{1\le \ell_1<\ell_2<\cdots \ell_m\le n}\mu^{*\ell_1}_k(g)\mu_k^{*(\ell_2-\ell_1)}(g)\cdots \mu_k^{*(\ell_m-\ell_{m-1})}(g).
	\end{align*}
	Then, we use Lemma \ref{lem: convolutions continuity} to see that each convolution power of $\mu_k$ converges as $k\to \infty$ to the respective convolution power of $\mu$.
\end{proof}

The following proposition will be used several times in Section \ref{section: entropy estimates}. Recall from Definition \ref{def: escape proba} that $\pesc{\mu}\coloneqq \P_{\mu}(w_n\neq e_G \text{ for every }n\ge 1)$.

\begin{prop}\label{prop: uniform decay for return to a finite subset} Let $G$ be a countable group and let $\mu$ be a probability measure on $G$ such that $\langle\supp{\mu}\rangle_{+}$ is symmetric. Consider a sequence of probability measures $\{\mu_k\}_{k\ge 1}$ on $G$ such that $\lim_{k\to \infty}\mu_{k}(g)=\mu(g)$ for every $g\in G$. Suppose that the $\mu$-random walk on $G$ is transient, and that $\lim_{k\to \infty}\pesc{\mu_k}=\pesc{\mu}$. Then for every $\varepsilon>0$ and any finite subset $F\subseteq G$, there exist $K,n_0\ge 1$ such that for each $k\ge K$ we have
	\[
	\P_{\mu_k}\Big( \text{there is } \ell > n_0 \text{ such that }w_{\ell}\in F \Big)<\varepsilon.
	\]
\end{prop}
\begin{proof} 
	Recall that from Lemma \ref{lem: finite time return to e continuity} we have that the convergence of $\mu_k$ to $\mu$ implies that for every $n_0\ge 1$, we have
	\begin{equation}\label{eq: convergence at finite time of return}
		\P_{\mu_k}\left(\exists \ 1\le \ell\le n_0 \text{ such that }w_{\ell}=e_G\right)\xrightarrow[k\to \infty]{}\P_{\mu}\left(\exists\  1\le \ell\le n_0 \text{ such that }w_{\ell}=e_G\right).
	\end{equation}	
	Consider now an arbitrary $n_0 \ge 1$, and note that we have
	\begin{align*}
		\P_{\mu_k}\left( \exists\ell \ge 1 \text{ such that }w_{\ell}=e_G\right)&=\P_{\mu_k}\left(\exists 1\le \ell\le n_0 \text{ such that }w_{\ell}=e_G\right)+\\ & \hspace{30pt}+\P_{\mu_k}\left(\exists \ell >n_0 \text{ such that }w_{\ell}=e_G\right).
	\end{align*}
	The hypothesis of continuity of the escape probability implies that the term on the left side above converges to $\P_{\mu}\left( \exists\ \ell \ge 1 \text{ such that }w_{\ell}=e_G\right)$. In addition, the first term on the right converges to $\P_{\mu}\left( \exists\  1\le \ell\le n_0 \text{ such that }w_{\ell}=e_G\right)$ thanks to Equation \eqref{eq: convergence at finite time of return}. From this we obtain that, for each $n_0\ge 1$, we have
	\begin{equation}\label{eq: convergence of return prob beyond n0}
		\P_{\mu_k}\left(\exists \ell >n_0 \text{ such that }w_{\ell}=e_G\right)\xrightarrow[k\to \infty]{}\P_{\mu}\left( \exists\ell >n_0 \text{ such that }w_{\ell}=e_G\right)
	\end{equation}
	
	Let $\varepsilon>0$ and fix any finite subset $F\subseteq G$.
	Recall that we are assuming that $\langle\supp{\mu}\rangle_{+}$ is symmetric: hence, for each $g\in \langle\supp{\mu}\rangle_{+}$, there is $N\ge 1$ such that $g^{-1} \in \supp{\mu^{*N}}$. Together with the fact that $F$ is finite, this implies that we can find $m\ge 1$ such that 
	\[
	\inf_{g\in F} \P_{\mu}\left(\exists\ 0\le i\le m \text{ such that }w_i=e_G\mid w_0=g \right)
	>0.
	\]
	Furthermore, since the sequence $\{\mu_k\}_{k\ge 1}$ converges to $\mu$, we can find $K_1\ge 1$ and $\delta>0$ such that for every $k\ge K_1$ we have
	\[
	\inf_{g\in F} \P_{\mu_k}\left(\exists\ 0\le i\le m \text{ such that }w_i=e_G\mid w_0=g \right)>\delta>0.
	\]
	
	Let $K_1$ be as above. Then, for each $g\in F $ and any $n_0\ge1 $, we have
	\begin{align*}
		\P_{\mu_k}\Big(\exists \ell > n_0 \text{ s.t. }w_{\ell}=e_G \Big)&\ge \P_{\mu_k}\Big(\exists \ell > n_0 \text{ and }0\le i\le m \text{ s.t. }w_{\ell} =g \text{ and }w_{\ell+i}=e_G\Big)\\
		&=\P_{\mu_k}\Big(\exists 0\le i\le m \text{ s.t. }w_{\ell+i}=e_G\mid \exists \ell > n_0 \text{ s.t. }w_{\ell} =g\Big)\cdot\\ &\hspace{40pt}\cdot\P_{\mu_k}\left(\exists \ell > n_0 \text{ such that }w_{\ell} =g\right)\\
		&\ge \delta\P_{\mu_k}\left(\exists \ell > n_0 \text{ such that }w_{\ell} =g\right).
	\end{align*}
	As a consequence, we obtain
	\begin{align*}
		\P_{\mu_k}\Big( \exists \ell > n_0 \text{ such that }w_{\ell}\in F \Big)&\le\sum_{g\in F }\P_{\mu_k}\Big( \exists \ell > n_0 \text{ such that }w_{\ell}=g \Big)\\
		&\le\frac{|F|}{\delta}\P_{\mu_k}\Big(\exists \ell >n_0 \text{ such that }w_{\ell}=e_G \Big).
	\end{align*}
	
	We are assuming that the $\mu$-random walk on $G$ is transient, thus we can find $n_0\ge 1$ such that
	\[
	\P_{\mu}\left( \exists\  \ell> n_0 \text{ such that }w_{\ell}=e_G \right)<\frac{\varepsilon\delta}{2|F|}
	\]
	We can now use the convergence of Equation \eqref{eq: convergence of return prob beyond n0} to find $K_2\ge K_1$ such that, for every $k\ge K_2$, we have 
	\[
	\P_{\mu_k}\Big(\exists \ell > n_0 \text{ such that }w_{\ell}=e_G \Big)\le\P_{\mu}\Big(\exists \ell > n_0 \text{ such that }w_{\ell}=e_G \Big)+\frac{\varepsilon \delta}{2|F|}<\frac{\varepsilon \delta}{|F|}.
	\]
	Replacing this in the above inequality, we obtain that, for every $k\ge K_2$, 
	\begin{align*}
		\P_{\mu_k}\Big( \exists \ell > n_0 \text{ such that }w_{\ell}\in F \Big)&\le\frac{|F|}{\delta}\frac{\varepsilon \delta}{|F|}\le\varepsilon,
	\end{align*}
	which allows us to conclude the proof.
\end{proof}

In the remaining of this section, we prove convergence results for the Shannon entropy and asymptotic entropy of probability measures on countable groups.

The following result is contained in \cite[Section 3]{AmirAngelVirag2013}. 	We recall that a sequence $\{\mu_k\}_{k\ge 1}$ of probability measures on $G$ is said to be \emph{tight} if, for every $\varepsilon>0$, there is a finite subset $F\subseteq G$ such that $\sum_{g\in G\backslash F}\mu_k(g)<\varepsilon$ for every $k\ge 1$. In particular, every weakly convergent sequence of probability measures is tight. The sequence $\{\mu_k\}_{k\ge 1}$ is said to be \emph{entropy-tight} if, for every $\varepsilon>0$, there is a finite subset $F\subseteq G$ such that for all $k\ge 1$ we have $\sum_{g\in G\backslash F}-\mu_k(g)\log(\mu_k(g))<\varepsilon$.

\begin{lem}\label{lem: convolutions entropy convergence} Let $G$ be a countable group and consider a probability measure $\mu$ on $G$ with $H(\mu)<\infty$. Consider a sequence of probability measures $\{\mu_k\}_{k\ge 1}$ on $G$ such that $\lim_{k\to \infty}\mu_k(g)=\mu(g)$ for all $g\in G$, $H(\mu_k)<\infty$ for all $k\ge 1$, and $\lim_{k\to \infty}H(\mu_k)=H(\mu)$. Then, for any $n\ge 1$, we have $\lim_{k\to\infty}H(\mu^{*n}_k)=H(\mu^{*n})$.
\end{lem}
\begin{proof}
	For each $n\ge 1$, it follows from Lemma \ref{lem: convolutions continuity} that the sequence $\{\mu_k^{*n}\}_{k\ge 1}$ converges weakly to $\mu^{*n}$, so that this sequence is tight. In addition, it follows from \cite[Lemma 3.5]{AmirAngelVirag2013}, together with our assumption that $H(\mu_k)\xrightarrow[k\to \infty]{}H(\mu)$, that this sequence is entropy-tight. Then, we conclude from \cite[Lemma 3.2]{AmirAngelVirag2013} that $H(\mu_k^{*n})\xrightarrow[k\to \infty]{}H(\mu^{*n})$.
\end{proof}

Finally, we recall the fact that the asymptotic entropy is upper-semicontinuous as a function of the step distribution $\mu$. The following result is proved in \cite[Proposition 4]{AmirAngelVirag2013}; see also \cite[Lemma 1]{Erschler2011}, where this result is proved with the additional assumption that the supports of all the probability measures in the sequence are contained in a fixed finite set.
\begin{prop}[{\cite[Proposition 4]{AmirAngelVirag2013}}]\label{prop: upper semicontinuous entropy} Let $G$ be a countable group and consider a probability measure $\mu$ on $G$ together with a sequence of probability measures $\{\mu_{k}\}_{k\ge 1}$ on $G$. Suppose that $\lim_{k\to \infty}\mu_{k}(g)=\mu(g)$ for every $g\in G$, and that $\lim_{k\to \infty}H(\mu_k)=H(\mu)$. Then we have $\limsup_{k\to \infty}h(\mu_k)\le h(\mu)$.
\end{prop}

\section{Continuity of return probability to the origin}\label{section: continuity of return probability}

It is proved in \cite[Proposition 2]{Erschler2011} that, if $\mu$ is a finitely supported symmetric probability measure on a finitely generated group $G$ and $\{\mu_k\}_{k\ge 1}$ is a sequence of probability measures on $G$ such that 
\begin{enumerate}
	\item there exists a finite subset $F\subseteq G$ such that $\supp{\mu_k}\subseteq F$ for all $k\ge 0$, and
	\item $\lim_{k\to \infty}\mu_k(g)=\mu(g)$, for all $g\in G$,
\end{enumerate}
then $\lim_{k\to \infty} \pesc{\mu_k}=\pesc{\mu}.$ The results in \cite{Erschler2011} are written in terms of the asymptotic range of the $\mu$-random walk, but this value is the same as the escape probability; see Equation \eqref{equation: asymptotic range} in Subsection \ref{subsection: escape prob and the asymptotic range}. The objective of this section is to prove Theorem \ref{thm: continuity of range}, which generalizes the above result to a setting that includes infinitely supported probability measures, and non-symmetric limit probability measures $\mu$, under assumptions on the growth of the subgroup generated by the support of $\mu$.
\subsection{The comparison lemma of Coulhon and Saloff-Coste}

In this subsection, it will be convenient to use the notation $p_{\mu}^{(n)}(x,y)\coloneqq \mu^{*n}(x^{-1} y)$, for $\mu$ a probability measure on $G$, $n\ge 1$ and $x,y\in G$.

\begin{defn} Let $\mu$ be a probability measure on $G$. We define the Markov operator $T_{\mu}$ associated with $\mu$ by
	\[
	T_{\mu}f(g)\coloneqq\sum_{h\in G} p_{\mu}(g,h)f(h)=\sum_{h\in G} f(gh)\mu(h), \text{ for every }f:G\to \R.
	\]	
\end{defn} 
Note that, if $\mu$ is a probability measure on $G$ and $n\ge 1$, then the $n$-th convolution $\mu^{*n}$ is also a probability measure on $G$ and, moreover, $T_{\mu^{*n}}=T_{\mu}^n$.

For each probability measure $\mu$ on $G$ and for every $1\le p\le \infty$, the operator $T_{\mu}$ is bounded from $\ell^{p}(G)$ to $\ell^{p}(G)$. Since $T_{\mu}$ is also bounded from $\ell^{1}(G)$ to $\ell^{\infty}(G)$, the Riesz-Thorin interpolation theorem implies that $T_{\mu}$ is bounded from $\ell^{p}(G)$ to $\ell^{q}(G)$ for every $1\le p\le q\le \infty$.  We will denote by $\|T_{\mu}\|_{p\to q}$ the operator norm of $T_{\mu}:\ell^{p}(G)\to \ell^{q}(G)$, for each $1\le p\le q\le \infty$.  In the following lemma we recall well-known properties of the operator $T_{\mu}$ which we will use in the proofs below.
\begin{lem}\label{lem: Lp Markov operator basic facts}
	Let $\mu$ be a probability measure on a countable group $G$, and consider its associated operator $T_{\mu}$. Then the following hold.
	\begin{enumerate}
		\item \label{item: pp lem LpLq} For each $1\le p<\infty$, $\|T_{\mu}\|_{p\to p}\le 1$.
		\item \label{item: 12 lem LpLq} We have $\|T_{\mu}\|_{1\to 2}\le 1$.
		\item \label{item:  1infty lem LpLq} We have $\|T_{\mu}\|_{1\to \infty}=\sup_{x,y\in G}p_{\mu}(x,y).$
		\item \label{item: adjoint lem LpLq} Let $\check{\mu}$ be the reflected measure of $\mu$, i.e.\ $\check{\mu}(g)=\mu(g^{-1})$ for each $g\in G$. Consider $1\le p\le q\le \infty$, and $p^{\prime},q^{\prime}$ such that $1/p+1/p^{\prime}=1/q+1/q^{\prime}=1$. Then the adjoint operator of $T_{\mu}:\ell^p(G)\to \ell^q(G)$ is $T_{\check{\mu}}:\ell^{q^{\prime}}(G)\to \ell^{p^{\prime}}(G)$. In particular, $\|T_{\mu}\|_{p\to q}=\|T_{\check{\mu}}\|_{q^{\prime}\to p^{\prime}}$.
	\end{enumerate}
	
\end{lem}
\begin{proof}
	Items \eqref{item: pp lem LpLq}, \eqref{item: 12 lem LpLq} and \eqref{item:  1infty lem LpLq} can be checked by direct computation. For Item \eqref{item: adjoint lem LpLq}, we first note that with $p,p^{\prime},q,q^{\prime}$ as in the statement it holds that $\ell^{p^{\prime}}(G)$ is the dual of $\ell^p(G)$ and that $\ell^{q^{\prime}}(G)$ is the dual of $\ell^q(G)$. In addition, for any $f_1\in \ell^p(G)$ and $f_2\in \ell^q(G)$ we have
	\begin{align*}
	\langle T_\mu f_1,f_2\rangle &=\sum_{g\in G}\sum_{h\in G} f_1(gh)\mu(h)f_2(g)\\
	&=\sum_{h\in G}\sum_{g\in G}f_1(gh)\mu(h)f_2(g)\\
	&=\sum_{h\in G}\sum_{g\in G} f_1(g)\mu(h)f_2(gh^{-1})\\
	&=\sum_{g\in G}\sum_{h\in G} f_1(g)\check{\mu}(h)f_2(gh)\\
	&=\langle f_1,T_{\check{\mu}}f_2\rangle.
	\end{align*}
	This shows that $T_{\check{\mu}}$ is the adjoint operator of $T_{\mu}$. The last equality in the statement of Item \eqref{item: adjoint lem LpLq} follows from the general fact that the an operator between Banach spaces has the same norm as its adjoint.
\end{proof}

The following proposition is proved in \cite[Proposition IV.4]{Coulhon1996} (see also \cite[Section II]{CoulhonSaloffCoste1990}), without keeping track of what is the precise value of the constant $C$. The proof that we present here is a self-contained recollection from the results of \cite[Section IV]{Coulhon1996}, with minor modifications to control the value of $C$ that appears in the statement of the proposition.
\begin{prop}\label{prop: universal comparison constants} Let $\mu_1$ and $\mu_2$ be probability measures on a countable group $G$. Suppose that 
	\begin{enumerate}
		\item $\mu_1$ is symmetric,
		\item there is $C_1>1$ such that $\mu_1(g)\le C_1\mu_2(g)$ for every $g\in G$, and
		\item there exist $C_2>0$ and $d\ge 1$ such that 
		\[
		\sup_{x,y \in G} p_{\mu_1}^{(n)}(x,y)\le C_2 n^{-d/2}, \text{ for all }n\ge 1.
		\]
	\end{enumerate}
	Then there exists $C>0$ such that
	\begin{equation*}
		\sup_{x,y \in G} p_{\mu_2}^{(n)}(x,y) \le C n^{-d/2}, \text{ for all }n\ge 1.
	\end{equation*}
	
	The constant $C$ is completely determined by the values of $C_1$, $C_2$ and $d$, and does not depend on $\mu_2$ in any additional manner. 
\end{prop}
\begin{proof}
	Let $\mu_1$ and $\mu_2$ be probability measures on $G$ and $C_1,C_2>0$ satisfying the hypotheses from the proposition. We will write $T_i\coloneqq T_{\mu_i}$, for $i=1,2$.
	
	Let us consider the function $m:[1,\infty[\to ]0,C_2]$ defined by $m(x)\coloneqq C_2x^{-d/2}$ for each $x\ge 1$. Using Item \eqref{item:  1infty lem LpLq} from Lemma \ref{lem: Lp Markov operator basic facts}, the second hypothesis of the proposition can be expressed as \[
	\|T_1^n\|_{1\to \infty} =	\sup_{x,y \in G} p_{\mu_1}^{(n)}(x,y)\le m(n) \text{ for all }n\ge 1. 	\]
	In addition, using Item \eqref{item: pp lem LpLq} from Lemma \ref{lem: Lp Markov operator basic facts}, we get that for each $f\in \ell^1(G)$ we have
	\begin{align*}
		\|T_1^n f\|_{2}^2&=\sum_{x\in G}\left| T_1^{n}f (x) \right|^2\\
		&\le \sum_{x\in G}\left| T_1^{n}f (x) \right| \|T_1^{n}\|_{1\to \infty}\|f\|_1\\
		&=\|T_1^{n}f\|_1 \|T_1^{n}\|_{1\to \infty}\|f\|_1 \le \|T_1^{n}\|_{1\to \infty}\|f\|^2_1.
	\end{align*}
	
	Hence, we obtain \begin{equation}\label{eq: explicit constant step1}
		\|T_1^n\|^2_{1\to 2}\le \|T_1^{n}\|_{1\to \infty}\le m(n), \text{ for all }n\ge 1.
	\end{equation}
	As it is shown in the proof of \cite[Proposition IV.2]{Coulhon1996}, from Equation \eqref{eq: explicit constant step1} and the assumption that $\mu_1$ is symmetric it follows that for each $f\in \ell^1(G)\cap \ell^2(G) \backslash\{0\}$ with $\|f\|_1\le 1$, we have
	\begin{equation}\label{eq: explicit constant step2}
		\|f\|_2^2-\|T_1f\|_2^2\ge \frac{\|T_1f\|_2^2}{n}\log\left(\frac{\|f\|_2^2}{m(n)}\right), \text{ for all }n\ge 1.
	\end{equation}
	
	Before proceeding, we will need some auxiliary facts.  Define $$\theta(x)\coloneqq -m^{\prime}\circ m^{-1}(x)=\frac{d}{2}C_2^{-\frac{2}{d}}x^{1+\frac{2}{d}},\text{ and}$$ 
	\[\widetilde{\theta}(x)\coloneqq \sup_{n\ge 1}\frac{x}{n}\log\left( \frac{x}{m(n)} \right)=\sup_{n\ge 1}\frac{x}{n}\log\left( \frac{xn^{\frac{d}{2}}}{C_2} \right).\] 
	Then, it holds that
	\begin{equation}\label{eq: explicit constant step3part1}
		\widetilde{\theta}(x)\ge \frac{\log(2)}{3} \theta(x), \text{ for each }x\in ]0,C_2].
	\end{equation}
	
	Indeed, consider any $x\in ]0,C_2]$. Let $n_x\coloneqq \left\lceil 2 C_2^{\frac{2}{d}} x^{-\frac{2}{d}}\right\rceil\in \Z_{+}$, where $\lceil y \rceil$ denotes the smallest integer larger or equal than $y$. Then we have
	\begin{align*}
		\widetilde{\theta}(x)&\ge \frac{x}{n_x}\log\left( \frac{xn_x^{\frac{d}{2}}}{C_2 } \right)\\
		&= \frac{x}{\left\lceil 2 C_2^{\frac{2}{d}} x^{-\frac{2}{d}}\right\rceil}\log\left( \frac{ x\left(\left\lceil 2 C_2^{\frac{2}{d}} x^{-\frac{2}{d}}\right\rceil\right)^{\frac{d}{2}}}{C_2} \right) \\	
		&\ge \frac{x}{ 2 C_2^{\frac{2}{d}} x^{-\frac{2}{d}}+1}\log\left( \frac{ x\left( 2 C_2^{\frac{2}{d}} x^{-\frac{2}{d}}\right)^{\frac{d}{2}}}{C_2} \right) \\
		&= \frac{x^{1+\frac{d}{2}}}{ 2 C_2^{\frac{2}{d}}+ x^{\frac{2}{d}}}\frac{d}{2}\log(2)=\frac{C_2^{\frac{2}{d}} }{ 2 C_2^{\frac{2}{d}}+ x^{\frac{2}{d}}}\log(2)\theta(x)\ge \frac{\log(2)}{3}\theta(x).
	\end{align*}
	Additionally, note that
	\begin{equation}\label{eq: explicit constant step3part2}
		\widetilde{\theta}(x) =\sup_{n\ge 1}\frac{x}{n}\log\left( \frac{x}{C_2n^{-d/2}} \right)\le  \frac{x}{2}\widetilde{\theta}(2), \text{ for each }x\in ]0,2].
	\end{equation}
	Let $f\in \ell^1(G)\cap \ell^2(G) \backslash\{0\}$ with $\|f\|_1\le 1$, and recall that Equation \eqref{eq: explicit constant step2} is satisfied. Let us consider the following three cases, which together prove that Equation \eqref{eq: conclusion step 4} below holds.
	\\
	\textbf{Case 1.} Suppose that $\|T_1f\|_2^2\ge \frac{1}{2}\|f\|_2^{2}.$ Recall from Item \eqref{item: pp lem LpLq} from Lemma \ref{lem: Lp Markov operator basic facts} that  $\|T_2\|_{2\to 2}\le 1$. Using this in Equation \eqref{eq: explicit constant step2}, together with the fact that $\widetilde{\theta}$ is non-decreasing, we have that
	\begin{equation}\label{eq: explicit constant step4 case1}
		\|f\|_2^2-\|T_1f\|_2^2\ge \frac{1}{2}\widetilde{\theta}\left( \|f\|_2^2 \right)\ge \frac{1}{2}\widetilde{\theta} \left(\|T_2f\|_2^2 \right).
	\end{equation}
	\\
	\textbf{Case 2.} Suppose that $\|f\|_2^{2}-\|T_2f\|_2^2\ge 1.$ Since we are assuming that $\|f\|_1\le 1$, we obtain
	\begin{equation*}
		\frac{1}{2}\widetilde{\theta} \left( \|T_2f\|_2^2 \right)\le	\widetilde{\theta}\left( \|T_2f\|_2^2 \right)\le 	\widetilde{\theta}\left( \|T_2\|_{1\to 2}^2 \right)\le 	\widetilde{\theta}\left( \|T_2\|_{1\to 2}^2 \right)\left(\|f\|_2^{2}-\|T_1f\|_2^2\right).
	\end{equation*}
	
	Item \eqref{item: 12 lem LpLq} from Lemma \ref{lem: Lp Markov operator basic facts} says that $\|T_2\|_{1\to 2}\le 1$, so that we obtain
	\begin{equation} \label{eq: explicit constant step4 case2}
		\|f\|_2^{2}-\|T_1f\|_2^2\ge 
		\frac{1}{2\widetilde{\theta}\left(1\right)}\widetilde{\theta} \left( \|T_2f\|_2^2 \right).
	\end{equation}
	\\
	\textbf{Case 3.} Finally, suppose that $\|T_1f\|_2^2\le \frac{1}{2}\|f\|_2^{2}$ and $\|f\|_2^{2}-\|T_1f\|_2^2\le 1.$ Then we have $\frac{1}{2}\|f\|_2^2\le \|f\|_2^{2}-\|T_1f\|_2^2\le 1 $, so that $\|f\|_2^2\le 2$. Hence, we can use Equation \eqref{eq: explicit constant step3part2} to obtain
	\begin{equation*}
		\frac{1}{2}\widetilde{\theta} \left(\|T_2f\|_2^2 \right)\le \frac{1}{2}  \widetilde{\theta} \left( \|f\|_2^2 \right)\le   \frac{\widetilde{\theta}(2)}{4} \|f\|_2^2\le   \frac{\widetilde{\theta}(2)}{2} \left(\|f\|_2^2-\|T_1f\|_2^2 \right).
	\end{equation*}
	
	We get from this that
	\begin{equation} \label{eq: explicit constant step4 case3}
		\|f\|_2^{2}-\|T_1f\|_2^2\ge 
		\frac{1}{\widetilde{\theta}(2)}\widetilde{\theta} \left(\|T_2f\|_2^2 \right).
	\end{equation}

	Now let us consider 
	\[
	C_3=\min\left\{  \frac{1}{2},	\frac{1}{2\widetilde{\theta}\left(1\right)},\frac{1}{\widetilde{\theta}(2)} \right\}.
	\]
	
	It follows from Equations \eqref{eq: explicit constant step4 case1}, \eqref{eq: explicit constant step4 case2} and \eqref{eq: explicit constant step4 case3} that for all $f\in \ell^1(G)\cap \ell^2(G)$ with $\|f\|_1\le 1$ we have 
	\begin{equation*}
		\|f\|_2^{2}-\|T_1f\|_2^2\ge 
		C_3\widetilde{\theta} \left(\|T_2f\|_2^2 \right).
	\end{equation*}

	Recall that $\|f\|_1\le 1$ guarantees that $ \|T_2f\|_2^2 \le 1$. By using Equation \eqref{eq: explicit constant step3part1}, we get

	\begin{equation}\label{eq: conclusion step 4}
		\|f\|_2^2-\|T_1f\|_2^2\ge  \frac{C_3\log(2)}{3}\theta\left( \|T_2f\|_2^2 \right)=\frac{dC_3C_2^{-\frac{2}{d}}\log(2)}{6} \left( \|T_2f\|_2^2 \right)^{1+\frac{2}{d}},
	\end{equation}
	for all $f\in \ell^1(G)\cap \ell^2(G)$ with $\|f\|_1\le 1$.

	We have as a hypothesis that there is $C_1>1$ such that $\mu_1(g)\le C_1\mu_2(g)$ for every $g\in G$. Then, we can define the probability measure $\mu_3(g)\coloneqq \left(1-C_1^{-1}\right)^{-1}\left(\mu_2(g)-C_1^{-1}\mu_1(g)\right)$ for each $g\in G$. Hence, we have that $\mu_2=C_1^{-1}\mu_1+(1-C_1^{-1})\mu_3$. From this, we obtain
	\begin{align*}
		\|T_2f\|_2^{2}&= \| C_1^{-1}T_1f+(1-C_1^{-1})T_{\mu_3}f\|_2^2\\
		&\le \left( C_1^{-1}\|T_1f\|_2+(1-C_1^{-1})\|T_{\mu_3}f\|_2 \right)^2\\
		&\le  C_1^{-1}\|T_1f\|^2_2+(1-C_1^{-1})\|T_{\mu_3}f\|^2_2\\
		&\le C_1^{-1}\|T_1f\|^2_2+(1-C_1^{-1})\|f\|_2^2,
	\end{align*}
	where we used Jensen's inequality and Item \eqref{item: pp lem LpLq} from Lemma \ref{lem: Lp Markov operator basic facts}.
	Hence, we get		
	\[
	\|f\|_2^2-\|T_1f\|_2^2\le C_1\left( 	\|f\|_2^2-\|T_2f\|_2^2\right).
	\]
	
	Thus, we obtain from Equation \eqref{eq: conclusion step 4} that for each $f\in \ell^1(G)\cap \ell^2(G)$ with $\|f\|_1\le 1$,
	\begin{equation}\label{eq: explicit constant step 5}
		\|f\|_2^2-\|T_2f\|_2^2\ge\xi\left(\|T_2f\|_2^2 \right),
	\end{equation}
	where $\xi(x)\coloneqq \frac{dC_3C_2^{-\frac{2}{d}}\log(2)}{6C_1}x^{1+\frac{2}{d}}$. 

	Let us define $\widetilde{m}_1=1\ge \|T_2\|_{1\to 2}^2$, and for each $n\ge 1$ define $\widetilde{m}_{n+1}$ by the equation $\xi\left( \widetilde{m}_{n+1} \right)=\widetilde{m}_n- \widetilde{m}_{n+1}$. The latter is well-defined since the function $x\mapsto \xi(x)+x$ is increasing. Following the proof of \cite[Proposition IV.1]{Coulhon1996}, we have
	\[
	\|T_2^n\|_{1\to 2}^2\le \widetilde{m}_n\text{ for every }n\ge 1.
	\]
	Indeed, this holds for $n=1$. Suppose that it does not hold for some $n\ge 2$ and choose a minimal such $n$, so that $\|T_2^{(n-1)}\|_{1\to 2}^2\le \widetilde{m}_{n-1}$ but $\|T_2^n\|_{1\to 2}^2> \widetilde{m}_n$. Since $\xi$ is strictly increasing, we would have $\xi\left(\|T_2^n\|_{1\to 2}^2\right)> \xi\left(\widetilde{m}_n\right)$ and hence, for every $f\in \ell^1(G)\cap \ell^2(G)$ with $\|f\|_1=1$, it would follow that
	\[
	\xi\left(\|T_2^n f\|_2^2\right)\le 	\|T_2^{(n-1)}f\|_2^2-\|T_2^nf\|_{1\to 2}^2<\widetilde{m}_{n-1}-\widetilde{m}_{n}=\xi\left(\widetilde{m}_{n}\right)\le \xi\left(\|T_2^n f\|_2^2\right),
	\]
	which is a contradiction.

	Note that everything that we have done so far is also valid for the probability measure $\check{\mu}_2$. Hence, we also have that 
	\[
	\|T_{\check{\mu}_2}^n\|_{1\to 2}^2\le \widetilde{m}_n\text{ for every }n\ge 1.
	\]
	Using Item \eqref{item: adjoint lem LpLq} from Lemma \ref{lem: Lp Markov operator basic facts} we see that $\|T_2^n\|_{2\to \infty}	=\|T_{\check{\mu}_2}^n\|_{1\to 2}$, and hence we conclude that
	\[ \|T_2^{2n}\|_{1\to \infty}\le \|T_2^n\|_{1\to 2}\|T_2^n\|_{2\to \infty}\le \widetilde{m}_n.  \]
	
	Note that $\xi$ and the sequence $\{\tilde{m}_n\}_{n\ge 1}$ are defined only in terms of the constants $C_1$ and $C_2$. Recall that at the beginning of the proof we defined  $m(x)\coloneqq C_2x^{-d/2}$. Next, by following the proof of  \cite[Theorem IV.3]{Coulhon1996}, we see that there is $K\ge 1$ such that $\widetilde{m}_{Kn}\le m(n)$, for all $n\ge 1$. Indeed, the sequence  $\{\tilde{m}_n\}_{n\ge 1}$ is decreasing and converges to $0$ for $n\to \infty$. Therefore, we can choose $K\ge 1$ sufficiently large such that 
	\[
	\tilde{m}_K\le m(1)=C_2, \text{ and }K>\frac{3C_1}{C_3\log(2)2^{1+d/2}}.
	\]
	Thus,
	\begin{equation}\label{eq: sum 1}
		\widetilde{m}_{Kn}-\widetilde{m}_{K(n+1)}=\sum_{j=0}^{K-1}\widetilde{m}_{Kn+j}-\widetilde{m}_{Kn+j+1}=\sum_{j=0}^{K-1}\xi\left(\widetilde{m}_{Kn+j+1}\right)\ge K\xi(\widetilde{m}_{K(n+1)}).
	\end{equation}
	We already have that $\tilde{m}_K\le m(1)$, so let us suppose that there is some $n\ge 2$ such that $\tilde{m}_{Kn}> m(n)$. Choose a minimal such $n$, so that we have $\tilde{m}_{K(n-1)}\le m(n-1)$. Then, using Equation \eqref{eq: sum 1} we have
	\begin{equation}\label{eq: sum 2}
		\xi(\tilde{m}_{Kn})\le \frac{1}{K}\left(\widetilde{m}_{K(n-1)}-\widetilde{m}_{Kn}\right)\le \frac{1}{K}\left( m(n-1)-m(n) \right).
	\end{equation}
	
	In addition, we can find $x^{*}\in [0,1]$ such that $m(n)-m(n-1)=m^{\prime}(n-1+x^{*})$ by using the mean value theorem. Hence, from Equation \eqref{eq: sum 2} and thanks to our choice of $K$, we obtain 
	
	\begin{align*}
		\xi(\tilde{m}_{Kn})&\le -\frac{1}{K}m^{\prime}(n-1+x^{*})	\\
		&=\frac{dC_2}{2K}(n-1+x^{*})^{-1-d/2}\\
		&\le \frac{dC_2 C_3\log(2)}{6C_1 2^{1+d/2}}(n-1)^{-1-d/2}\\
		&\le \frac{dC_2 C_3\log(2)}{6C_1 2^{1+d/2}}n^{-1-d/2} \left(\frac{n}{n-1}\right)^{1+d/2}\\
		&\le \frac{dC_2 C_3\log(2)}{6C_1}n^{-1-d/2}\\
		&=\frac{dC_3C_2^{-2/d}\log(2)}{6C_1}\left(C_2 n^{-d/2}\right)^{1+2/d}=\xi\left(m(n)\right).
	\end{align*}
	This is a contradiction, since we had supposed that $\tilde{m}_{Kn}> m(n)$, and $\xi$ is increasing. Hence, we conclude that $\tilde{m}_{Kn}\le m(n)$ for every $n\ge 1$. Note that the constant $K$ depends only on the values of $C_1$ and $C_2$. Thus, we have that \[ \|T_2^{2Kn}\|_{1\to \infty}\le m(n)=C_2n^{-d/2}, \text{ for all } n\ge 1.  \]
	From this, we see that for every $x,y\in G$, any $n\ge 1$ and every $j=0,1,\ldots,\lceil 2K\rceil$, we have
	
	\begin{align*}
		p_{\mu_2}^{(2Kn+j)}(x,y)	&=\sum_{z\in G}p_{\mu_2}^{(2Kn)}(x,z)p_{\mu_2}^{(j)}(z,y)\\
		&\le C_2n^{-d/2}\sum_{z\in G}\mu_2^{*j}(z^{-1}y)=C_2n^{-d/2}.
	\end{align*}
	Let us consider $C=C_2\left( 4K+1\right)^{d/2}$. Then, for any $n\ge 1$ and every $j=0,1,\ldots,\lceil 2K\rceil$, we obtain
	\begin{align*}
		C_2n^{-d/2}&=C_2\left(2K+\frac{j}{n}\right)^{d/2}(2Kn+j)^{-d/2}\\
		&\le C_2\left(2K+\frac{2K+1}{1}\right)^{d/2}(2Kn+j)^{-d/2}\\
		&\le C_2\left(4K+1\right)^{d/2}(2Kn+j)^{-d/2}=C(2Kn+j)^{-d/2}.
	\end{align*}
	
	We finally conclude that 
	\[
	\sup_{x,y \in G} p_{\mu_2}^{(n)}(x,y) \le C n^{-d/2}, \text{ for all }n\ge 1,	
	\]
	where $C>0$ is a constant that is completely determined by the values of $C_1$ and $C_2$ from the statement of the theorem.
\end{proof}

\begin{lem}\label{lem: universal constant return proba}
	Let $\mu$ be a probability measure on a countable group $G$. Suppose that there is $N\ge 1$ such that the support of the $N$-th convolution power $\mu^{*N}$ contains a finite symmetric subset that generates a group $H$ whose growth function satisfies $v_H(n)\succcurlyeq n^d$ for some $d\ge1$. Consider a sequence $\{\mu_k\}_{k\ge 1}$ of probability measures on $G$ such that $\lim_{k\to \infty}\mu_{k}(g)=\mu(g)$ for every $g\in G$. 

	Then there exist $C,K>0$ such that for all $n\ge 1$ we have
	\begin{equation*}
		\sup_{k\ge K} \sup_{x,y \in G} p_{\mu_k}^{(n)}(x,y) \le Cn^{-d/2} \text{ and } \sup_{x,y \in G} p_{\mu}^{(n)}(x,y) \le Cn^{-d/2}.
	\end{equation*}
\end{lem}
\begin{proof}
	Let $N\ge 1$ be such that there is a finite symmetric subset $S\subseteq \supp{\mu^{*N}}$ that generates a group of growth at least $n^d$.	Let us define $\nu_S$ to be the uniform probability measure on $S$. Next, consider $C_1\coloneqq |S|\min_{s\in S}\mu^{*N}(s)>0$, so that we have
	\begin{equation}
		\mu^{*N}(g)\ge  C_1 \nu_S(g)\text{ for all }g\in G.
	\end{equation}
	
	Let $\{\mu_k\}_{k\ge 1}$ be a sequence of probability measures on $G$ such that $\lim_{k\to \infty}\mu_{k}(g)=\mu(g)$ for every $g\in G$. Then, thanks to Lemma \ref{lem: convolutions continuity}, we also have $\lim_{k\to \infty}\mu^{*N}_{k}(g)=\mu^{N}(g)$ for every $g \in G$. This implies that we can find $K\ge 1$ such that for each $k\ge K$, we have $\mu^{*N}_k(s)\ge \frac{1}{2}\mu^{*N}(s)\ge \frac{C_1}{2}\nu_S(s)$ for each $s\in S$. Let $C_2\coloneqq \frac{C_1}{2}$, and recall that $\nu_S(g)=0 $ for $g\in G\backslash S$. Then we have
	\begin{equation}
		\mu^{*N}_k(g)\ge C_2\nu_S(g), \text{ for every }g\in G \text{ and every }k\ge K.
	\end{equation}
	Since the growth function of the group $H\coloneqq \langle \supp{\mu}\rangle$ satisfies $v_H(n)\succcurlyeq n^d$ and $\nu_S$ is symmetric and finitely supported, it is a classical result of Varopoulos \cite{Varopoulos1987} (see also \cite[Theorems VI.3.3 \& VI.5.1]{VaropoulosSaloffCosteCoulhon1992}) that there exists a constant $C_3>0$ such that
	\[
	\sup_{x,y\in G}p^{(n)}_{\nu_S}(x,y)\le C_3 n^{-d/2}.
	\]
	Now we can use Proposition \ref{prop: universal comparison constants} to conclude that there is a constant $C>0$ such that 
	\[
	\sup_{k\ge K} \sup_{x,y \in G} p_{\mu_k}^{(Nn)}(x,y) \le Cn^{-d/2} \text{ and } \sup_{x,y \in G} p_{\mu}^{(Nn)}(x,y) \le Cn^{-d/2}.
	\]
	Note that the same constant $C$ works for all probability measures $\mu_k$, $k\ge K$, as well as for $\mu$. This implies that, for a possibly different constant $\widetilde{C}>0$, we have
	\begin{equation}\label{eq: final inequalities}
		\sup_{k\ge K} \sup_{x,y \in G} p_{\mu_k}^{(n)}(x,y) \le \widetilde{C}n^{-d/2} \text{ and } \sup_{x,y \in G} p_{\mu}^{(n)}(x,y) \le \widetilde{C}n^{-d/2}.
	\end{equation}
	Indeed, for any $x,y\in G$, $k\ge K$ and $j=0,1,\ldots, N-1$ we have
	
	\begin{align*}
		p_{\mu_k}^{(Nn+j)}(x,y)&= \sum_{z\in G}p_{\mu_k}^{(Nn)}(x,z) p_{\mu_k}^{(j)}(y,z)\\
		&\le Cn^{-d/2}\sum_{z\in G} \mu_k^{*j}(y^{-1}z)\\
		&=Cn^{-d/2}\\
		&=C(nN+j)^{-d/2}\left( \frac{nN+j}{n} \right)^{d/2}\\
		&=C(nN+j)^{-d/2}\left( N+\frac{j}{n} \right)^{d/2}\\
		&\le C(2N)^{d/2} (nN+j)^{-d/2}.
	\end{align*}
	One obtains a completely analogous inequality for $\mu$, and hence we get Equation \eqref{eq: final inequalities} with $\widetilde{C}\coloneqq  C(2N)^{d/2}$. This finishes the proof of the lemma.
\end{proof}

\subsection{The proof of Theorem \ref{thm: continuity of range}}\label{subsection: proof of continuity of range}
Now we are ready to prove the main result of this section.
\begin{proof}[Proof of Theorem \ref{thm: continuity of range}]
	Let $\mu$ be a probability measure on $G$, and suppose that $\langle \supp{\mu}\rangle_{+}$ contains a finitely generated subgroup of at least cubic growth. Then, in particular, this means that, for some $N\ge 1$,  $\supp{\mu^{*N}}$ contains a symmetric finite subset that generates a group of at least cubic growth. Thanks to Lemma \ref{lem: universal constant return proba}, we can find $C,K>0$ such that for all $n\ge 1$, we have
	\begin{equation}\label{eq: upper bound n32}
		\mu_k^{*n}(e_G)\le C n^{-3/2}, \text{ for all }k\ge K\text{ and }\mu^{*n}(e_G)\le C n^{-3/2}.
	\end{equation}
	
	Choose an integer $M\ge 1$ such that 
	\begin{equation}\label{eq: tail sum}
		\sum_{n= M+1}^{\infty} n^{-3/2}<\frac{\varepsilon}{4C}.
	\end{equation}
	Choose $\widetilde{K}\ge K$ such that for every $k\ge \widetilde{K}$ and every $n=0,1,\ldots,M$ we have
	\begin{equation}\label{eq: convolutions up tp M are close}
		|\mu_k^{*n}(e_G)-\mu^{*n}(e_G)|\le \frac{\varepsilon}{2(M+1)}
	\end{equation}
	
	Then for each $k\ge 0$ we have
	\begin{align*}
		\left|  \sum_{n= 0}^{\infty}\mu_k^{*n}(e_G)- \sum_{n= 0}^{\infty}\mu^{*n}(e_G) \right|&\le 	\left|  \sum_{n= 0}^{\infty}\mu_k^{*n}(e_G)- \sum_{n=0}^M\mu_k^{*n}(e_G) \right|+\left|  \sum_{n= 0}^{M}\mu_k^{*n}(e_G)- \sum_{n= 0}^{M}\mu^{*n}(e_G) \right|+\\
		&\hspace{10pt}+ \left|  \sum_{n= 0}^M\mu^{*n}(e_G)- \sum_{n= 0}^{\infty}\mu^{*n}(e_G) \right|\\
		&\le 	 \sum_{n= M+1}^{\infty}\mu_k^{*n}(e_G)+ \sum_{n= 0}^{M}\left|\mu_k^{*n}(e_G)- \mu^{*n}(e_G)\right| + \sum_{n= M+1}^{\infty}\mu^{*n}(e_G).
	\end{align*}
	Next, using Equations \eqref{eq: upper bound n32}, \eqref{eq: tail sum} and \eqref{eq: convolutions up tp M are close}, we get that for every $k\ge \widetilde{K}$ we have
	\begin{align*}
		\left|  \sum_{n= 0}^{\infty}\mu_k^{*n}(e_G)- \sum_{n= 0}^{\infty}\mu^{*n}(e_G) \right|&\le 2 C \sum_{n= M+1}^{\infty} n^{-3/2} + \sum_{n= 0}^{M} \frac{\varepsilon}{2(M+1)}\\
		&\le 2 C\frac{\varepsilon}{4C}+\frac{\varepsilon}{2}=\varepsilon.
	\end{align*}	
	From this, we conclude that \[\lim_{k\to \infty} \sum_{n= 0}^{\infty}\mu_k^{*n}(e_G)= \sum_{n= 0}^{\infty}\mu^{*n}(e_G).\]
	Finally, thanks to Lemma \ref{lem: expected number of visits is 1 over 1-return prob}, we obtain $\lim_{k\to \infty}\pesc{\mu_k}=\pesc{\mu}$, which is what we wanted to prove.
\end{proof}

\begin{rem}
	In the case of non-degenerate probability measures on $G=\Z^d$, $d\ge 3$, it is possible to provide an alternative proof of Theorem \ref{thm: continuity of range} that does not rely on Coulhon and Saloff-Coste's comparison lemma (Proposition \ref{prop: universal comparison constants}). The argument goes as follows. Let $\mu$ be a probability measure on $\Z^d$, $d\ge 1$, and consider the characteristic function $\phi_{\mu}:\Z^d\to \mathbb{C}$ of $\mu$, defined by $\phi_{\mu}(\mathbf{k})\coloneqq \E(e^{i\mathbf{k}\cdot X})$, $\mathbf{k}\in \Z^d$, where $X$ is a random variable distributed according to $\mu$. It is a result of Chung and Fuchs \cite[Theorem 3]{ChungFuchs1951} that
	\[ 
	\frac{1}{\pesc{\mu}}=\sum_{n= 0}^{\infty}\mu^{*n}(e_G)=\lim_{t\to 1^{-}}\frac{1}{(2\pi)^d}\int_{[-\pi,\pi]^d}\frac{1}{\re{1-t\phi_{\mu}(\mathbf{k})}}\ d \mathbf{k};
	\]
	see also \cite[Proposition II.8.1]{Spitzer1976}. In \cite[Proposition II.7.5]{Spitzer1976} it is proved that, if $d\ge 3$, then there exists a constant $C>0$ such that for any $\mathbf{k}\in [-\pi,\pi]^d$, we have $\re{1-\phi_{\mu}(\mathbf{k})}\ge C\|k\|^2$. Therefore, for each $\mathbf{k}\in [-\pi,\pi]^d$ and $t\in [0,1]$, we have
	\begin{equation}\label{eq: uniform decay Zd}
		\frac{1}{\re{1-t\phi_{\mu}(\mathbf{k})}}\le 1+\frac{1}{C\|k\|^2}.
	\end{equation}
	
	Since the function on the right side is integrable, the dominated convergence theorem implies that 
	\[
	\frac{1}{\pesc{\mu}}=\frac{1}{(2\pi)^d}\int_{[-\pi,\pi]^d}\frac{1}{\re{1-\phi_{\mu}(\mathbf{k})}}\ d \mathbf{k}<\infty.
	\]
	
	Now suppose that $\mu,\{\mu_k\}_{k\ge 1}$ are non-degenerate probability measures on $\Z^d$	such that $\mu_k(g)\xrightarrow[k\to \infty]{}\mu(g)$ for each $g\in \Z^d$. Then the characteristic functions $\{\phi_{\mu_k}\}_{k\ge 1}$ will converge pointwise to $\phi_{\mu}$. Furthermore, by following a similar argument to the one in the proof of \cite[Proposition II.7.5]{Spitzer1976}, one can find $K\ge1$ and a uniform constant $C>0$ such that Equation \eqref{eq: uniform decay Zd} holds simultaneously for all probability measures $\mu_k$, $k\ge K$. Then, another application of the dominated convergence theorem guarantees that  
	\[
	\frac{(2\pi)^d}{\pesc{\mu_k}}=\int_{[-\pi,\pi]^d}\frac{1}{\re{1-\phi_{\mu_k}(\mathbf{k})}}\ d \mathbf{k}\xrightarrow[k\to \infty]{}	\int_{[-\pi,\pi]^d}\frac{1}{\re{1-\phi_{\mu}(\mathbf{k})}}\ d \mathbf{k}=\frac{(2\pi)^d}{\pesc{\mu}},
	\]
	
	and hence we conclude that $\pesc{\mu_k}\xrightarrow[k\to \infty]{}\pesc{\mu}.$
\end{rem}
\section{Entropy estimates on wreath products}\label{section: entropy estimates}
In this, section we obtain entropy estimates for random walks on wreath products which will be used in the proof of Theorem \ref{thm: continuity asymptotic entropy wreath prods} in Section \ref{section: proof of the main theorem}, and hence subsequently in the proof of Theorem \ref{thm: main corollary continuity of asymptotic entropy over Zd}.

We will denote a sample path of a random walk on $A\wr B$ as
\[ 
w_n=(\varphi_n,X_n), \text{ where }\varphi_n:B\to A \text{ and }X_n\in B.
\]
Let us denote the independent increments of the random walk by $g_i=(f_i,Y_i)$, $i\ge 1$, so that

\[
(\varphi_n,X_n)=(\varphi_{n-1},X_{n-1})\cdot (f_n,Y_n), \text{ for each }n\ge 1.
\]

\subsection{The coarse trajectory}
We start by making a general definition of the $t_0$-coarse trajectory of a random walk on a group, for a fixed $t_0\ge 1$. Intuitively, it consists on recording the group element that is hit by the random walk every $t_0$ steps.
\begin{defn}\label{def: coarse trajectory} Let $\mu$ be a probability measure on a group $G$, and let $t_0\ge 1$. Recall that we denote by $\{w_n\}_{n\ge 0}$ a sample path of the random walk on $G$. We define the \emph{$t_0$-coarse trajectory at instant $n$ on the group $G$} as the ordered tuple
	\begin{equation*}
		\mathcal{P}_n^{t_0}(G)=\left(w_{t_0},w_{2t_0},\ldots, w_{\left\lfloor n/t_0\right \rfloor t_0}\right).
	\end{equation*}
	
	That is, $\mathcal{P}_n^{t_0}(G)$ consists of an ordered collection of the group element visited by the random walk, every $t_0$ steps, between time $0$ and $n$. 
\end{defn}

Throughout Sections \ref{section: entropy estimates}, \ref{section: entropy lamps inside} and \ref{section: proof of the main theorem} we will focus on random walks on wreath products $A\wr B$, and we will use the $t_0$-coarse trajectory for the induced random walk on the base group $B$. That is, if we denote by $\pi:A\wr B\to B$ the canonical projection, we will work with $\mathcal{P}_n^{t_0}(B)$ to be the $t_0$-coarse trajectory on $B$ defined by the probability measure $\pi_{*}\mu$. If there is no risk of confusion, we will just write $\mathcal{P}_n^{t_0}$ without making an explicit reference to the group $B$.

\begin{lem}\label{lem: coarse trajectory has small entropy} Let $\mu$ be a probability measure on a countable group $G$ with $H(\mu)<\infty$ and such that $h(\mu)=0$. Consider a sequence $\{\mu_k\}_{k\ge 1}$ of probability measures on $G$ such that
	\begin{itemize}
		\item  $H(\mu_k)<\infty$ for all $k\ge 1$,
		\item  $\lim_{k\to \infty}H(\mu_k)= H(\mu)$, and
		\item  $\lim_{k\to \infty}\mu_k(g)=\mu(g)$ for each $g\in G$.
	\end{itemize}
	Then for every $\varepsilon>0$, there exist $K,T\ge 1$ such that for all $k\ge K$ and $n\ge t_0\ge T$, we have $H_{\mu_k}(\mathcal{P}_n^{t_0})<\varepsilon n$.
\end{lem}
\begin{proof}
	Let $\varepsilon>0$. Since $\lim_{n\to \infty}\frac{H(w_n)}{n}=h(\mu)=0$, we can find $T\ge 1$ such that for each $t_0\ge T$ we have $H_{\mu}(w_{t_0})<\frac{\varepsilon}{2} t_0$. We are assuming that $\lim_{k\to \infty}\mu_k(g)=\mu(g)$ for each $g\in G$, and that $\lim_{k\to \infty}H(\mu_k)= H(\mu)$, so that thanks to Lemma \ref{lem: convolutions entropy convergence} we have that $\lim_{k\to \infty}H_{\mu_k}(w_{t_0})=H_{\mu}(w_{t_0})$. This implies that there is $K\ge 1$ such that $H_{\mu_k}(w_{t_0})<\varepsilon t_0$ holds for every $k\ge K$. The rest of the proof is similar to the proof of \cite[Lemma 3.4]{FrischSilva2024}. We repeat the argument for the convenience of the reader.
	
	Let $k\ge K$, $t_0\ge T$, $n\ge t_0$ and denote $s=\lfloor n/t_0\rfloor.$ Using Item \eqref{item: entropy 1} of Lemma \ref{lem: basic properties entropy}, we have that 
	\[ H_{\mu_k}(w_{t_0},w_{2t_0},w_{3t_0},\ldots, w_{st_0})= H_{\mu_k}(w_{t_0})+ H_{\mu_k}(w_{2t_0},w_{3t_0},\ldots, w_{st_0}\mid w_{t_0}). \]
	
	In addition, using Items \eqref{item: entropy 1} and \eqref{item: entropy 3} of Lemma \ref{lem: basic properties entropy} we get
	\begin{align*}
		H_{\mu_k}(w_{2t_0},w_{3t_0},\ldots, w_{st_0}\mid w_{t_0})&= H_{\mu_k}(w_{2t_0}\mid w_{t_0}) + H_{\mu_k}(w_{3t_0},\ldots, w_{st_0}\mid w_{t_0},w_{2t_0})\\
		&\le H_{\mu_k}(w_{2t_0}\mid w_{t_0}) + H_{\mu_k}(w_{3t_0},\ldots, w_{st_0}\mid w_{2t_0}),
	\end{align*} 
	so that \[ H_{\mu_k}(w_{t_0},w_{2t_0},w_{3t_0},\ldots, w_{st_0})\le H_{\mu_k}(w_{t_0})+ H_{\mu_k}(w_{2t_0}\mid w_{t_0}) + H_{\mu_k}(w_{3t_0},\ldots, w_{st_0}\mid w_{2t_0}). \]
	
	By repeating this argument $s-1$ times, we obtain the inequality \[H_{\mu_k}(w_{t_0},w_{2t_0},\ldots, w_{st_0})\le\sum_{j=0}^{s-1}H_{\mu_k}\left(w_{(j+1)t_0}\mid w_{jt_0}\right).\] Finally, using the above together with the fact that $H_{\mu_k}\left(w_{(j+1)t_0}\mid w_{jt_0}\right)=H_{\mu_k}(w_{t_0})$ for every $j=0,\ldots,s-1$, we can conclude that
	\begin{equation*}
		H_{\mu_k}(w_{t_0},w_{2t_0},\ldots, w_{st_0})\le sH_{\mu_k}(w_{t_0})
		< s\varepsilon t_0\le \varepsilon n.
	\end{equation*}
\end{proof}

\begin{rem}
	The only part of the proof of Theorem \ref{thm: continuity asymptotic entropy wreath prods} where the hypothesis $h(\pi_{*}\mu)=0$ is used is when we will apply Lemma \ref{lem: coarse trajectory has small entropy} to the $\pi_{*}\mu$-random walk on the base group $B$.
\end{rem}
\subsection{The partition of bad increments}
We now partition the time interval between $0$ and $n$ into subintervals of length $t_0$. Afterward, we define what it means for these intervals to be \emph{good} or \emph{bad}, depending on whether the increments of the random walk during the time instants belonging to the interval belong to a given fixed finite subset or not.
\begin{defn}\label{def: time intervals}
	Consider $t_0\ge 1$ and $n\ge t_0$. For $j=1,2,\ldots,\lfloor n/t_0\rfloor$, let us define the \emph{$j$-th interval} $I_j\coloneqq \{(j-1)t_0+1,\ldots, jt_0\}$, and the \emph{final interval} $I_{\mathrm{final}}\coloneqq \{\lfloor n/t_0\rfloor t_0+1,\ldots, n \}$.
\end{defn}
If $n$ is a multiple of $t_0$, then we have $I_{\mathrm{final}}=\varnothing.$

Recall that we denote by $\{g_i\}_{i\ge 1}$ the sequence of independent increments of the random walk on $A\wr B$.
\begin{defn}\label{def: bad elements and bad intervals}
	Let $A, B$ be countable groups and consider finite subsets $L\subseteq A$ and $R\subseteq B$ with $e_A\in L$ and $e_B\in R$.
	\begin{itemize}
		\item We say that a group element $g=(f,x)\in A\wr B$ is \emph{$(L,R)$-good} if $x\in R$, $\supp{f}\subseteq R$ and $f(b)\in L$ for each $b\in B$. Otherwise, call $g$ an \emph{$(L,R)$-bad} element.
		\item Let $t_0\ge 1$ and $n\ge t_0$. For each $j=1,2,\ldots,\lfloor n/t_0\rfloor$, we say that the interval $I_j$ is \emph{$(L,R)$-good} if all the increments $g_i$, $i\in I_j$, are $(L,R)$-good. Otherwise, we say that the interval $I_j$ is \emph{$(L,R)$-bad}.
	\end{itemize}
\end{defn}

\begin{rem}
	If $A$ and $B$ are finitely generated, we could alternatively define $g\in G$ to be a bad element if the word length of $g$ (with respect to some finite generating set of $A\wr B$) is larger than a given constant. The rest of the paper, and in particular the proofs of the main theorems in Section \ref{section: proof of the main theorem}, would work with this definition as well.
\end{rem}

We now introduce the auxiliary symbol $\star$, which will play the role of ``unknown'' information. We will use the symbol $\star$ to transform a random variable with finite Shannon entropy events into one with low entropy, by replacing its values by $\star$ in a large measurable subset.
\begin{defn}\label{defn: define the partition of bad and final intervals.}
	Let $n\ge t_0\ge 1$ and consider finite subsets $L\subseteq A$ and $R\subseteq B$ with $e_A\in L$ and $e_B\in R$. For every $j=1,2,\ldots,\lfloor n/t_0 \rfloor$, define the random variable
	\begin{equation*}
		Z_j\coloneqq\begin{cases}
			(g_i)_{i\in I_j}, &\text{ if }I_j \text{ is an }(R,L)\text{-bad interval, and}\\
			\star, &\text{otherwise,}
		\end{cases}
	\end{equation*}
	which has values in the space $(A\wr B)^{I_j}\cup \{\star\}$.
	Let us also define $Z_{\mathrm{final}}\coloneqq (g_i)_{i\in I_{\mathrm{final}}}$, which has values in the space $(A\wr B)^{I_{\mathrm{final}}}$. We define the partition of \emph{$(t_0,L,R)$-bad increments} by
	\begin{equation*}
		\beta_n(t_0,L,R)\coloneqq (Z_{1},Z_{2},\ldots,Z_{\lfloor n/t_0 \rfloor},Z_{\mathrm{final}}).
	\end{equation*}
\end{defn}
The symbol $\star$ serves as a tool for reducing the entropy of events with high entropy. By partially obscuring the values of a random variable using the symbol $\star$ when some very likely event occurs, we obtain a random variable with small entropy. This is used in the Lemma \ref{lem: bad increments have small entropy} below, and more precisely in Claim \ref{claim: uniform obscuring lemma}, in order to show that, if $L$ and $R$ are large enough, then the partition of $(t_0,L,R)$-bad increments has a small amount of entropy. 

\begin{lem}\label{lem: bad increments have small entropy}  Let $\mu$ be a probability measure on $A\wr B$ with $H(\mu)<\infty$, and consider a sequence $\{\mu_k\}_{k\ge 1}$ of probability measures on $A\wr B$ with $H(\mu_k)<\infty$ for all $k\ge 1$. Suppose that $\lim_{k\to \infty}\mu_k(g)=\mu(g)$ for all $g\in A\wr B$, and that $\lim_{k\to \infty}H(\mu_k)=H(\mu)$. Then for any $t_0\ge 1$ and every $\varepsilon>0$, there exist finite symmetric subsets $L\subseteq A$ and $R\subseteq B$ with $e_A\in L$ and $e_B\in R$, $C\ge 0$ and $K\ge 1$ such that for every $k\ge K$ and $n\ge t_0$, we have
	\begin{equation*}\label{eq: bad increments entropy}
		H_{\mu_k}(\beta_n(t_0,L,R))<\varepsilon n + C.
	\end{equation*}
\end{lem}
\begin{proof}
	Consider $t_0\ge 1$ and $\varepsilon>0$. We first remark that the interval $I_{\mathrm{final}}$ is composed of at most $t_0$ instants, so that
	\(
	H_{\mu_k}(Z_{\mathrm{final}})=H_{\mu_k}((g_i)_{i\in I_{\mathrm{final}}})\le t_0 H(\mu_k).
	\)
	We are assuming that $H(\mu_k)\xrightarrow[k\to \infty]{}H(\mu)$, and hence there is a constant $C\ge 0$ and $K_1\ge 1$ such that $H(\mu_k)\le C/t_0$ holds for every $k\ge K_1$. Combining this with the above, we have that
	\(
	H_{\mu_k}(Z_{\mathrm{final}})\le C \text{ for every } k\ge K_1.
	\)

	We now proceed to estimate the entropy of $(Z_1,\ldots, Z_{\lfloor n/t_0\rfloor})$.
	
	Let us introduce the auxiliary random variable $W_{j}\coloneqq(g_i)_{i\in I_j}$, for each $1\le j \le \lfloor n/t_0 \rfloor$. That is, $W_j$ corresponds to the ordered sequence of increments of the random walk that occurred during the interval $I_j$. Note that for each fixed step distribution $\mu_k$, the variables $W_{j}$ are identically distributed and have finite entropy $t_0H(\mu_k)$. 
	
	The following is a version of the ``obscuring lemma'' which appears in \cite[Lemma 2.4]{ChawlaForghaniFrischTiozzo2022} and \cite[Lemma 3.7]{FrischSilva2024}, that is uniform along the sequence of probability measures $\mu_k$.
	\begin{claim}\label{claim: uniform obscuring lemma} For any $\varepsilon>0$ and $t_0\ge 1$ there are $\delta>0$ and $K_2\ge 1$ such that for every $1\le j \le \lfloor n/t_0 \rfloor$ the following holds: if $E\subseteq (A\wr B)^{t_0}$ is a subset that satisfies $\mu^{t_0}(E)<\delta$, then the random variable
		\begin{equation*}
			\widetilde{W}_{j}\coloneqq\begin{cases}
				W_j, &\text{ if }E\text{ occurs}, \text{ and}\\
				\star, &\text{ otherwise,}
			\end{cases}
		\end{equation*}
		satisfies $H_{\mu_k}(\widetilde{W}_{j})<\varepsilon$ for every $k\ge K_2$.
	\end{claim}
	\begin{proof}
		For this proof, we introduce the notation $\kappa:[0,1]\to [0,1]$ for the function $\kappa(x)\coloneqq -x\log(x)$, where we use the standard convention $0\log(0)=0$. The function $\kappa$ is continuous, concave, strictly increasing in the interval $[0,1/e]$ and strictly decreasing in the interval $[1/e,1].$
		
		Let $\varepsilon>0$, $t_0\ge 1$ and fix any $1\le j \le \lfloor n/t_0 \rfloor$. Note that $W_j$ takes values in the countable set $D\coloneqq (A\wr B)^{t_0}$. Since $\sum_{d\in D}\kappa(\mu^{t_0}(d))=H_{\mu}(W_j)=t_0H(\mu)<\infty$, we can find a finite subset $Q\subseteq D$ such that
		\begin{equation}\label{eq: concentration entropy mu}
			\sum_{d\in D\backslash Q}\kappa(\mu^{t_0}(d))<\frac{\varepsilon}{2}.	
		\end{equation}
		Note that the choice of $Q$ is independent of the value of $j$.
		
		In addition to the above, the assumptions $H(\mu_k)\xrightarrow[k\to \infty]{}H(\mu)$ and $\mu_k\xrightarrow[k\to \infty]{}\mu$ imply that we can find $K_2\ge 1$ such that if $k\ge K_2$, then we have
		\begin{itemize}
			\item  $|H(\mu_k)-H(\mu)|<\frac{\varepsilon}{4}$, and
			\item  $\left|\sum_{d\in Q}\kappa(\mu_k^{t_0}(d))-\kappa(\mu^{t_0}(d))\right|<\frac{\varepsilon}{4}$.
		\end{itemize}
		Together with Equation \eqref{eq: concentration entropy mu}, this implies that for $k\ge K_2$ we have
		\begin{align*}
			\sum_{d\in D\backslash Q}\kappa(\mu_k^{t_0}(d)) &= H(\mu_k)-\sum_{d\in  Q}\kappa(\mu_k^{t_0}(d))\\
			&\le H(\mu) -\sum_{d\in  Q}\kappa(\mu^{t_0}(d))+ \frac{\varepsilon}{2}\\
			&=\sum_{d\in D\backslash Q}\kappa(\mu^{t_0}(d))+\frac{\varepsilon}{2}\\& <\varepsilon.
		\end{align*}
		Thanks to the continuity of $\kappa$ and the fact that $\kappa(0)=\kappa(1)=0$, we can find $ 0<\delta<1/e $ small enough such that $\kappa(1-\delta)+|Q|\kappa(\delta)<\varepsilon/2$. Consider an arbitrary measurable subset $E$ with $\mu^{t_0}(E)<\delta$ and consider $\widetilde{W_j}$ defined as in the statement of the claim. Then, for every $k\ge K_2$, we obtain
		\begin{align*}
			H_{\mu_k}(\widetilde{W_j})&=\kappa(\mu^{t_0}_k(\Omega\backslash E))+\sum_{d\in Q}\kappa(\mu^{t_0}_k( \{d\}\cap E ))+\sum_{d\in D\backslash Q}\kappa(\mu^{t_0}_k(\{d\}\cap E))\\
			&\le \kappa(1-\delta)+\sum_{d\in Q}\kappa(\mu^{t_0}_k(E))+\sum_{d\in D\backslash Q}\kappa(\mu^{t_0}_k(d))\\
			&\le\kappa(1-\delta)+|Q|\kappa(\delta)+\frac{\varepsilon}{2}\\
			&<\varepsilon,
		\end{align*}
		which concludes the proof.
	\end{proof}	
	Let $\delta$ and $K_2$ be as in Claim \ref{claim: uniform obscuring lemma}. Since $\mu$ is a probability measure, we can find finite symmetric subsets $L\subseteq A$ and $R\subseteq B$ with $e_A\in L$ and $e_{B}\in R$ such that $$\mu(\{g\in G\mid g\text{ is not }(L,R)\text{-bad}\})>1-\frac{\delta}{2t_0}.$$ Then, the pointwise convergence of the sequence $\{\mu_k\}_{k\ge 1}$ to $\mu$ guarantees that there is $K_3\ge K_2$ such that $\mu_k(\{g\in G\mid g\text{ is not }(L,R)\text{-bad}\})>1-\frac{\delta}{t_0}$ for every $k\ge K_3$. In particular, for every $1\le j \le \lfloor n/t_0 \rfloor$, we have
	\begin{equation*}
		\mu_k^{t_0}(I_j \text{ is an }(L,R)\text{-bad interval})\le t_0\mu_k(\{g\in G\mid g\text{ is }(L,R)\text{-bad}\})<\delta \text{ for all }k\ge K_3.
	\end{equation*}
	Consider the random variable $\widetilde{W}_{j}$ as defined in Claim \ref{claim: uniform obscuring lemma} associated with the event $$E_j\coloneqq \left \{I_j \text{ is an }(L,R)\text{-bad interval} \right\}.$$ Then $H_{\mu_k}(\widetilde{W}_{j})<\varepsilon$ for every $k\ge K_3$, and we have $Z_j=\widetilde{W}_{j}$. We therefore conclude that, for $k\ge \max\{K_1,K_3\}$, we have
	\begin{equation*}
		H(\beta_n(t_0,L,R))\le \sum_{j=1}^{ \lfloor n/t_0 \rfloor}H_{\mu_k}(Z_j)+Z_{\mathrm{final}}<\varepsilon n +C	\ \text{ for every }n\ge t_0.
	\end{equation*}
\end{proof}
\subsection{The coarse neighborhood}
In this subsection, we will define the coarse neighborhood in terms of the coarse trajectory in the base group $B$. For that, we write $\mathcal{P}_n^{t_0}\coloneqq \mathcal{P}_n^{t_0}(B)$. Intuitively, the coarse neighborhoods consists of the positions in the base groups where there could have been lamp modifications during good time intervals.

\begin{defn}\label{def:neighborhood of coarse traj}
	Let $t_0\ge 1$ and consider a finite subset $R\subseteq B$ with $e_B\in R$. For each $n\ge t_0$ define the \emph{$(t_0,R)$-coarse neighborhood} of the trajectory at instant $n$ by
	\begin{equation*}
		\mathcal{N}_n(t_0,R)\coloneqq \bigcup_{j=0}^{\lfloor n/t_0 \rfloor -1} 	X_{jt_0}R^{t_0},  \text{ where}
	\end{equation*}
	\begin{equation*}
		X_{jt_0}R^{t_0}\coloneqq \left\{X_{jt_0}r_1r_2\cdots r_{t_0}\ \Big|\ r_k\in R, \text{ for }k=1,2,\ldots,t_0 \right \}, \text{ for each }j=0,1,\ldots \lfloor n/t_0 \rfloor -1.
	\end{equation*}
\end{defn}

In order to estimate the entropy of the $n$-th instant of the random walk, it will be useful to divide the values of the lamp configuration into the ones inside the coarse neighborhood and the ones outside of it.

\begin{defn}\label{def: entropy in and out lamps} Consider $t_0\ge 1$ and a finite subset $R\subseteq B$ with $e_B\in R$. We define the \emph{lamp configuration inside the $(t_0,R)$-coarse neighborhood} $\mathcal{N}_n(t_0,R)$ as
	\begin{equation*} \Phi_n^{\mathrm{in}}\coloneqq \varphi_n|_{\mathcal{N}_n(t_0,R)},
	\end{equation*}
	and the \emph{lamp configuration outside the $(t_0,R)$-coarse neighborhood} $\mathcal{N}_n(t_0,R)$ as
	\begin{equation*}
		\Phi_{n}^{\mathrm{out}}\coloneqq \varphi_n|_{B\backslash \mathcal{N}_n(t_0,R)}.
	\end{equation*}
\end{defn}

For the proofs in Section \ref{section: proof of the main theorem}, it will be important that we are able to estimate not only the entropy of the lamp configuration $\varphi_n$ for $n$ large enough, but also the values $\varphi_{Nt_0}, \varphi_{Nt_0},\ldots, \varphi_{\lfloor n/Nt_0\rfloor Nt_0}$ for $N\ge 1$.

The following lemma states that the values of the lamp configuration outside of the $(t_0,R)$-coarse neighborhood every $Nt_0$ instants can be determined completely from the information of the $t_0$-coarse trajectory in the base group $B$, together with the partition of $(L,R)$-bad increments.

\begin{lem}\label{lem: entropy lamps outside}
	Let $\{\mu_k\}_{k\ge 1}$ be a sequence of probability measures on $A\wr B$ with $H(\mu_k)<\infty$ for all $k\ge 1$.	Then, for any $N,t_0\ge 1$, each $\varepsilon>0$, all finite subsets $L\subseteq A$ and $R\subseteq B$ with $e_A\in L$ and $e_B\in R$, every $k\ge 1$ and any $n\ge Nt_0$ it holds that
	\[
	H_{\mu_k}\left( \Phi_{Nt_0}^{\mathrm{out}},\Phi_{2Nt_0}^{\mathrm{out}},\ldots, \Phi_{\left\lfloor \frac{n}{Nt_0} \right \rfloor Nt_0}^{\mathrm{out}}  \Big | \ \mathcal{P}_n^{t_0} \vee \beta_n(t_0,L,R)\right)=0.
	\]
\end{lem}
\begin{proof}
	The random variable $\Phi_{n}^{\mathrm{out}}(t_0,L,R)$ consists of the lamp configuration at positions outside of the $(t_0,L,R)$-coarse neighborhood. By definition, any increment that modified the lamp configuration in one of these positions must have been realized during an $(L,R)$-bad interval, so that its value is contained in the partition $\beta_n(t_0,L,R)$. Hence, we conclude that $\Phi_{n}^{\mathrm{out}}(t_0,L,R)$ is completely determined by the values of $\mathcal{P}_n^{t_0}\vee \beta_n(t_0,L,R)$.
\end{proof}

\section{The entropy of lamps inside the coarse neighborhood}\label{section: entropy lamps inside}
The goal of this section is to estimate the entropy of the lamp configuration \emph{inside} the $(t_0,R)$-neighborhood, sampled every $Nt_0$ steps up to time $n$, and conditioned simultaneously on the element $w_n$ of the random walk at time $n$, the $t_0$-coarse trajectory in the base group, and the $(t_0,R,L)$-bad increments.

\begin{prop}\label{prop: entropy lamps inside} Let  $\mu$ be a probability measure on $A\wr B$ with $H(\mu)<\infty$, and consider a sequence $\{\mu_k\}_{k\ge 1}$ of probability measure on $A\wr B$ with $H(\mu_k)<\infty$ for all $k\ge 1$. Suppose that $\lim_{k\to \infty}\mu_k(g)=\mu(g)$ for all $g\in A\wr B$ and that $\lim_{k\to \infty}H(\mu_k)=H(\mu)$. Denote by $\pi:A\wr B\to B$ the canonical epimorphism, and let us furthermore suppose that 
	\begin{itemize}
		\item $\lim_{k\to \infty}\pesc{\pi_{*}\mu_k}=\pesc{\pi_{*}\mu}$
		\item  $\pi_{*}\mu$-random walk on $B$ is transient, and
		\item $\langle \supp{\pi_{*}\mu}\rangle_{+}$ is symmetric.
	\end{itemize}
	Consider any finite subsets $L\subseteq A$ and $R\subseteq B$ with $e_A\in L$ and $e_B\in R$. Then for every $\varepsilon>0$, there exist $K,n_0, T\ge 1$ such that for all $k\ge K$, every $t_0\ge T$, any $N>n_0$ and all $n> Nt_0$ we have 
	\[
	H_{\mu_k}\left( \Phi_{Nt_0}^{\mathrm{in}},\Phi_{2Nt_0}^{\mathrm{in}},\ldots, \Phi_{\left\lfloor \frac{n}{Nt_0} \right \rfloor Nt_0}^{\mathrm{in}}  \Big | \ w_n\vee \mathcal{P}_n^{t_0} \vee \beta_n(t_0,L,R)\right)<\varepsilon n +(H(\mu)+1)\frac{nn_0}{N}.
	\]
\end{prop}
This result is proved in Subsection \ref{subsection: proof of entropy inside}. In order to do this, we introduce additional partitions of the space of trajectories in Subsections \ref{subsection: unstable elements and visits} and \ref{subsection: unstable increments}, which do not appear in the proof of our main theorems outside of this section.
\subsection{Unstable elements and their visit times}\label{subsection: unstable elements and visits}
The objective of this subsection is to estimate the entropy contained in the lamp configuration at positions that are visited by the random walk in the base group at time instants that are far apart.

\begin{defn}\label{def: unstable elements and visits to unstable elements} Let $n_0,t_0\ge 1$ and let $F\subseteq B$ be a finite subset such that $e_{B}\in F$. For each $n>n_0t_0$ and each trajectory $\{(\varphi_i,X_i)\}_{i=0}^n$ of length $n$ on $A\wr B$, we define
	\begin{enumerate}
		\item \label{item: def unstable elements} the set $\mathcal{U}_n(n_0,t_0,F)$ of \emph{$(n_0,t_0,F)$-unstable points at instant $n$} as the one composed of the elements $b\in B$ such that there are $j\in \left\{0,1,2,\ldots,\left\lfloor n/t_0\right\rfloor-n_0\right\}$ and \\$\ell \in \left\{0,1,2,\ldots,\left\lfloor n/t_0\right\rfloor\right\}$ with $\ell > j+n_0$ and $b\in X_{jt_0}F\cap X_{\ell t_0}F$, and
		
		\item \label{item: def unstable visits} the set $\mathcal{V}_n(n_0,t_0,F)$ of \emph{visits to $(n_0,t_0,F)$-unstable points by time $n$} as the one composed by the instants $j\in \left\{0,1,2,\ldots,\left\lfloor n/t_0\right\rfloor -n_0\right\}$ such that \\ $X_{jt_0}F\cap \mathcal{U}_n(n_0,t_0,F)\neq \varnothing$.
	\end{enumerate}
\end{defn}

Lemmas \ref{lem: expected number of unstable points is small}, \ref{lem: entropy of unstable points is small}, \ref{lem: expected number of unstable visits is small}, \ref{lem: entropy of unstable visits is small} and \ref{lem: unstable increments have small entropy} below are all stated in the following context: let $\mu$ be a probability measure on $A\wr B$. Consider a sequence $\{\mu_k\}_{k\ge 1}$ of probability measures on $A\wr B$ such that $\mu_k(g)\xrightarrow[k\to \infty]{}\mu(g)$ for every $g\in A\wr B$. Denote by $\pi:A\wr B\to B$ the canonical epimorphism. Suppose that 
\begin{itemize}
	\item $\lim_{k\to \infty}\pesc{\pi_{*}\mu_k}=\pesc{\pi_{*}\mu}$
	\item  $\pi_{*}\mu$-random walk on $B$ is transient, and
	\item $\langle \supp{\pi_{*}\mu}\rangle_{+}$ is symmetric.
\end{itemize}
where we recall that $\pesc{\pi_{*}\mu}$ denotes the probability that the $\pi_{*}\mu$-random walk on $B$ never returns to the identity element $e_B\in B$. We will not repeat these hypotheses in the statements of these lemmas in order to make the exposition easier to read.

\begin{lem}\label{lem: expected number of unstable points is small} For each $\varepsilon>0$, every $t_0\ge 1$ and any finite subset $F\subseteq B$ with $e_{B}\in F$, there exist $ K, n_0\ge 1$ such that for all $k\ge K$ we have
	\[
	\E_{\mu_{k}}\Big(|\mathcal{U}_n(n_0,t_0,F)|\Big)<\varepsilon n, \text{ for all }n\ge n_0t_0.
	\]
\end{lem}

\begin{proof}
	Let us fix $\varepsilon>0$, $t_0\ge 1$ and any finite subset $F\subseteq B$ such that $e_{B}\in F$.
	
	For each $m\ge 1$ and every $j=0,1,\ldots,  \left\lfloor n/t_0\right \rfloor-m$, let us say that $j$ is \emph{poorly $m$-stabilized} if there is some $\ell \ge j+m$ such that $X_{jt_0}F\cap X_{\ell t_0}F\neq \varnothing$.
	
	Below we prove three claims, and then use them to conclude the statement of the lemma.
	\begin{claim}\label{claim: unstable points 1}
		For all $k,m\ge 1$ and $j\ge 0$, we have that
		\[
		\P_{\mu_k}\left( j\text{ is poorly }m\text{-stabilized} \right)\le\P_{\mu_k}\left( 0\text{ is poorly }m\text{-stabilized} \right).
		\]
	\end{claim}
	\begin{proof}
		Indeed, we see that
		\begin{align*}
			\P_{\mu_k}(j \text{ is poorly }m\text{-stabilized})&=\P_{\mu_k}(\text{there is }\ell\ge j+m\text{ such that }X_{jt_0}F\cap X_{\ell t_0}F\neq \varnothing)\\
			&=\P_{\mu_k}(\text{there is }\ell\ge j+m\text{ such that }X_{0}F\cap X_{(\ell-j) t_0}F\neq \varnothing)\\
			&\le\P_{\mu_k}(\text{there is }\ell\ge m\text{ such that }X_{0}F\cap X_{\ell t_0}F\neq \varnothing)\\
			&=\P_{\mu_k}(0 \text{ is poorly }m\text{-stabilized}).
	\end{align*}\end{proof}
	\begin{claim}\label{claim: unstable points 2}
		For all $k,t_0\ge 1$ and $m_1\ge m_2\ge 1$, we have that
		\[
		\P_{\mu_k}\left( 0\text{ is poorly }m_1\text{-stabilized} \right)\le \P_{\mu_k}\left( 0\text{ is poorly }m_2\text{-stabilized} \right).
		\]
	\end{claim}
	\begin{proof}
		This is a consequence of the fact that if $m_1\ge m_2$, then the event where $0$ is poorly $m_1$-stabilized is contained in the event where $0$ is poorly $m_2$-stabilized.
	\end{proof}
	\begin{claim}\label{claim: unstable points 3}
		There exist $K,n_0\ge 1$ such that for all $k\ge K$, all $m\ge n_0$ and any $j\ge 0$ we have
		\(
		\P_{\mu_k}\left(j \text{ is poorly }m\text{-stabilized} \right)< \frac{\varepsilon}{|F|}.
		\)
	\end{claim}
	\begin{proof}
		Using Proposition \ref{prop: uniform decay for return to a finite subset}, we can find $K, n_0\ge 1$ such that, for every $k\ge K$, we have 
		\begin{equation*}\label{eq: uniform decay to finite subset}
			\P_{\mu_k}\left(\exists \ \ell>n_0 \text{ such that }X_{\ell}\in FF^{-1} \right)<\frac{\varepsilon}{|F|}.
		\end{equation*}
		Then, for all $m\ge n_0$, $k\ge K$ and $j\ge 0$ we have
		\begin{align*}
			\P_{\mu_k}\left( j\text{ is poorly }m\text{-stabilized} \right)&\le \P_{\mu_k}\left( 0\text{ is poorly }m\text{-stabilized} \right)\\
			&=		\P_{\mu_k}\left( \text{there is }\ell\ge m \text{ such that }F \cap X_{\ell t_0}F\neq \varnothing \right)\\
			&=		\P_{\mu_k}\left( \text{there is }\ell\ge m \text{ such that } X_{\ell t_0}\in FF^{-1} \right)\\
			&\le	\P_{\mu_k}\left( \text{there is }\ell\ge n_0 \text{ such that } X_{\ell}\in FF^{-1} \right)\\
			&<\frac{\varepsilon}{|F|}.
		\end{align*}
	\end{proof}
	
	Let us now continue with the proof of the lemma. Note that the size of $\mathcal{U}_n(n_0,t_0,F)$ is at most the number of poorly $n_0$-stabilized instants, times the size of $F$. Then, for each $k\ge K$, we obtain
	\begin{align*}
		\E_{\mu_k}\left(|\mathcal{U}_n(n_0,t_0,F)|\right)&\le |F|\left( \sum_{j=0}^{\left\lfloor n/t_0\right \rfloor-n_0} \mathds{1}_{\left\{\exists \ell >j+n_0 \text{ s.t. }X_{jt_0}F\cap X_{mt_0}F\neq \varnothing \right\}} \right)\\
		&\le|F|\sum_{j=0}^{\left\lfloor n/t_0\right \rfloor-n_0}\P_{\mu_k}\left(\exists \ell>j+n_0 \text{ s.t. }X_{jt_0}F\cap X_{\ell t_0}F\neq \varnothing  \right)\\
		&\le|F|\sum_{j=0}^{\left\lfloor n/t_0\right \rfloor-n_0}\P_{\mu_k}\left(j\text{ is poorly }n_0\text{-stabilized} \right)\\
		&\le |F|\left(\lfloor n/t_0\rfloor -n_0+1 \right)\frac{\varepsilon}{|F|}\\
		&\le \varepsilon n,
	\end{align*}
	which concludes the proof of the lemma.
\end{proof}

The upper bound that we obtained for the expectation of the number unstable points gives an upper bound for the entropy of the set. The proof of the following result is completely analogous to the proof of \cite[Lemma 4.5]{FrischSilva2024}, which is a consequence of a general well-known entropy estimate (see \cite[Lemma 4.4]{FrischSilva2024} for the precise statement we use together with a self-contained proof).
\begin{lem}\label{lem: entropy of unstable points is small}
	Let $t_0\ge 1$ and consider any finite subset $F\subseteq B$ with $e_{B}\in F$ . Then for every $\varepsilon>0$, there exists $K,n_0\ge 1$ such that for all $k\ge K$ and $n>n_0t_0$ we have $H_{\mu_k}(\mathcal{U}_n(n_0,t_0,F)\mid \mathcal{P}_n^{t_0})<\varepsilon n.$
\end{lem}

\begin{lem}\label{lem: expected number of unstable visits is small} For each $\varepsilon>0$, any $t_0\ge 1$ and every finite subset $F\subseteq B$ with $e_{B}\in F$, there exist $K, n_0\ge 1$ such that for all $k\ge K$ we have
	\[
	\E_{\mu_{k}}\Big(|\mathcal{V}_n(n_0,t_0,F)|\Big)<\varepsilon n, \text{ for all }n>n_0 t_0.
	\]
\end{lem}
\begin{proof}
	Consider $\varepsilon>0$, $t_0\ge 1$ and $F\subseteq B$ a finite subset with $e_{B}\in F$. For every $m\ge 1$ let us define
	\begin{align*}
		J_{m}\coloneqq \{j\in \mathbb{Z}_{\ge 0}\mid \text{ there exist }n_1>n_2>\ldots> n_m>&j \text{ such that }\\
		&X_{n_\ell t_0}\in X_{jt_0}FF^{-1}, \text{ for }\ell=1,2,\ldots,m \}.
	\end{align*}
	That is, the set $J_{m}$ is formed by the instants $j\ge 0$ such that the neighborhood $X_{jt_0}FF^{-1}$ is visited at least $m$ more times in the future by the random walk on the base group $B$ at instants multiple of $t_0$.
	\begin{claim}\label{claim: visits 1}
		For every $j\ge 0$ and $m\ge 1$ we have $\P_{\mu_k}(j\in J_{m})\le \P_{\mu_k}(0\in J_{m})$.
	\end{claim}
	\begin{proof}
		Indeed, we can write 
		\begin{align*}
			J_{m}\coloneqq \{j\in \mathbb{Z}_{\ge 0}\mid \text{ there exist }n_1>n_2>&\ldots> n_m>j \text{ such that }\\
			&X_{jt_0}^{-1}X_{n_\ell t_0}\in FF^{-1}, \text{ for }\ell=1,2,\ldots,m \},
		\end{align*}
		so that the claim follows from the Markov property.
	\end{proof}
	\begin{claim}\label{claim: visits 2}
		For every $m_1\ge m_2\ge 1$ we have $\P_{\mu_k}(0\in J_{m_1})\le \P_{\mu_k}(0\in J_{m_2})$.
	\end{claim}
	\begin{proof}	
		This follows from the fact that, whenever the $FF^{-1}$-neighborhood of $e_B$ is visited at least $m_1$ times, then it is also visited at least $m_2$ times.
	\end{proof}
	\begin{claim}\label{claim: visits 3}
		There are $K_1,n_0\ge 1$ such that for all $k\ge K_1$ and every $m\ge n_0$ we have
		\[
		\P_{\mu_k}(0\in J_{m})<\frac{\varepsilon}{2(t_0+1)}.
		\]
	\end{claim}
	\begin{proof}
		Indeed, we know from Proposition \ref{prop: uniform decay for return to a finite subset} that there are $K_1,n_0\ge 1$ such that for every $k\ge K_1$ we have
		$\P_{\mu_k}(\text{ there is }\ell \ge n_0 \text{ such that }X_{\ell}\in FF^{-1})<\frac{\varepsilon}{2(t_0+1)}.$
		
		Then, for every and $k\ge K_1$ and $m\ge n_0$, we get
		\begin{align*}
			\P_{\mu_k}(0\in J_{m})&\le \P_{\mu_k}(FF^{-1}\text{ is visited at least }m \text{ times in the future})\\
			&\le \P_{\mu_k}(\exists \ \ell\ge m\text{ such that }X_{\ell}\in FF^{-1})\\
			&\le \P_{\mu_k}(\exists \ \ell\ge n_0\text{ such that }X_{\ell}\in FF^{-1})\\
			&<\frac{\varepsilon}{2(t_0+1)}.
		\end{align*}
	\end{proof}
	\begin{claim}\label{claim: visits 4}
		For every $k\ge 1$ and $n>n_0t_0$ we have \[\E_{\mu_k}(|\mathcal{V}_n(n_0,t_0,F)\backslash J_{n_0}|)\le n_0 |F| \E_{\mu_k}(|\mathcal{U}_n(n_0,t_0,F)|).\]
	\end{claim}
	\begin{proof}
		Let us denote $\mathcal{U}_n\coloneqq\mathcal{U}_n(n_0,t_0,F)$ and $\mathcal{V}_n\coloneqq\mathcal{V}_n(n_0,t_0,F)$. Then, we have
		\begin{align*}
			\E_{\mu_k}\left( |\mathcal{V}_n\backslash J_{n_0}| \right)&\le\E_{\mu_k}\left( \sum_{j=0}^{\lfloor n/t_0\rfloor-n_0}\mathds{1}_{\{X_{jt_0}F\cap \mathcal{U}_n\neq \varnothing\}}\mathds{1}_{\left\{\substack{ X_{jt_0}FF^{-1}\text{ is visited }< n_0 \\ \text{ times after instant }jt_0} \right\} } \right)\\		
			&=\E_{\mu_k}\left(\sum_{u\in \mathcal{U}_n}\sum_{g\in F} \sum_{j=0}^{\lfloor n/t_0\rfloor-n_0}\mathds{1}_{\{X_{jt_0}=ug^{-1}\}}\mathds{1}_{\left\{\substack{  X_{jt_0}FF^{-1}\text{ is visited }< n_0 \\ \text{ times after instant }jt_0} \right\} }\right)\\
			&\le\E_{\mu_k}\left(| \mathcal{U}_n| |F|n_0\right)\\
			&=n_0 |F| \E_{\mu_k}(| \mathcal{U}_n|),
		\end{align*}
		which proves the claim.
	\end{proof}
	
	We now use Lemma \ref{lem: expected number of unstable points is small} to find $K_2\ge K_1$ such that for all $k\ge K_2$ and every $n\ge n_0t_0$ we have 	$\E_{\mu_{k}}\Big(|\mathcal{U}_n(n_0,t_0,F)|\Big)<\frac{\varepsilon}{2n_0|F|} n.$ Combining this with Claims \ref{claim: visits 3} and \ref{claim: visits 4}, we obtain that for all $k\ge K_2$ and any $n>n_0t_0$ we have
	\begin{align*}
		\E_{\mu_k}\left( |\mathcal{V}_n| \right)&=\E_{\mu_k}\left( |\mathcal{V}_n\backslash J_{n_0}| \right)+\E_{\mu_k}\left( |J_{n_0}| \right)\\
		&\le   n_0 |F| \E_k\left(|\mathcal{U}_n|\right)+\sum_{j=0}^{t_0}\P_{\mu_k}\left(j \in J_{n_0}\right)\\
		&<  n_0 |F|\frac{\varepsilon}{2n_0|F|} n+(t_0+1)\P_{\mu_k}\left(0 \in J_{n_0}\right)\\
		&\le \frac{\varepsilon}{2}+ (t_0+1)\frac{\varepsilon}{2(t_0+1)}=\varepsilon.
	\end{align*}
	This finishes the proof of the lemma.
\end{proof}

\begin{lem}\label{lem: entropy of unstable visits is small} For every $\varepsilon>0$ and $n_0\ge 1$, there is $T\ge 1$ such that for all $k\ge 1$, all $t_0\ge T$ and any finite subset $F\subseteq B$ with $e_{B}\in F$, we have
	\[
	H_{\mu_{k}}\Big(\mathcal{V}_n(n_0,t_0,F))<\varepsilon n, \text{ for all }n> n_0t_0.
	\]
\end{lem}
\begin{proof}
	Consider any values of $\varepsilon>0$ and $n_0\ge 1$. The set $\mathcal{V}_n(n_0,t_0,F)$ is a subset of $\{0,1,\ldots, \lfloor n/t_0\rfloor -n_0\}$, so that \[H_{\mu_{k}}\Big(\mathcal{V}_n(n_0,t_0,F))\le \left( \lfloor n/t_0\rfloor -n_0+1\right)\log\left( 2 \right).\]
	
	Choose $T=\frac{\varepsilon}{\log(2)}$. Then, for every $t_0\ge T$, any finite subset $F\subseteq B$ with $e_{B}\in F$ and any $k\ge 1$, we obtain 
	\begin{equation*}
		H_{\mu_{k}}\Big(\mathcal{V}_n(n_0,t_0,F))\le \left( \frac{n}{t_0} -n_0+1\right)\log\left( 2 \right)\le   \left( \frac{n}{T}\right)\log\left( 2 \right)=\varepsilon n,
	\end{equation*}
	which concludes the proof of the lemma.
\end{proof}

\subsection{Increments at unstable points}\label{subsection: unstable increments}
In this subsection we will prove that, at the cost of adding a small amount of entropy to our process, we can reveal all the lamp increments done at positions that are visited by the random walk on the base group at distant time instants.
\begin{defn}\label{defn: unstable increments} Consider $n_0,t_0\ge 1$, finite subsets $L\subseteq A$ and $R\subseteq B$ with $e_A\in L$ and $e_B\in R$, and $n>n_0t_0$. We define $\Delta_n:\mathcal{U}_n(n_0,t_0,R^{t_0})\times\{1,2,\ldots ,\lfloor n/t_0\rfloor \} \to A$ such that, for each $j=1,2,\ldots,\lfloor n/t_0\rfloor$ and every $b\in \mathcal{U}_n(n_0,t_0,R^{t_0})$, the value $\Delta_n(b,j)$ is given by
	\[
	\Delta_n(b,j)\coloneqq \begin{cases}
		\varphi_{(j-1)t_0}(b)^{-1}\varphi_{jt_0}(b), &\text{ if }I_j\text{ is an }(R,L)\text{-good interval, and}\\
		*, &\text{ otherwise.}
	\end{cases}
	\]
	We call $\Delta_n(b,j)$ the \emph{unstable increment at $b$ during the interval $I_j$}. 
\end{defn}

\begin{lem}\label{lem: unstable increments have small entropy}
	Let $t_0\ge 1$ and consider any finite subsets $L\subseteq A$ and $R\subseteq B$ with $e_A\in A$ and $e_B\in R$. Then for every $\varepsilon>0$ there are $K,n_0\ge 1$ such that for every $k\ge K$ and any $n> n_0t_0$ we have
	\[
	H_{\mu_k}\left(\Delta_n\mid \mathcal{P}_n^{t_0}\vee \beta_n(t_0,L,R) \vee \mathcal{U}_n(n_0,t_0,R^{t_0})\vee \mathcal{V}_n(n_0,t_0,R^{t_0})\right)<\varepsilon n.
	\]
\end{lem}
\begin{proof}
	From Lemma \ref{lem: expected number of unstable visits is small}, we can find $K,n_0\ge 1$ such that, for all $k\ge K$  and $n>n_0t_0$, we have
	\[
	\E_{\mu_k}\left(\left|\mathcal{V}_n\right| \right)<\frac{\varepsilon n}{\log(|L^{t_0}|+1)} .
	\]
	
	In order to simplify the notation, let us denote $\beta_n\coloneqq \beta_n(t_0,L,R)$, $\mathcal{U}_n\coloneqq \mathcal{U}_n(n_0,t_0,R^{t_0})$ and $\mathcal{V}_n\coloneqq \mathcal{V}_n(n_0,t_0,R^{t_0})$.
	
	Recall that, for each $j=1,2,\ldots,\lfloor n/t_0\rfloor$ and $b\in \mathcal{U}_n$, the value of $\Delta_n(b,j)$ is either the symbol $*$ (if $I_j$ is an $(L,R)$-bad interval), or a product of at most $t_0$ elements of $L\subseteq A$. In particular, $\Delta_n(b,j)$ can take at most $|L^{t_0}|+1$ values. Additionally, note that if $\Delta_n(b,j)\neq *$ and $\Delta_n(b,j)\neq e_A$, then the lamp configuration at $b$ was modified during the interval $I_j$, which is an $(L,R)$-good interval. This implies that $j\in \mathcal{V}_n$. Moreover, the information of which intervals are $(L,R)$-good and which intervals are $(L,R)$-bad is contained in the partition $\beta_n$.

	Thus, we have
	\begin{align*}
		H_{\mu_k}\left(\Delta_n\mid \mathcal{P}_n^{t_0}\vee \beta_n\vee \mathcal{U}_n\vee \mathcal{V}_n \right)\le \E\left( |\mathcal{V}_n|\right)\log(|L^{t_0}|+1)<\varepsilon n.
	\end{align*}
\end{proof}
\subsection{The proof of Proposition \ref{prop: entropy lamps inside}}\label{subsection: proof of entropy inside}
We are now ready to prove the main proposition of this section, which estimates the entropy of the lamp configuration inside the $(t_0,R)$-coarse neighborhood $\mathcal{N}_n(t_0,R)$.
\begin{proof}[Proof of Proposition \ref{prop: entropy lamps inside}]
	Let us fix arbitrary finite subsets $L\subseteq A$ and $R\subseteq B$ with $e_A\in L$ and $e_B\in R$. Let $\varepsilon>0$.
	
	By using Lemmas \ref{lem: entropy of unstable points is small}, \ref{lem: entropy of unstable visits is small} and \ref{lem: unstable increments have small entropy} we can find $K\ge 1$, $n_0\ge 1$ and $T\ge 1$ such that for all $k\ge K$, $t_0\ge T$ and $n>n_0t_0$ we have
	
	\begin{enumerate}
		\item $H_{\mu_k}(\mathcal{U}_n\mid \mathcal{P}_n^{t_0})<\frac{\varepsilon}{3}n,$
		\item $H_{\mu_k}(\mathcal{V}_n)<\frac{\varepsilon}{3}n,$ and
		\item $H_{\mu_k}(\Delta_n\mid \mathcal{P}_n^{t_0}\vee \beta_n\vee \mathcal{U}_n\vee \mathcal{V}_n)<\frac{\varepsilon}{3}n$,
	\end{enumerate}

	where we denote $\beta_n\coloneqq\beta_n(t_0,L,R)$, $\mathcal{U}_n\coloneqq \mathcal{U}_n(n_0,t_0,R^{t_0})$, $\mathcal{V}_n\coloneqq \mathcal{V}_n(n_0,t_0,R^{t_0})$. Consider any $N>n_0$ and, in order to simplify the notation below, let us denote
	$$
	\Phi^{\mathrm{in}}_{Nt_0\mathrm{-coarse}}\coloneqq\left(\Phi_{Nt_0}^{\mathrm{in}},\Phi_{2Nt_0}^{\mathrm{in}},\ldots, \Phi_{\left\lfloor \frac{n}{Nt_0} \right \rfloor Nt_0}^{\mathrm{in}}\right).
	$$

	We thus obtain
	\begin{align*}
		H_{\mu_k}\left( \Phi^{\mathrm{in}}_{Nt_0\mathrm{-coarse}}  \Big | \ w_n\vee \mathcal{P}_n^{t_0} \vee \beta_n\right)&\le H_{\mu_k}\left( \Phi^{\mathrm{in}}_{Nt_0\mathrm{-coarse}}  \Big | \ w_n\vee \mathcal{P}_n^{t_0} \vee \beta_n\vee \mathcal{U}_n\vee \mathcal{V}_n\vee \Delta_n\right)+ \\ &\hspace{12pt}+
		H_{\mu_k}\left( \mathcal{U}_n\vee\mathcal{V}_n\vee \Delta_n  \Big | \ w_n\vee \mathcal{P}_n^{t_0} \vee \beta_n\right)\\
		&\le  H_{\mu_k}\left( \Phi^{\mathrm{in}}_{Nt_0\mathrm{-coarse}}  \Big | \ w_n\vee \mathcal{P}_n^{t_0} \vee \beta_n\vee \mathcal{U}_n\vee \mathcal{V}_n\vee \Delta_n\right)+ \\ &\hspace{12pt}+
		H_{\mu_k}\left(\mathcal{U}_n \Big | \mathcal{P}_n^{t_0}\right)+ H_{\mu_k}\left(\mathcal{V}_n\right)+ \\ &\hspace{12pt}+H_{\mu_k}\left(\Delta_n  \Big | \mathcal{P}_n^{t_0}\vee \beta_n\vee \mathcal{U}_n\vee \mathcal{V}_n\right)\\
		&\le  H_{\mu_k}\left( \Phi^{\mathrm{in}}_{Nt_0\mathrm{-coarse}}  \Big | \ w_n\vee \mathcal{P}_n^{t_0} \vee \beta_n\vee \mathcal{U}_n\vee \mathcal{V}_n\vee \Delta_n\right)+ \\ &\hspace{12pt}+\frac{\varepsilon}{3} n + \frac{\varepsilon}{3} n+\frac{\varepsilon}{3} n\\
		&=  H_{\mu_k}\left( \Phi^{\mathrm{in}}_{Nt_0\mathrm{-coarse}}  \Big | \ w_n\vee \mathcal{P}_n^{t_0} \vee \beta_n\vee \mathcal{U}_n\vee \mathcal{V}_n\vee \Delta_n\right)+ \varepsilon n.
	\end{align*}
	
	To conclude the result of the proposition, it suffices to justify that
	\begin{equation*}
		H_{\mu_k}\left( \Phi^{\mathrm{in}}_{Nt_0\mathrm{-coarse}}  \Big | \ w_n\vee \mathcal{P}_n^{t_0} \vee \beta_n\vee \mathcal{U}_n\vee \mathcal{V}_n\vee \Delta_n\right)\le (H(\mu)+1)\frac{nn_0}{N}.
	\end{equation*}
	We will write $\mathcal{Q}_n\coloneqq w_n\vee \mathcal{P}_n^{t_0} \vee \beta_n\vee \mathcal{U}_n\vee \mathcal{V}_n\vee \Delta_n$.

	For each $\ell=1,2,\ldots, \left \lfloor \frac{n}{Nt_0}\right \rfloor $, let us look at the value of $\Phi_{\ell Nt_0}^{\mathrm{in}}$, which corresponds to the lamp configuration at the time instant $\ell Nt_0$ at positions inside the subset \\ $F_{\ell}\coloneqq\bigcup_{j=0}^{\ell N-1}X_{jt_0}R^{t_0}$. We decompose the positions in $F_{\ell}$ into two disjoint subsets as follows:
	
	\begin{enumerate}
		\item Let us first consider elements $b\in F_{\ell}$ such that there is $j\in \{0,1,\ldots, \ell N-n_0-1\}$ for which $b\in X_{jt_0}R^{t_0}$.
		There are now two cases to consider.
		\begin{enumerate}
			\item If $b\notin \mathcal{U}_n$, then this means that any modification to the value of the lamp configuration at $b$ at any instant beyond $jt_0+n_0t_0$ must have occurred during an $(L,R)$-bad interval. Then, the value of $\varphi_{\ell N t_0}(b)$ can be completely determined from the value $\varphi_n(b)$ (which is part of the information of $w_n$), together with the partition of bad increments $\beta_n$.
			\item If $b\in \mathcal{U}_n$, then $\varphi_{\ell N t_0}(b)$ can be recovered by multiplying in the appropriate order the increments in the partition $\beta_n$ and the increments in the partition $\Delta_n$.
		\end{enumerate}
		
		\item We now consider all remaining elements $b\in F_{\ell}$. That is, elements $b\in B$ such that there is $j\in \{\ell N-n_0,\ldots, \ell N\}$ for which $b\in X_{jt_0}R^{t_0}$, and such that for every $j^{\prime}\in \{0,1,\ldots, \ell N-n_0-1\}$ it holds $b\notin X_{j^{\prime}t_0}R^{t_0}$. This last condition means that the value of $\varphi_{(\ell N-n_0) t_0}(b)$ is completely determined by the partition of bad increments $\beta_n$. From this, the value of $\varphi_{\ell N t_0}(b)$ can be obtained by further conditioning on the value of the increments done between instants $(\ell N-n_0) t_0$ and $\ell N t_0$. This corresponds to the group element $w^{-1}_{(\ell N-n_0)t_0}w_{\ell N t_0}$.
	\end{enumerate}
	Then, for each $\ell=1,2,\ldots, \left \lfloor \frac{n}{Nt_0}\right \rfloor$, we have that 
	\begin{align*}
		H_{\mu_k}\left( \Phi^{\mathrm{in}}_{\ell Nt_0}  \Big | \mathcal{Q}_n\right)&\le H_{\mu_k}\left( \Phi^{\mathrm{in}}_{\ell Nt_0}  \Big | \mathcal{Q}_n\vee w^{-1}_{(\ell N-n_0)t_0}w_{\ell N t_0}\right) + H_{\mu_k}(w^{-1}_{(\ell N-n_0)t_0}w_{\ell N t_0})\\
		&=0+H_{\mu_k}(w^{-1}_{(\ell N-n_0)t_0}w_{\ell N t_0})\\
		&\le H(\mu_k)n_0t_0.
	\end{align*}
	Hence, we obtain
	\begin{equation*}
		H_{\mu_k}\left( \Phi^{\mathrm{in}}_{Nt_0\mathrm{-coarse}}  \Big |\mathcal{Q}_n\right)\le \sum_{\ell=1}^{\left\lfloor \frac{n}{Nt_0}\right\rfloor}H_{\mu_k}\left( \Phi^{\mathrm{in}}_{\ell Nt_0}  \Big | \mathcal{Q}_n\right) \le\left\lfloor \frac{n}{Nt_0}\right\rfloor H(\mu_k)n_0t_0\le  H(\mu_k)\frac{nn_0}{N}.
	\end{equation*}
	Finally, since the sequence $\{H(\mu_k)\}_{k\ge 1}$ converges to $H(\mu)$, for each $k$ sufficiently large we will have $H(\mu_k)\le H(\mu)+1$. Thus, we conclude that there  are $K\ge 1$, $n_0\ge 1$ and $T\ge 1$ such that, for all $k\ge K$, $t_0\ge T$, $N>n_0$ and $n>Nt_0$, we have
	
	\begin{align*}
		H_{\mu_k}\left( \Phi^{\mathrm{in}}_{Nt_0\mathrm{-coarse}}  \Big | \ w_n\vee \mathcal{P}_n^{t_0} \vee \beta_n\right)&\le \varepsilon n +(H(\mu)+1)\frac{nn_0}{N},
	\end{align*}
	
	which is what we wanted.
	
\end{proof}
\section{The coarse trajectory on $A\wr B$ and the proof of Theorem \ref{thm: continuity asymptotic entropy wreath prods}}\label{section: proof of the main theorem}
We begin this section by stating the most general version of continuity of entropy on wreath products that we prove in this paper.
\begin{thm}\label{thm: continuity asymptotic entropy wreath prods}
	Let $A$ and $B$ be countable groups and let $\mu$ be a probability measure on $A\wr B\coloneqq \bigoplus_{B}A\rtimes B$ with $H(\mu)<\infty$. Consider a sequence $\{\mu_k\}_{k\ge 1}$ of probability measures on $A\wr B$ with $H(\mu_k)<\infty$ for all $k\ge 1$, and such that 
	\begin{enumerate}[(1)]
		\item\label{item: main thm 1} $\lim_{k\to \infty }\mu_k(g)=\mu(g)$ for each $g \in A\wr B$, and
		\item\label{item: main thm 2} $\lim_{k\to \infty} H(\mu_k)=H(\mu)$.
	\end{enumerate}
	Denote by $\pi:A\wr B\to B$ the canonical epimorphism to the base group $B$. Suppose furthermore that
	\begin{enumerate}[(1)]\setcounter{enumi}{2}
		\item\label{item: main thm 3}  the $\pi_{*}\mu$-random walk on $B$ is transient,
		\item \label{item: main thm 4} $h(\pi_{*}\mu)=0$,
		\item\label{item: main thm 5} $\langle \supp{\pi_{*}\mu}\rangle_{+}$ is symmetric, and
		\item\label{item: main thm 6} $\lim_{k\to \infty}\pesc{\pi_{*}\mu_k}=\pesc{\pi_{*}\mu}$.
	\end{enumerate} Then $\lim_{k\to \infty}h(\mu_k)=h(\mu)$.
\end{thm}

Now, we will use the results from Sections \ref{section: entropy estimates} and \ref{section: entropy lamps inside} in order to prove Lemma \ref{lem: main technical lemma} below. It states that, conditioned simultaneously on the value of the random walk on $A\wr B$ at time $n$, the coarse trajectory, and the visits to unstable points, the joint entropy contained in the values of the random walk every $Nt_0$ steps is small. 

\begin{lem}\label{lem: main technical lemma} Let  $\mu$ be a probability measure on $A\wr B$ with $H(\mu)<\infty$, and consider a sequence $\{\mu_k\}_{k\ge 1}$ of probability measure on $A\wr B$ with $H(\mu_k)<\infty$ for all $k\ge 1$. Suppose that $\lim_{k\to \infty}\mu_k(g)=\mu(g)$ for all $g\in A\wr B$ and that $\lim_{k\to \infty}H(\mu_k)=H(\mu)$. Denote by $\pi:A\wr B\to B$ the canonical epimorphism, and let us furthermore suppose that $\lim_{k\to \infty}\pesc{\pi_{*}\mu_k}=\pesc{\pi_{*}\mu}$, that the $\pi_{*}\mu$-random walk on $B$ is transient and that $\langle \supp{\pi_{*}\mu}\rangle_{+}$ is symmetric. Then, for every $\varepsilon>0$ and every finite subsets $L\subseteq A$ and $R\subseteq B$ with $e_A\in L$ and $e_{B}\in R$, there exist $K,n_0,T\ge 1$ such that for all $k\ge K$, $t_0\ge T$, $N>n_0$ and $n> Nt_0$ we have
	\[
	H_{\mu_k}\Big(\mathcal{P}_n^{Nt_0}(A\wr B)\Big \vert w_n\vee \mathcal{P}_n^{t_0}(B)\vee  \beta_n(t_0,L,R) \Big)\le \varepsilon n +(H(\mu)+1)\frac{nn_0}{N}.
	\]
\end{lem}
\begin{proof}
	Let us consider any $\varepsilon>0$ and finite subsets $L\subseteq A$ and $R\subseteq B$ with $e_A\in L$ and $e_{B}\in R$. We can apply Proposition \ref{prop: entropy lamps inside} to find $K,n_0,T\ge 1$ such that for all $k\ge K$, every $t_0\ge T$, any $N>n_0$ and all $n> Nt_0$ we have
	\[
	H_{\mu_k}\left(\Phi^{\mathrm{in}}_{Nt_0},\Phi^{\mathrm{in}}_{2Nt_0},\ldots, \Phi^{\mathrm{in}}_{\left\lfloor \frac{n}{Nt_0} \right\rfloor Nt_0}\Big| w_n\vee \mathcal{P}_n^{t_0}\vee \beta_n(t_0,L,R)\right)<\varepsilon n +(H(\mu)+1)\frac{nn_0}{N}.
	\]
	In addition, from Lemma \ref{lem: entropy lamps outside}, we have that 	
	\[
	H_{\mu_k}\left(\Phi^{\mathrm{out}}_{Nt_0},\Phi^{\mathrm{out}}_{2Nt_0},\ldots, \Phi^{\mathrm{out}}_{\left\lfloor \frac{n}{Nt_0} \right\rfloor Nt_0}\Big| \mathcal{P}_n^{t_0}(B)\vee \beta_n(t_0,L,R)\right)=0.
	\]
	
	Combining these two equations together with the fact that for each $j=1,\ldots, \left\lfloor \frac{n}{Nt_0} \right\rfloor$, the lamp configuration $\varphi_{jNt_0}$ at some instant $jNt_0$, is completely determined by the combination of the values of $\Phi^{\mathrm{out}}_{jNt_0}$ and $\Phi^{\mathrm{in}}_{jNt_0}$, we get that
	
	\[
	H_{\mu_k}\left(\varphi_{Nt_0},\varphi_{2Nt_0},\ldots, \varphi_{\left\lfloor \frac{n}{Nt_0} \right\rfloor Nt_0}\Big| w_n\vee \mathcal{P}_n^{t_0}(B)\vee \beta_n(t_0,L,R)\right)<\varepsilon n +(H(\mu)+1)\frac{nn_0}{N}.
	\]	
	
	Finally, we note that $\mathcal{P}_n^{Nt_0}(A\wr B)$ is completely determined by the values $\Big(\varphi_{jNt_0}\Big)_{j=1,\ldots,\left\lfloor \frac{n}{Nt_0} \right\rfloor}$ of the lamp configurations, together with the values $\Big(X_{jNt_0}\Big)_{j=1,\ldots,\left\lfloor \frac{n}{Nt_0} \right\rfloor}$ of the projections to the base group $B$. Since the latter information is contained in the partition $\mathcal{P}_n^{t_0}$, we obtain
	\begin{align*}
		H_{\mu_k}\Big((\mathcal{P}_n^{Nt_0}(A\wr B)\Big \vert w_n\vee \mathcal{P}_n^{t_0}(B)\vee  \beta_n \Big)&=	H_{\mu_k}\left(\Big(\varphi_{jNt_0}\Big)_{j=1,\ldots,\left\lfloor \frac{n}{Nt_0} \right\rfloor}\Big \vert w_n\vee \mathcal{P}_n^{t_0}(B)\vee  \beta_n\right)\\
		&<  \varepsilon n +(H(\mu)+1)\frac{nn_0}{N}.
	\end{align*}
\end{proof}

In the statement of Lemma \ref{lem: starting point for iterated wreath products} below, there are two coarse trajectories that appear. For a probability measure $\mu$ on the wreath product $A\wr B$ and for $N,t_0\ge 1$, we consider $\mathcal{P}_n^{t_0}(B)$ the $t_0$-coarse trajectory of the induced random walk on the base group $B$, and $\mathcal{P}_n^{Nt_0}(A\wr B)$ the $Nt_0$-coarse trajectory of the $\mu$-random walk on $A\wr B$.

\begin{lem}\label{lem: starting point for iterated wreath products} 
	Let  $\mu$ be a probability measure on $A\wr B$ with $H(\mu)<\infty$, and consider a sequence $\{\mu_k\}_{k\ge 1}$ of probability measure on $A\wr B$ with $H(\mu_k)<\infty$ for all $k\ge 1$. Suppose that $\lim_{k\to \infty}\mu_k(g)=\mu(g)$ for all $g\in A\wr B$ and that $\lim_{k\to \infty}H(\mu_k)=H(\mu)$. Denote by $\pi:A\wr B\to B$ the canonical epimorphism, and let us furthermore suppose that $\lim_{k\to \infty}\pesc{\pi_{*}\mu_k}=\pesc{\pi_{*}\mu}$, that the $\pi_{*}\mu$-random walk on $B$ is transient, and  that $\langle \supp{\pi_{*}\mu}\rangle_{+}$ is symmetric. Then for every $\varepsilon>0$ there exist $C\ge 0$ and $K,n_0,T\ge 1$ such that for all $k\ge K$, $t_0\ge T$, $N>n_0$ and $n> Nt_0$ we have
	\[
	H_{\mu_k}\Big(\mathcal{P}_n^{Nt_0}(A\wr B)\Big \vert w_n\Big)\le \varepsilon n +(H(\mu)+1)\frac{nn_0}{N}+H_{\mu_k}\Big(\mathcal{P}_n^{t_0}(B)\Big| w_n\Big)+C.
	\]
\end{lem}
\begin{proof}	
	By using Lemmas \ref{lem: bad increments have small entropy} and \ref{lem: main technical lemma}, for every $\varepsilon>0$ we can find finite symmetric subsets $L\subseteq A$ and $R\subseteq B$ with $e_A\in L$ and $e_B\in R$, a constant $C>0$, and $K,n_0,T\ge 1$ such that for every $t_0\ge T$, any $N>n_0$, every $k\ge K$ and all $n>Nt_0$ we have
	\begin{enumerate}
		\item $H_{\mu_k}(\beta_n(t_0,L,R)\mid \mathcal{P}_n^{t_0}(B))<\frac{\varepsilon}{2}n+C$, and
		\item $	H_{\mu_k}\Big(\mathcal{P}_n^{Nt_0}(A\wr B)\Big \vert w_n\vee \mathcal{P}_n^{t_0}(B)\vee  \beta_n(t_0,L,R) \Big)\le \frac{\varepsilon}{2} n +(H(\mu)+1)\frac{nn_0}{N}.$
	\end{enumerate}
	Note also that $\widetilde{W}_n^{Nt_0}=\mathcal{P}_n^{Nt_0}(A\wr B).$
	From this, we obtain
	\begin{align*}
		H_{\mu_k}(\mathcal{P}_n^{t_0}(B)\vee  \beta_n\mid w_n)	&\le 	H_{\mu_k}(\mathcal{P}_n^{t_0}(B)\mid w_n)+H_{\mu_k}(\beta_n\mid \mathcal{P}_n^{t_0}(B))\\
		&\le 	H_{\mu_k}(\mathcal{P}_n^{t_0}(B)\mid w_n) + \frac{\varepsilon}{2} n+ C.
	\end{align*}
	so that \begin{equation}\label{eq: 2bound of entropy all partitions}
		H_{\mu_k}( \mathcal{P}_n^{t_0}(B)\vee  \beta_n)\le 	H_{\mu_k}(\mathcal{P}_n^{t_0}(B)\mid w_n)+\frac{\varepsilon}{2} n +C.
	\end{equation}
	
	Therefore, using Equation \eqref{eq: 2bound of entropy all partitions}, we can conclude that 
	\begin{align*}
		H_{\mu_k}\Big(\mathcal{P}_n^{Nt_0}(A\wr B)\Big \vert w_n\Big)&\le 	H_{\mu_k}\Big(\widetilde{W}_n^{Nt_0}\Big \vert w_n\vee \mathcal{P}_n^{t_0}(B)\vee   \beta_n \Big)+H_{\mu_k}(\mathcal{P}_n^{t_0}(B)\vee  \beta_n\mid w_n)\\
		&\le \varepsilon n +(H(\mu)+1)\frac{nn_0}{N} + 	H_{\mu_k}(\mathcal{P}_n^{t_0}(B)\mid w_n) +C,
	\end{align*}
	which is the inequality from the statement of the proposition.
\end{proof}

We will furthermore need the following lemma.
\begin{lem}\label{lem: slices of independent increments have additive entropy} 
	Let $N,t_0\ge 1$. Then for any probability measure $\mu$ on $A\wr B$, it holds that 
	\[H_{\mu}\Big(\mathcal{P}_n^{Nt_0}(A\wr B)\Big)=\left\lfloor \frac{n}{Nt_0}\right\rfloor H_{\mu}(w_{Nt_0}) \text{ for each }n>Nt_0.\]
\end{lem}
\begin{proof}
	This follows from the fact that the values of $\mathcal{P}_n^{Nt_0}(A\wr B)$ are completely determined by $\left\lfloor\frac{n}{Nt_0}\right\rfloor$ independent random variables, each one with entropy $H_{\mu}\left( w_{Nt_0} \right)$.
\end{proof}

\begin{prop}\label{prop: final entropy estimates for continuity of entropy on wreath products} 
	Let  $\mu$ be a probability measure on $A\wr B$ with $H(\mu)<\infty$, and consider a sequence $\{\mu_k\}_{k\ge 1}$ of probability measure on $A\wr B$ with $H(\mu_k)<\infty$ for all $k\ge 1$. Suppose that $\lim_{k\to \infty}\mu_k(g)=\mu(g)$ for all $g\in A\wr B$ and that $\lim_{k\to \infty}H(\mu_k)=H(\mu)$. Denote by $\pi:A\wr B\to B$ the canonical epimorphism, and let us furthermore suppose that 
	\begin{itemize}
		\item $\lim_{k\to \infty}\pesc{\pi_{*}\mu_k}=\pesc{\pi_{*}\mu}$,
		\item  $\pi_{*}\mu$-random walk on $B$ is transient,
		\item  $h(\pi_{*}\mu)=0$, and that
		\item $\langle \supp{\pi_{*}\mu}\rangle_{+}$ is symmetric.
	\end{itemize}
	Then for each $\varepsilon>0$ there is a constant $C>0$, and $K,n_0,T\ge 1$ such that for every $t_0\ge T$, any $N>n_0$, every $k\ge K$ and all $n>Nt_0$ we have
	\[
	H_{\mu_k}(w_n)\ge \left(\frac{n}{Nt_0}-1\right)H_{\mu_k}(w_{Nt_0})- \varepsilon n-(H(\mu)+1)\frac{nn_0}{N}-C.
	\]
\end{prop}
\begin{proof}
	
	Let $\varepsilon>0$. We can use Lemmas \ref{lem: coarse trajectory has small entropy} and \ref{lem: starting point for iterated wreath products}  to find $C\ge 0$ and $K,n_0,T\ge 1$ such that for all $k\ge K$, $t_0\ge T$, $N>n_0$ and $n> Nt_0$ we have
	\[
	H_{\mu_k}\Big(\mathcal{P}_n^{Nt_0}(A\wr B)\Big \vert w_n\Big)\le \frac{\varepsilon}{2}n +(H(\mu)+1)\frac{nn_0}{N}+H_{\mu_k}\Big(\mathcal{P}_n^{t_0}(B)\Big| w_n\Big)+C,
	\]
	as well as $H_{\mu_k}\Big(\mathcal{P}_n^{t_0}(B)\Big| w_n\Big)\le H_{\mu_k}(\mathcal{P}_n^{t_0}(B))<\frac{\varepsilon}{2} n$.

	Recall that, from Lemma \ref{lem: slices of independent increments have additive entropy}, we know that \[H_{\mu_k}\Big(\mathcal{P}_n^{Nt_0}(A\wr B)\Big)=\left\lfloor \frac{n}{Nt_0}\right\rfloor H_{\mu_k}(w_{Nt_0}).\]
	
	Thus, using the above, we obtain 
	\begin{align*}
		\left\lfloor \frac{n}{Nt_0}\right\rfloor H_{\mu_k}(w_{Nt_0})&=H_{\mu_k}\Big(\mathcal{P}_n^{Nt_0}(A\wr B)\Big)\\
		&\le 	H_{\mu_k}\Big(\mathcal{P}_n^{Nt_0}(A\wr B)\Big \vert w_n\Big)+H_{\mu_k}(w_n)\\
		&\le   \frac{\varepsilon}{2}n +(H(\mu)+1)\frac{nn_0}{N}+H_{\mu_k}\Big(\mathcal{P}_n^{t_0}(B)\Big| w_n\Big)+C+H_{\mu_k}(w_n)\\
		&\le  \varepsilon n +(H(\mu)+1)\frac{nn_0}{N}+C+H_{\mu_k}(w_n).
	\end{align*}
	
	From this, one deduces directly the inequality from the statement of the proposition.
\end{proof}

We can now proceed with the proof of our main theorem.

\begin{proof}[Proof of Theorem \ref{thm: continuity asymptotic entropy wreath prods}]
	Let $\varepsilon>0$, and let us use Proposition \ref{prop: final entropy estimates for continuity of entropy on wreath products} to find $C,K,t_0\ge 1$ such that for every $t_0\ge T$, any $N>n_0$, every $k\ge K$ and all $n>Nt_0$ we have
	\[
	H_{\mu_k}(w_n)\ge \left(\frac{n}{Nt_0}-1\right)H_{\mu_k}(w_{Nt_0})- \varepsilon n-(H(\mu)+1)\frac{nn_0}{N}-C.
	\]

	Dividing this inequality by $n$ and then taking the limit $n\to \infty$, we obtain
	\[	h(\mu_k)\ge \frac{1}{Nt_0}H_{\mu_k}(w_{Nt_0})- \varepsilon -(H(\mu)+1)\frac{n_0}{N}.\]
	Now we consider this inequality as $k$ tends to infinity, and obtain
	\[\liminf_{k\to \infty}	h(\mu_k)\ge \frac{1}{Nt_0}H_{\mu}(w_{Nt_0})- \varepsilon -(H(\mu)+1)\frac{n_0}{N},\]
	where we used Lemma \ref{lem: convolutions entropy convergence} to justify that $H_{\mu_k}(w_{Nt_0})\xrightarrow[k\to \infty]{}H_{\mu}(w_{Nt_0})$.
	
	Next, we let $N$ tend to infinity and we get
	\[\liminf_{k\to \infty}	h(\mu_k)\ge h(\mu)- \varepsilon.\]
	Finally, since this inequality holds for all $\varepsilon>0$, we can consider the limit as $\varepsilon$ tends to $0$ and obtain $\liminf_{k\to \infty}	h(\mu_k)\ge h(\mu).$ Since the inequality $\limsup_{k\to \infty}	h(\mu_k)\le h(\mu)$ always holds (see Proposition \ref{prop: upper semicontinuous entropy}), we conclude that 
	\[
	\lim_{k\to \infty}h(\mu_k)=h(\mu),
	\]
	which is what we wanted to prove.
\end{proof}

Finally, we explain how one deduces Theorem \ref{thm: main corollary continuity of asymptotic entropy over Zd} from Theorem \ref{thm: continuity asymptotic entropy wreath prods}.
\begin{proof}[Proof of Theorem \ref{thm: main corollary continuity of asymptotic entropy over Zd}]
	Let $A$ be a countable group and let $B$ be a countable hyper-FC-central group. Let us suppose $B$ contains a finitely generated subgroup of at least cubic growth.  Let us denote by $\pi:A\wr B\to B$ the canonical epimorphism to the base group $B$.
	
	Consider any non-degenerate probability measure $\mu$ on $A\wr B$ with $H(\mu)<\infty$, and a sequence $\{\mu_k\}_{k\ge 1}$ of probability measures on $A\wr B$ with $H(\mu_k)<\infty$ for all $k\ge 1$, such that $\lim_{k\to \infty }\mu_k(g)=\mu(g)$ for each $g \in A\wr B$ and $\lim_{k\to \infty} H(\mu_k)=H(\mu)$.
	
	Since $\mu$ is non-degenerate and $B$ contains a finitely generated subgroup of at least cubic growth, it follows from Varopoulos \cite{Varopoulos1986} that the $\pi_{*}\mu$-random walk on $B$ is transient. Furthermore, we can use Theorem \ref{thm: continuity of range} to conclude that $\lim_{k\to \infty}\pesc{\pi_{*}\mu_k}=\pesc{\pi_{*}\mu}$. Since $B$ is hyper-FC-central and $H(\mu)<\infty$, it follows from \cite{LinZaidenberg1998,Jaworski2004} and the entropy criterion \cite{Derrienic1980,KaimanovcihVershik1983} that $h(\pi_{*}\mu)=0.$ Finally, as $\mu$ is non-degenerate, we have $\langle \supp{\pi_{*}\mu}\rangle_{+}=B$. Hence, all the hypotheses of Theorem \ref{thm: continuity asymptotic entropy wreath prods} are verified, and we obtain that $\lim_{k\to \infty} h(\mu_k)=h(\mu)$.

\end{proof}
\section{Continuity of harmonic measures}\label{section: convergence of harmonic measures}
\subsection{Continuity of asymptotic entropy as a consequence of the continuity of harmonic measures}\label{subsection: continuity of harmonic measures}
We start by proving a lemma that connects the weak convergence of stationary probability measures to the stationarity of their weak limits.

\begin{lem}\label{lem: weak limits are mu stationary}
	Let $G$ be a countable group and let $X$ be a separable completely metrizable space. Suppose that $G$ acts on $X$ continuously. Let $\mu$, $\{\mu_k\}_{k\ge 1}$ be probability measures on $G$, such that $\lim_{k\to \infty}\mu_k(g)=\mu(g)$ for each $g\in G$. Suppose that there are probability measures $\{\nu_k\}_{k\ge 1}$ on $X$ such that $\nu_k$ is $\mu_k$-stationary for each $k\ge 1$. Then, any weak sequential accumulation point of $\{\nu_k\}_{k\ge 1}$ is $\mu$-stationary.
\end{lem}
\begin{proof}
	
	Let $\nu$ be a weak accumulation point of the sequence $\{\nu_k\}_{k\ge 1}$. Without loss of generality, let us suppose that $\nu_k\xrightarrow[k\to \infty]{}\nu$. The stationarity hypothesis says that, for each $k\ge 1$, we have $ \nu_{k}=\mu_{k}*\nu_{k}=\sum_{g\in G}\mu_{k}(g)g_{*}\nu_{k}.$
	
	Hence, for each bounded continuous function $f:X\to \R$, we get
	\begin{equation}\label{eq: stationarity sequence}
		\int_X f(x)d\nu_{k}=\sum_{g\in G}\mu_{k}(g)\int_{X}f(gx)d\nu_{k}(x).	
	\end{equation}
	
	By definition of weak convergence, for the term on the left side of Equation \eqref{eq: stationarity sequence}, we have $\int_X f(x)d\nu_{k}(x)\xrightarrow[k\to \infty]{}\int_X f(x)d\nu(x).$ Now let us show that, for the term on the right side of Equation \eqref{eq: stationarity sequence}, we have $\sum_{g\in G}\mu_{k}(g)\int_{X}f(gx)d\nu_{k}(x)\xrightarrow[k\to \infty]{}\sum_{g\in G}\mu(g)\int_{X}f(gx)d\nu(x).$
	
	Indeed, let $\varepsilon>0$. Since $\mu$ is a probability measure on $G$, we can find a finite subset $F\subseteq G$ such that $\mu(F)>1-\varepsilon$. Then, thanks to the pointwise convergence of $\{\mu_k\}_{k\ge 1}$ to $\mu$, we can find $K_1\ge 1$ such that for each $k\ge K$ we have $\mu_k(F)>1-2\varepsilon$. This implies that $\mu_k(G\backslash F)<2\varepsilon$ for every $k\ge K_1$. Furthermore, from Lemma \ref{lem: pointwise convergence is equivalent to total variation convergence} we can find $K_2\ge K_1$ such that $\|\mu_k-\mu\|_1<\varepsilon$ holds for every $k\ge K_2$.
	
	Using the weak convergence, the fact that $G$ acts by continuous transformations on $X$ and the finiteness of $F$, we can find $K_3\ge K_2$ such that for each $k\ge K_3$ and every $g\in G$ we have 
	$$
	\left|\int_{X}f(gx)d\nu_{k}(x)-\int_{X}f(gx)d\nu(x) \right|<\varepsilon.
	$$
	From this, we obtain
	\footnotesize
	\begin{align*}
		\left| \sum_{g\in G}\mu_{k}(g)\int_{X}f(gx)d\nu_{k}(x)-\sum_{g\in G}\mu(g)\int_{X}f(gx)d\nu(x) \right|&\le 	\sum_{g\in F}\mu_{k}(g)\left| \int_{X}f(gx)d\nu_{k}(x)-\int_{X}f(gx)d\nu(x) \right|+\\
		&+ \sum_{g\in G\backslash F}\mu_{k}(g)\left| \int_{X}f(gx)d\nu_{k}(x)-\int_{X}f(gx)d\nu(x) \right|+\\
		&+\sum_{g\in G}\left|\mu_k(g)-\mu(g)\right|\left|\int_{X}f(gx)d\nu(x)\right|\\
		&\le |F|\varepsilon+ 2\|f\|_{\infty}\mu_k(G\backslash F)+\|\mu_k-\mu\|_1\|f\|_{\infty}\\
		&\le \left(|F|+5\|f\|_{\infty}\right)\varepsilon.
	\end{align*}
	\normalsize
	Since $\varepsilon$ was arbitrary, we conclude the desired convergence. Hence, by taking the limit as $k$ tends to infinity on Equation \eqref{eq: stationarity sequence}, we finally get
	$$
	\int_X f(x)d\nu(x)=\sum_{g\in G}\mu(g)\int_X f(gx)d\nu(x)=\sum_{g\in G}\mu(g)\int_X f(x)dg_{*}\nu(x),
	$$
	for every bounded continuous function $f:X\to \R$. In other words, $\nu=\mu*\nu$, and hence $\nu$ is $\mu$-stationary.		
\end{proof}

Let $X$ be a measurable space, and consider two equivalent probability measures $\nu_1,\nu_2$ on $X$. We recall that the \emph{Kullback-Leibler distance} between $\nu_1$ and $\nu_2$ is defined by $$I(\nu_1\mid \nu_2)\coloneqq -\int_{X}\log\left(\frac{d\nu_1}{d\nu_2}(x)\right)d\nu_2(x).$$ 

The following theorem of Kaimanovich and Vershik expresses the asymptotic entropy of a random walk on a group as the average of Kullback-Leibler distances on the Poisson boundary.

\begin{thm}[{\cite[Theorem 3.1]{KaimanovcihVershik1983}}]\label{thm: entropy as furstenberg formula} Let $\mu$ be a non-degenerate probability measure on a countable group $G$ with $H(\mu)<\infty$. Then,
	\[
	h(\mu)=-\sum_{g\in G}\mu(g)\int_{\partial_{\mu}G}\log\left(\frac{d g^{-1}_{*}\nu}{d\nu}(\xi)\right) d\nu(\xi)=\sum_{g\in G}\mu(g)I(g^{-1}_{*}\nu\mid \nu),
	\]
	where $(\partial_{\mu}G,\nu)$ is the Poisson boundary of $(G,\mu)$.
\end{thm}

The following result guarantees the joint weak lower-semicontinuity of the Kullback-Leibler distance, and its proof can be found in \cite[Theorem 1]{Posner1975}. Alternatively, this result can be obtained as a consequence of the Donsker-Varadhan variational formula \cite[Theorem 2.1]{DonskerVaradhan1983}; see e.g. \cite[Lemma 1.4.3 (b)]{DupuisEllis1997}.

\begin{prop}\label{prop: KL is lower semicontinuous}
	Let $X$ be a separable completely metrizable space (i.e\ a Polish space). Consider sequences $\{\nu_k\}_{k\ge 1}$ and $\{\eta_k\}_{k\ge 1}$ of probability measures on $X$. Suppose that there are probability measures $\nu,\eta$ on $X$ such that $\nu_k\xrightarrow[k\to \infty]{}\nu$ and $\eta_k\xrightarrow[k\to \infty]{}\eta$ weakly. Then
	\[
	\liminf_{k\to \infty} I(\nu_k|\eta_k)\ge I(\nu|\eta).
	\]
\end{prop}

With this, we are ready to present the proof of Theorem \ref{thm: convergence of harmonic measures}.

\begin{proof}[Proof of Theorem \ref{thm: convergence of harmonic measures}]
	Let us first note that the hypotheses $\lim_{k\to \infty}\mu_k(g)=\mu(g)$ and $\lim_{k\to \infty}H(\mu_k)=H(\mu)$ imply that $\limsup_{k\to \infty} h(\mu_k)\le h(\mu)$, thanks to \cite[Proposition 4]{AmirAngelVirag2013} (see also Proposition \ref{prop: upper semicontinuous entropy}). Hence, it suffices to prove that $\liminf_{k\to \infty}h(\mu_k)\ge h(\mu)$.

	Theorem \ref{thm: entropy as furstenberg formula}, together with the assumption that $X$ provides a model for the Poisson boundary, imply that we have the equalities $h(\mu_k)=\sum_{g\in G} \mu_k(g)I(g^{-1}_{*}\nu_k\mid \nu_k)$, for every $k\ge 1$, and $h(\mu)=\sum_{g\in G} \mu(g)I(g^{-1}_{*}\nu\mid \nu)$. 
	
	We know that $\{\nu_k\}_{k\ge 1}$ converges weakly to $\nu$ and that $G$ acts by continuous transformations on $X$. Together, these imply that for each $g\in G$ we also have that $\{g^{-1}_*\nu_k\}_{k\ge 1}$ converges weakly to $g^{-1}_{*}\nu$. It is known that the Kullback-Leiber distance is lower-semicontinuous with respect to weak convergence; see Proposition \ref{prop: KL is lower semicontinuous}. Thus, for each $g\in G$, we obtain $\liminf_{k\to \infty} I(g^{-1}_{*}\nu_k\mid\nu_k)\ge I(g^{-1}_{*}\nu\mid\nu).$ Next, using the fact that $I(g_{*}^{-1}\nu_k\mid \nu_k)\ge 0$ for each $g\in G$ and $k\ge 1$ and Fatou's lemma, we obtain
	\begin{align*}
		\liminf_{k\to \infty}h(\mu_k)&=	\liminf_{k\to \infty} \sum_{g\in G}\mu_k(g)I(g^{-1}_{*}\nu_k\mid \nu_k)\\
		&\ge \sum_{g\in G}\liminf_{k\to \infty}\mu_k(g)I(g^{-1}_{*}\nu_k\mid \nu_k)\\
		&=\sum_{g\in G}\mu(g)\liminf_{k\to \infty}I(g^{-1}_{*}\nu_k\mid \nu_k)\\
		&\ge\sum_{g\in G}\mu(g)I(g^{-1}_{*}\nu\mid \nu) =h(\mu).
	\end{align*}
	We conclude that $\liminf_{k\to \infty}h(\mu_k)\ge h(\mu)$, and hence that $\lim_{k\to \infty}h(\mu_k)=h(\mu)$.
	
	Finally, let us prove that if $X$ admits a unique $\mu$-stationary probability measure $\nu$, then $\nu_k\xrightarrow[k\to \infty]{}\nu$ weakly. This is a well-known argument, that has been remarked already in \cite[Lemma 3]{ErschlerKaimanovich2013} in the case where $G$ is a non-elementary hyperbolic group. By the compactness of $X$, any subsequence $\{\nu_{k_{\ell}}\}_{\ell \ge 1}$ admits a convergent subsequence $\{\nu_{k_{\ell_m}}\}_{m \ge 1}$ to some probability measure $\eta$ on $X$.  	
	Thanks to Lemma \ref{lem: weak limits are mu stationary}, the probability measure $\eta$ is $\mu$-stationary, and hence coincides with the unique $\mu$-stationary probability measure $\nu$ on $X$. We have thus proved that every subsequence of $\{\nu_k\}_{k\ge 1}$ has a further subsequence that converges to the unique $\mu$-stationary probability measure $\nu$ on $G$. This finally implies that the entire sequence converges to $\nu$.	
\end{proof}

\begin{rem}
	In the statement of Theorem \ref{thm: convergence of harmonic measures}, a sufficient condition for the uniqueness of stationary measures is that $X$ is a compact metric space on which each $G$-orbit is dense; see e.g.\ \cite[Proposition 11.4.13]{Lalley2023}.
\end{rem}

\subsection{Applications of Theorem \ref{thm: convergence of harmonic measures}}\label{subsection: applications}
We now present families of groups to which one can apply Theorem \ref{thm: convergence of harmonic measures} to establish the continuity of asymptotic entropy, within a suitable class of step distributions.
\subsubsection{Gromov hyperbolic groups and acylindrically hyperbolic groups}
Recall that a geodesic metric space $X$ is called Gromov hyperbolic if there is $\delta\ge 0$ such that every geodesic triangle in $X$ is $\delta$-thin, meaning that any side is contained in a $\delta$-neighborhood of the union of the other two sides. To each hyperbolic space there is an associated \emph{Gromov boundary} $\partial X$, which consists of equivalence classes of geodesic rays, where two rays are identified if they are at finite Hausdorff distance from each other. On what follows, we will suppose that the space $X$ is separable, which guarantees that the Gromov boundary $\partial X$ carries a natural topology that is separable and completely metrizable. If one furthermore supposes that $X$ is locally compact, then both $\partial X$ and $X\cup \partial X$ are also compact. As a general reference for Gromov hyperbolic spaces, we refer to \cite[Section III.H.3]{BridsonHaefliger} and \cite{Vaislala2005}. A finitely generated group is said to be Gromov hyperbolic if one (equivalently any) of its Cayley graphs with respect to a finite generating set is $\delta$-hyperbolic, for some $\delta\ge 0$. For more background on hyperbolic groups we refer to Gromov's original work \cite{Gromov1987}, and to \cite{GhysdlH1990}.

Let $G$ be a Gromov hyperbolic group, and denote by $\partial G$ its Gromov boundary. Consider a non-elementary probability measure $\mu$ on $G$. It is shown in \cite[Theorem 7.6]{Kaimanovich2000} that there is a unique $\mu$-stationary probability measure $\lambda$ on $\partial G$ (see also \cite[Corollary 1]{Woess1993} for the case where $\langle \supp{\mu}\rangle=G$). More generally, suppose that $G$ is a countable group acting by isometries on a separable Gromov hyperbolic space $X$. Consider a non-elementary probability measure $\mu$ on $G$. It is proved in \cite[Theorem 1.1]{MaherTiozzo2018} that for any $x_0\in X$, almost every sample path $\{w_nx_0\}_{n\ge 0}$ converges to an element $\xi\in \partial X$, and that the distribution of $\xi$ is the unique $\mu$-stationary probability measure on $\partial X$.

We first prove the following lemma, which guarantees the existence of convergent subsequences of a sequence of probability measures on $\partial X$, even though the Gromov boundary $\partial X$ may not be compact. We follow a similar argument to the first part of the proof of \cite[Theorem 1.1]{MaherTiozzo2018}, where Maher and Tiozzo show the existence of $\mu$-stationary probability measures on $\partial X$.
\begin{lem}\label{lem: existence of weak accumulation point}
	Let $G$ be a countable group and let $X$ be a separable Gromov hyperbolic space. Suppose that $G$ acts on $X$ by isometries. Let $\{\mu_k\}_{k\ge 1}$ be non-elementary probability measures on $G$ and denote by $\{\nu_k\}_{k\ge 1}$ the corresponding hitting measures on the Gromov boundary $\partial X$. Then $\{\nu_k\}_{k\ge 1}$ has a convergent subsequence.
\end{lem}
\begin{proof}
	For this proof, we will use the \emph{horofunction compactification} $\overline{X}^h$ of $X$. We refer to \cite[Section 3]{MaherTiozzo2018} for the definition of $\overline{X}^h$; see also \cite[Section 3.3]{BallmannGromovSchroeder1985}. We now recall the properties of the horofunction compactification that will be relevant in the rest of the proof.
	
	Given a separable Gromov hyperbolic space $X$, its horofunction compactification $\overline{X}^h$ is a compact metrizable space on which $X$ embeds. The horofunction compactification is partitioned as a disjoint union $\overline{X}^h=\overline{X}^h_{F}\cup \overline{X}^h_{\infty}$, where $\overline{X}^h_{F}$ (resp.\ $\overline{X}^h_{\infty}$) is called the set of \emph{finite} (resp.\ \emph{infinite}) horofunctions. One can construct a map $\phi:\overline{X}^h\to X\cup \partial X$, which is called the \emph{local minimum map}, and which satisfies that the restriction $\phi:\overline{X}^h_{\infty}\to \partial X$ is well-defined, surjective and continuous \cite[Proposition 3.14 \& Corollary 3.15]{MaherTiozzo2018}.
	
	It follows from \cite[Proposition 4.4]{MaherTiozzo2018} that for each $k\ge 1$, the pull-back probability measure $\phi^{-1}_{*}\nu_{k}$ is supported on $\overline{X}^h_{\infty}.$	 By compactness of $\overline{X}^h$, the sequence $\{\phi^{-1}_{*}\nu_{k}\}_{k\ge 1}$ has a weak convergent subsequence $\phi^{-1}_{*}\nu_{k_j}\xrightarrow[k\to \infty]{}\eta$, for $\eta$ a probability measure on $\overline{X}^h$. The surjectivity of $\phi$ implies that $\phi_{*}\left(\phi^{-1}_{*}\nu_{k_j}\right)=\nu_{k_j}$ for all $j\ge 1$, and its continuity that $\phi_{*}\left(\phi^{-1}_{*}\nu_{k_j}\right)=\nu_{k_j}\xrightarrow[j\to \infty]{}\phi_{*}\eta$.  We conclude that $\{\nu_k\}_{k\ge 1}$ has a weak convergent subsequence to a probability measure on $\partial X$.
\end{proof}

In the following proposition, we prove the convergence of harmonic measures on the Gromov boundary $\partial X$ of a separable Gromov hyperbolic space $X$ without assuming that it is proper, so that $\partial X$ is not necessarily compact. The proof follows a similar argument to the second part of the proof of Theorem \ref{thm: convergence of harmonic measures}. The main difference is that the existence of convergent subsequences, which is guaranteed by a compactness argument in Theorem \ref{thm: convergence of harmonic measures}, is now replaced by Lemma \ref{lem: existence of weak accumulation point}.
\begin{prop}\label{prop: ac hyperbolic convergence harmonic measures}
	Let $G$ be a countable group and let $X$ be a separable Gromov hyperbolic space. Suppose that $G$ acts on $X$ by isometries. Let $\mu$, $\mu_k$, $k\ge 1$, be non-elementary probability measures on $G$, and denote by $\nu$, $\nu_k$, $k\ge 1$, the corresponding harmonic measures on the Gromov boundary $\partial X$. Then $\nu_k\xrightarrow[k\to \infty]{}\nu$ weakly.
\end{prop}
\begin{proof}
	Thanks to  \cite[Theorem 1.1]{MaherTiozzo2018}, the probability measure $\nu$ is the unique $\mu$-stationary probability measure on $\partial X$. Consider any subsequence $\{\nu_{k_{\ell}} \}_{\ell\ge 1}$ of $\{\nu_k\}_{k\ge 1}$. Then, using Lemma \ref{lem: existence of weak accumulation point}, there is a further subsequence $\{\nu_{k_{\ell_m}} \}_{m\ge 1}$ that converges weakly to some probability measure $\eta$ on $\partial X$. Thanks to Lemma \ref{lem: weak limits are mu stationary}, the probability measure $\eta$ is $\mu$-stationary. As a consequence of the uniqueness of the $\mu$-stationary measure, we must have $\eta=\nu$. We conclude that every subsequence of $\{\nu_k\}_{k\ge 1}$ has a further subsequence that converges to $\nu$ on $G$. This implies that the entire sequence $\{\nu_k\}_{k\ge 1}$ converges to $\nu$.
	
\end{proof}

Let us recall that the action of $G$ on $X$ is called \emph{acylindrical} if for every $K\ge 0$ there are $N,R\ge 1$ such that for any $x,y\in X$ with $d_X(x,y)\ge R$, there are at most $N$ elements $g\in G$ such that $d_X(x,g.x)\le K$ and $d_X(y,gy)\le K.$ It is proved in \cite[Theorem 1.2]{ChawlaForghaniFrischTiozzo2022} that, if $\mu$ has finite entropy, then $(\partial X, \lambda)$ is the Poisson boundary of $(G,\mu)$, where $\lambda$ is the hitting measure on $\partial X$. This generalizes \cite[Theorem 1.5]{MaherTiozzo2018} which had the extra assumption of a finite logarithmic moment of $\mu$, as well as previous descriptions of the Poisson boundary for particular classes of groups that fall within the category of acylindrically hyperbolic groups. In particular, the above covers free groups \cite[Theorem 3]{DynkinMaljutov1961} \cite[Théorème 2]{Derrienic1975}, hyperbolic groups \cite[Theorem 8]{Ancona1987} \cite[Theorem 8]{Kaimanovich1994} \cite[Theorems 7.4 \& 7.7]{Kaimanovich2000}, groups with infinitely many ends \cite[Theorem 7.1]{Woess1989} \cite[Theorem 8.4]{Kaimanovich2000}, mapping class groups \cite[Theorem 2.3.1]{KaimanovichMasur1996} and $\mathrm{Out}(F_n)$ \cite[Theorem 0.2 \& 3.3]{Horbez2016} \cite[Theorem 1.3]{NamaziPettetReynolds2014}. The combination of Theorem \ref{thm: convergence of harmonic measures}, Proposition \ref{prop: ac hyperbolic convergence harmonic measures}, and the description of the Poisson boundary for acylindrically hyperbolic groups gives us the following result.
\begin{cor}\label{cor: cont entropy hyperbolic}
	Let $G$ be an acylindrically hyperbolic group. Let $\mu$, $\mu_k$, $k\ge 1$, be a sequence of non-elementary probability measures on $G$ with finite entropy, such that $\lim_{k\to \infty}\mu_k(g)=\mu(g)$ and $\lim_{k\to \infty}H(\mu_k)=H(\mu)$. Then $\lim_{k\to \infty}h(\mu_k)=h(\mu)$.
\end{cor}
This result was already known: it has been already proved in \cite[Theorem F]{Choi2024}, and also independently by Joshua Frisch and Anna Erschler. In the case of hyperbolic groups, with additional moment conditions on the measures, the above result is proved in \cite[Theorem 2]{ErschlerKaimanovich2013} and \cite[Theorem 2.9]{GouezelSebastienMatheusMacourant2018}.

\subsubsection{Discrete subgroups of linear groups}
Let $G$ be a countable discrete subgroup of $\mathrm{SL}_d(\R)$, $d\ge 2$, and consider the associated flag variety $\mathcal{F}$ (see \cite[Section I.1]{Ledrappier1985} for the definition). It is proved in \cite[Theorem A]{Ledrappier1985} \cite[Theorem 10.3]{Kaimanovich2000} that, if $\mu$ is a non-degenerate probability measure on $G$ with a finite first moment, then there is a unique $\mu$-stationary probability measure $\lambda$ on $\mathcal{F}$ such that $(\mathcal{F},\lambda)$ is the Poisson boundary of $(G,\mu)$. A more general version of this statement is proved in \cite[Theorem 5]{Kaimanovich1985conditional} \cite[Theorem 10.7]{Kaimanovich2000} for discrete subgroups semi-simple Lie groups, and which in particular covers the case where $G$ is a Zariski dense discrete subgroup of $\mathrm{SL}_d(\R)$ and $\mu$ is a non-degenerate probability measure on $G$ with finite entropy and a finite logarithmic moment. The above results together with Theorem \ref{thm: convergence of harmonic measures} imply the following.
\begin{cor}
	Let $G$ be a countable discrete subgroup of $\mathrm{SL}_d(\R)$, $d\ge 2$. Consider $\mu$, $\mu_k$, $k\ge 1$, non-degenerate probability measures on $G$, and such that $\lim_{k\to \infty}\mu_k(g)=\mu(g)$ and $\lim_{k\to \infty}H(\mu_k)=H(\mu)$. Suppose that we are in one of the following situations.
	\begin{enumerate}
		\item The probability measures $\mu$, $\mu_k$, $k\ge 1$, have a finite first moment.
		\item The group $G$ is Zariski dense and the probability measures $\mu$, $\mu_k$, $k\ge 1$, have finite entropy and a finite logarithmic moment.
	\end{enumerate} 
	Then $\lim_{k\to \infty}h(\mu_k)=h(\mu)$.
\end{cor}
\subsubsection{Groups acting on $\mathrm{CAT}(0)$-spaces}
For the basic definitions related to $\mathrm{CAT}(0)$ spaces, we refer to \cite[Part II]{BridsonHaefliger}. We recall that for a proper $\mathrm{CAT}(0)$ space $X$, there is a natural compactification $X\cup \partial_{\infty}X$, where the space $\partial_{\infty}X$ is called the \emph{visual boundary} of $X$ and it consists of equivalence classes of geodesic rays. Consider a $\mathrm{CAT}(0)$ space $X$, and let $G$ be a discrete group that acts properly and cocompactly by isometries on $X$. It follows from \cite[Corollary 6.2]{KarlssonMargulis1999} that, for any non-degenerate probability measure $\mu$ on $G$ with a finite first moment, the visual boundary $(\partial_{\infty}X,\nu)$ is the Poisson boundary of $(G,\mu)$. In \cite[Theorem 1.1]{Bars2022} it is proved that, if $G$ contains a rank one element, then $\nu$ is the unique $\mu$-stationary probability measure on $\partial_{\infty}X$.

Now, suppose that $X$ is a finite dimensional $\mathrm{CAT}(0)$ cube complex. Consider a discrete countable group $G$, and suppose that there is a non-elementary proper action of $G$ on $X$ by automorphisms. Then, there is a compact metric space $\partial_R X$, called the \emph{Roller boundary} of $X$, with the following property: for every non-degenerate probability measure on $G$ with finite entropy and a finite logarithmic moment, there exists a $\mu$-stationary probability measure $\lambda$ on $\partial_R X$ such that $(\partial_R X,\lambda)$ is the Poisson boundary of $(G,\mu)$. This result was first proved in \cite[Theorem 8.4]{NevoSageev2013} under the additional assumptions that $X$ is locally compact and the action of $G$ on $X$ is cocompact, and then in the general case in \cite{Fernos2018}. Furthermore, it is proved in \cite[Corollary 7.3]{FernosLecureuxMatheus2018} that in the above situation, the measure $\lambda$ is the unique $\mu$-stationary probability measure on $\partial_R X$. 

By using the above results together with Theorem \ref{thm: convergence of harmonic measures}, we deduce the following.
\begin{cor}
	Let $G$ be a countable discrete group of isometries of a proper $\mathrm{CAT}(0)$ metric space $X$. Let $\mu$, $\mu_k$, $k\ge 1$, be a sequence of non-degenerate probability measures on $G$ such that $\lim_{k\to \infty}\mu_k(g)=\mu(g)$ and $\lim_{k\to \infty}H(\mu_k)=H(\mu)$. Suppose that we are in one of the following situations.
	\begin{enumerate}
		\item The action of $G$ on $X$ is proper and cocompact, $G$ has a rank one element, and the measures $\mu$, $\mu_k$, $k\ge 1$, have a finite first moment.
		\item  $X$ is a finite dimensional $\mathrm{CAT}(0)$ cube complex, and the measures $\mu$, $\mu_k$, $k\ge 1$, have finite entropy and a finite logarithmic moment.
	\end{enumerate}
	Then $\lim_{k\to \infty}h(\mu_k)=h(\mu)$.
\end{cor}

\subsubsection{Strongly locally discrete subgroups of diffeomorphisms of the circle}
Let $G$ be a subgroup of homeomorphisms of the circle $S^1$ that does not admit an invariant probability measure on $S^1$. In this case, we say that $G$ is \emph{non-elementary}. In \cite[Proposition 5.5]{DeroinKleptsynNavas2007} it is shown that there exists a finite $G$-equivariant quotient $\mathbf{S}$ of $S^1$ that admits a unique $\mu$-stationary probability measure $\lambda$, and such that $(\mathbf{S},\lambda)$ is a $\mu$-boundary (see also the explanation in Section 1.1 of \cite{Deroin2013}). Moreover, it is proved in \cite[Theorem 1.1]{Deroin2013} that, if the elements of $G$ are sufficiently regular, the group $G$ satisfies a strong local discreteness assumption and $\mu$ has suitable moment conditions, then $(\mathbf{S},\lambda)$ is the Poisson boundary of $(G,\mu)$. In particular, this holds if the three following conditions are satisfied simultaneously:
\begin{enumerate}
	\item The group $G$ consists of diffeomorphisms of $S^1$ of class $C^2$, and $G$ does not preserve a probability measure on $S^1$. The latter implies that there is a unique $G$-invariant compact subset of $S^1$ on which every orbit is dense, called the \emph{limit set}.
	\item The group $G$ is \emph{strongly locally discrete in the $C^2$-topology}. This means that there exists a covering of its limit set by intervals $(I_j)_{j\in \mathcal{J}}$ such that, if a sequence of elements $(g_n)_{n\ge 1}$ on $G$ converges to the identity element in the $C^2$-topology on one of the $I_j$, $j\in \mathcal{J}$, then the element $g_n$ is the identity for sufficiently large $n$.
	\item The probability measure $\mu$ on $G$ is finitely supported and non-degenerate.
\end{enumerate}
The following corollary is a direct consequence of the above together with Theorem \ref{thm: convergence of harmonic measures}.
\begin{cor}
	Let $G$ be a countable group of $C^2$-diffeomorphisms of $S^1$, such that the action of $G$ on $S^1$ does not preserve a probability measure on $S^1$. Suppose furthermore that $G$ is strongly locally discrete in the $C^2$-topology. Let $\mu$, $\mu_k$, $k\ge 1$, be a sequence of non-degenerate finitely supported probability measures on $G$ such that $\lim_{k\to \infty}\mu_k(g)=\mu(g)$ and $\lim_{k\to \infty}H(\mu_k)=H(\mu)$. Then $\lim_{k\to \infty}h(\mu_k)=h(\mu)$.
\end{cor}

\subsubsection{Free-by-cyclic groups}
It is proved in \cite[Theorem 1]{GauteroMatheus2012} that if $G$ is a cyclic extension of a free group $F$ and $\mu$ is a non-degenerate probability measure on $G$ with a finite first moment, then there exists a unique $\mu$-stationary probability measure $\lambda$ on the Gromov boundary $\partial F$, and $(\partial F,\lambda)$ is the Poisson boundary of $(G,\mu)$. From this together with Theorem \ref{thm: convergence of harmonic measures} we obtain the following corollary.
\begin{cor}
	Let $G$ be a cyclic extension of a non-abelian free group, and consider non-degenerate probability measures $\mu$, $\mu_k$, $k\ge 1$, on $G$ with a finite first moment and such that $\lim_{k\to \infty}\mu_k(g)=\mu(g)$ and $\lim_{k\to \infty}H(\mu_k)=H(\mu)$. Then $\lim_{k\to \infty}h(\mu_k)=h(\mu)$.
\end{cor}

\subsubsection{Wreath products}
One can also apply Theorem \ref{thm: convergence of harmonic measures} to recover a particular case of Theorem \ref{thm: main corollary continuity of asymptotic entropy over Zd} regarding the continuity of asymptotic entropy on wreath products. We first prove the next convergence of harmonic measures on lamplighter boundaries.
\begin{prop}\label{prop: wr app}
	Let $A$, $B$ be countable groups and let $\mu,\mu_k$, $k\ge 1$, be non-degenerate probability measures on $A\wr B$ that define a transient random walk on $B$. Suppose that
	\begin{enumerate}
		\item\label{item: wr1} $\lim_{k\to \infty} \mu_k(g)=\mu(g)$ for each $g\in G$,
		\item\label{item: wr2} there is a finite subset $F\subseteq B$ such that $\supp{\mu_k}\subseteq \{(f,b)\in A\wr B\mid \supp{f}\subseteq F\}$ for each $k\ge 1$, and
		\item\label{item: wr3} $\lim_{k\to \infty}\pesc{\pi_{*}\mu_k}=\pesc{\pi_{*}\mu}$, where $\pi:A\wr B \to B$ is the canonical projection to the base group $B$.
	\end{enumerate}
	 Denote by $\nu$, $\nu_k$, $k\ge 1$, their corresponding hitting measures on $\prod_B A$. Then $\nu_k\xrightarrow[k\to \infty]{}\nu$ weakly.
\end{prop}
Item~\eqref{item: wr2} above can be interpreted as requiring the sequence of probability measures to have a ``uniformly bounded lamp range'', meaning that there exists a fixed finite neighborhood in which the lamplighter person may modify lamps around their current position at each step. Note that this condition together with the transience assumption for the random walk in the base group guarantees the almost sure stabilization of lamp configurations, so that the hitting measures $\nu$, $\nu_k$, $k\ge 1$, are well defined. Note also that in this case the space $\prod_B A$ will in general not be uniquely $\mu$-stationary. As an example, one can think of the case where $A$ is finite and $B$ is amenable. Then $\prod_B A$ is a compact space that must possess an $A\wr B$-invariant probability measure, due to the amenability of $A\wr B$, which will be different from the hitting measure of a random walk.
\begin{proof}[Proof of Proposition \ref{prop: wr app}]
	The space $\prod_B A$ has a countable basis for its topology consisting of unions of cylinders of the form
	\[\mathcal{C}_{\mathbf{p}}\coloneqq \left\{f:B\to A\mid f(g)=\mathbf{p}(g) \text{ for each }g\in \mathrm{dom}(\mathbf{p})\right\},\]
	where $\mathbf{p}:\mathrm{dom}(\mathbf{p})\to A$ is a function from a finite subset $\mathrm{dom}(\mathbf{p})\subseteq B$ to $A$.
	
	To show that $\nu_k\xrightarrow[k\to \infty]{}\nu$ weakly, it suffices to prove that for any such cylinder as above we have 
	$\nu_k(\mathcal{C}_{\mathbf{p}})\xrightarrow[k\to \infty]{}\nu(\mathcal{C}_{\mathbf{p}}).$ Indeed, this follows from the fact that any open set in $\prod_B A$ can be written as a countable union of disjoint cylinders, together with the Portmanteau theorem \cite[Theorem 2.1]{Billingsley}.
	
	Recall that we denote by $(\varphi_n, X_n)$, $n\ge 1$, a trajectory of a random walk on $A\wr B$. Let us introduce the event
	\[
	E_{\infty}=\{\varphi_{\infty}(g)=\mathbf{p}(g) \text{ for each }g\in \mathrm{dom}(\mathbf{p})\},
	\]
	and for each $N\ge 1$ the events
	\[
	E_{N}=\{\varphi_{N}(g)=\mathbf{p}(g) \text{ for each }g\in \mathrm{dom}(\mathbf{p})\}, 
	\]
	\[
	T_N=\{X_n\notin \mathrm{dom}(\mathbf{p})\cdot F^{-1} \text{ for all }n\ge N\},
	\]
	where $F\subseteq B$ is a finite subset such that $\supp{\mu_k}\subseteq \{(f,b)\in A\wr B\mid \supp{f}\subseteq F\}$ for each $k\ge 1$. This is a hypothesis in Item \eqref{item: wr2} above.
	
	Then for each $N\ge 1$ we have
	
	\begin{equation*}\label{eq: decomp}
	\nu_k(\mathcal{C}_{\mathbf{p}})=\P_{\mu_k}(E_\infty)=\P_{\mu_k}(E_N\cap T_N)+\P_{\mu_k}(E_{\infty}\cap T_N^c).
	\end{equation*}
	
	Note also that we have
	\begin{equation*}
	\P_{\mu_k}(E_N\cap T_N)=\sum_{b\in B}\P_{\mu_k}\left(E_N\cap\{X_N=b\}\cap\{X_N^{-1}X_n\notin b^{-1}\mathrm{dom}(\mathbf{p}) F^{-1} \text{ for all }n\ge N\}\right),
	\end{equation*}
	and conditioning together with the independence of the increments, we get
	\begin{equation*}\label{eq: conditioning}
	\P_{\mu_k}(E_N\cap T_N)=\sum_{b\in B}\P_{\mu_k}\left(E_N\cap\{X_N=b\}\right)\P_{\mu_{k}}\left(\{X_n\notin b^{-1}\mathrm{dom}(\mathbf{p}) F^{-1} \text{ for all }n\ge 0\}\right).
	\end{equation*}
	
	Let $\varepsilon>0$. First, let us choose $N\ge 1$ sufficiently large such that $\P_{\mu_k}(T_N^c)<\varepsilon$ for each $k\ge 1$ and $\P_{\mu}(T_N^c)<\varepsilon$. The existence of such an $N$ is guaranteed by the continuity of the escape probability of the random walk in the base group, assured by Item \eqref{item: wr3} in the statement above.  Next, we can find a finite subset $Z\subseteq B$ and $K_1\ge 1$ have $\P_{\mu_{k}}(X_N\notin Z)<\varepsilon$ for all $k\ge K_1$ and $\P_{\mu}(X_N\notin Z)<\varepsilon$. This is assured by the convergence of Item \eqref{item: wr1} above (see also Lemma \ref{lem: pointwise convergence is equivalent to total variation convergence}). Using again the convergence of escape probabilities in the base group, and using the convergence from Item \eqref{item: wr1} and Lemma \ref{lem: convolutions continuity} which guarantees the convergence of the convolutions of the measures, we can find $K_2\ge K_1$ such that for each $k\ge K_2$ we have
	
\begin{equation*}
	\left|
	\begin{aligned}
	&\sum_{b\in Z}\P_{\mu_k}\left(E_N\cap\{X_N=b\}\right)\P_{\mu_{k}}\left(\{X_n\notin b^{-1}\mathrm{dom}(\mathbf{p}) F^{-1} \text{ for all }n\ge 0\}\right)-\\
	&-\sum_{b\in Z}\P_{\mu}\left(E_N\cap\{X_N=b\}\right)\P_{\mu}\left(\{X_n\notin b^{-1}\mathrm{dom}(\mathbf{p}) F^{-1} \text{ for all }n\ge 0\}\right)
	\end{aligned}\right| <\varepsilon
\end{equation*}

	By combining the above equations and using the triangular inequality several times, we conclude that
	\begin{equation*}
	\left|\nu_k(\mathcal{C}_{\mathbf{p}})-\nu(\mathcal{C}_{\mathbf{p}})\right|\le 5\varepsilon
	\end{equation*}
	for each $k\ge K_2$. Since $\varepsilon>0$ was arbitrary, this shows the desired weak convergence.
\end{proof}

Let $A$ be a countable group and let $B$ be a hyper-FC-central countable group. It was proved in \cite[Corollary 1.4]{FrischSilva2024} that for any probability measure $\mu$ on $A\wr B$ with $H(\mu)<\infty$ and such that the lamp configurations stabilize almost surely, the space $\prod_B A$ endowed with the hitting measure is the Poisson boundary of $(A\wr B,\mu)$. This result generalized previous work that identified the Poisson boundary of $A\wr B$ under stronger assumptions both on the measure $\mu$ as well as on the groups $A$ and $B$ by A. Erschler \cite[Theorem 1]{Erschler2011wreath} and R. Lyons and Y. Peres \cite[Theorem 1.1]{LyonsPeres2021}. Combining this with Proposition \ref{prop: wr app} and with Theorems \ref{thm: continuity of range} and \ref{thm: convergence of harmonic measures} we obtain the following particular case of Theorem \ref{thm: continuity asymptotic entropy wreath prods}.

\begin{cor}\label{cor: wr}
		Let $A$ be a countable group and let $B$ be a hyper-FC-central countable group with at least cubic growth. Let $\mu,\mu_k$, $k\ge 1$, be non-degenerate probability measures on $A\wr B$ with finite entropy. Suppose that
	\begin{enumerate}
		\item $\lim_{k\to \infty} \mu_k(g)=\mu(g)$ for each $g\in G$,
		\item $\lim_{k\to \infty} H(\mu_k)=H(\mu)$,
		\item there is a finite subset $F\subseteq B$ such that $\supp{\mu_k}\subseteq \{(f,b)\in A\wr B\mid \supp{f}\subseteq F\}$ for each $k\ge 1$.
	\end{enumerate}
	Then $\lim_{k\to \infty}h(\mu_k)=h(\mu)$.
\end{cor}

We remark that this approach cannot recover Theorem~\ref{thm: main corollary continuity of asymptotic entropy over Zd} in full generality for arbitrary sequences of probability measures with finite entropy. At this level of generality, there is no identification of the Poisson boundary of $A \wr B$, let alone one that simultaneously identifies the Poisson boundaries associated with different probability measures within a common Polish model space. In particular, the hitting measure on the space $\prod_B A$ is not well defined, since lamp configurations need not stabilize in this level of generality; see the last paragraph of Subsection~\ref{subsection: wreath products}.

\bibliographystyle{alpha}
\bibliography{biblio}
\end{document}